\documentclass[8pt,a4paper]{article}
\usepackage[T1]{fontenc}
\usepackage[utf8]{inputenc}
\usepackage{authblk}
\usepackage{amsmath,amsthm,amssymb}
\usepackage{calligra,indentfirst}
\calligra \newtheorem{theo}{Theorem}[section]
\newtheorem{lem}{Lemma}[section]
\newtheorem{cor}{Corollary}[section]
\newtheorem{prop}{Proposition}[section]
\newtheorem{rem}{Remark}[section]

\allowdisplaybreaks
\title{Enumeration of complementary-dual cyclic $\mathbb{F}_{q}$-linear $\mathbb{F}_{q^t}$-codes}
\author[1]{Anuradha Sharma\thanks{anuradha@iiitd.ac.in}\thanks{Research support by DST India, under the grant SERB/F/3551/2012-13, is gratefully acknowledged.}}
\author[2]{Taranjot Kaur \thanks{Research support by NBHM, India, is gratefully acknowledged.}}
\affil[1]{Center for Applied Mathematics, IIIT Delhi,  India}
\affil[2]{Department of Mathematics, IIT Delhi, India}

\begin{document}
 \date{} \maketitle
\begin{abstract}
 Let $\mathbb{F}_q$ denote the finite field of order $q,$ $n$ be a positive integer coprime to $q$ and $t \geq 2$ be an integer. In this paper, we enumerate all the complementary-dual cyclic $\mathbb{F}_q$-linear $\mathbb{F}_{q^t}$-codes of length $n$ by placing $\ast$, ordinary and Hermitian trace bilinear forms on $\mathbb{F}_{q^t}^n.$
 \end{abstract}
\section{Introduction}
Cyclic codes form an  important class of linear codes having a rich algebraic structure. Their algebraic properties enable one to effectively detect or correct errors using linear shift registers. Self-dual, self-orthogonal  and  complementary-dual  codes constitute three important classes of cyclic codes, which have been studied and  enumerated  by Conway et al.  \cite{conway}, Huffman \cite{huff4}-\cite{huff5}, Jia et al. \cite{jia} and Pless and Sloane \cite{pless}.  Linear codes are further generalized to additive codes, which have recently attracted a lot of attention due to their connection with quantum error-correcting codes.

Calderbank et al. \cite{cal} introduced and studied additive codes of length $n$ over the finite field $\mathbb{F}_4.$ They further studied their dual codes with respect to the trace inner product on $\mathbb{F}_{4}^n.$  In the same work, they related the problem of finding quantum error-correcting codes with that of finding self-orthogonal additive codes over $\mathbb{F}_{4}.$  Huffman \cite{huff2} and \cite{huff3} provided a canonical form decomposition of cyclic additive codes  over $\mathbb{F}_{4}.$ Using this decomposition, he further studied and enumerated these codes. Besides this, he obtained the number of self-orthogonal and self-dual cyclic additive codes of length $n$ over $\mathbb{F}_4$ with respect to the trace inner product on $\mathbb{F}_{4}^n.$ In order to  further explore the properties of cyclic $\mathbb{F}_{q}$-linear $\mathbb{F}_{q^t}$-codes, Huffman \cite{huff} generalized the theory developed in \cite{huff2} and viewed cyclic $\mathbb{F}_{q}$-linear $\mathbb{F}_{q^t}$-codes of length $n$ as $\mathbb{F}_{q}[X]/\left<X^n-1\right>$-submodules of the quotient ring $\mathbb{F}_{q^t}[X]/\left<X^n-1\right>,$ where $\gcd(n,q)=1$ and $t \geq 2$ is an integer. He also determined the number of cyclic $\mathbb{F}_{q}$-linear $\mathbb{F}_{q^t}$-codes of length $n$ and studied their dual codes with respect to the ordinary and Hermitian trace bilinear forms on $\mathbb{F}_{q^t}^n.$ In addition to this,  he determined bases of all the  self-orthogonal and self-dual cyclic $\mathbb{F}_{q}$-linear $\mathbb{F}_{q^2}$-codes with respect to these two bilinear forms on $\mathbb{F}_{q^2}^{n}.$ Furthermore, for any integer $t \geq 2,$ he enumerated all the self-dual and self-orthogonal cyclic $\mathbb{F}_{q}$-linear $\mathbb{F}_{q^t}$-codes of length $n$ with respect to these two bilinear forms. In a recent work \cite{sh}, we  introduced and studied a new trace bilinear form, denoted by $\ast,$ on $\mathbb{F}_{q^t}^n$ for any integer $t\geq 2$ satisfying $t\not \equiv 1(\text{mod }p),$ where $p$ is the characteristic of the finite field $\mathbb{F}_{q}.$ We also observed  that the $\ast$ bilinear form on $\mathbb{F}_{q^t}^n$ coincides with the trace inner product considered by Calderbank et al. \cite{cal}  when $q=t=2$  and Hermitian trace inner product considered by  Ezerman et al. \cite{ezer} when $q$ is even and $t=2,$ which are well-studied inner products in coding theory. Further, by placing the bilinear form $\ast$ on $\mathbb{F}_{q^t}^{n},$ we enumerated all the  self-orthogonal and self-dual cyclic $\mathbb{F}_{q}$-linear $\mathbb{F}_{q^t}$-codes of length $n,$ provided $\gcd(n,q)=1.$  Besides this, we explicitly determined bases of all the complementary-dual cyclic $\mathbb{F}_{q}$-linear $\mathbb{F}_{q^2}$-codes of length $n$ with respect to $\ast,$ ordinary and Hermitian trace bilinear forms on $\mathbb{F}_{q^2}^n$ and enumerated these three classes of codes, where $\gcd(n,q)=1.$

The main goal of this paper is
 to enumerate all the complementary-dual cyclic $\mathbb{F}_{q}$-linear $\mathbb{F}_{q^t}$-codes of length $n$ with respect to $\ast$, ordinary and Hermitian trace bilinear forms on $\mathbb{F}_{q^t}^{n},$
where $q$ is a prime power,  $n$ is a positive integer coprime to $q$ and $t \geq 2$ is an integer. Although our enumeration technique is based upon the theory of cyclic $\mathbb{F}_{q}$-linear $\mathbb{F}_{q^t}$-codes developed by Huffman \cite{huff2}, it is quite different from the technique employed by Huffman \cite{huff2} in the enumeration of self-orthogonal and self-dual cyclic $\mathbb{F}_{q}$-linear $\mathbb{F}_{q^t}$-codes.

This paper is organized as follows: In Section \ref{prem}, we state some preliminaries that are needed to derive our main result. In Section \ref{comple}, we determine the number of complementary-dual cyclic $\mathbb{F}_{q}$-linear $\mathbb{F}_{q^t}$-codes of length $n$ with respect to  $\ast,$ ordinary and Hermitian trace bilinear forms on $\mathbb{F}_{q^t}^{n},$ where $q$ is a prime power, $n$ is a positive integer with $\gcd(n,q)=1$ and $t \geq 2$ is an integer (Theorem \ref{completheo}). \vspace{-5mm}
\section{Some preliminaries}\label{prem}
\vspace{-2mm}In this section, we will state some basic definitions and results that are needed to derive our main result.

Throughout this paper, let $q$ be a power of the prime $p,$ $\mathbb{F}_{q}$ denote the finite field with $q$ elements, $n$ be a positive integer coprime to $p$ and $t \geq 2$ be an integer.
Let $\mathcal{R}_n^{(q)}$ and $\mathcal{R}_n^{(q^t)}$ denote the quotient rings $\mathbb{F}_q[X]/ \langle X^n -1\rangle$ and  $\mathbb{F}_{q^t}[X]/ \langle X^n -1\rangle$ respectively, where $X$ is an indeterminate over $\mathbb{F}_p$ and over any extension field of $\mathbb{F}_{p}.$ As $\gcd(n,q)=1,$ by Maschke's Theorem, both the rings $\mathcal{R}_n^{(q)}$ and $\mathcal{R}_n^{(q^t)}$ are semi-simple, and hence can be written as direct sums of minimal ideals. More explicitly, if $\{ \ell_0=0,\ell_1,\cdots,\ell_{s-1}\}$ is a complete set of representatives of $q$-cyclotomic cosets modulo $n,$ then $X^n-1=f_0(X)f_1(X)\cdots f_{s-1}(X)$ is the factorization of $X^n-1$ into monic irreducible polynomials over $\mathbb{F}_{q}$ with $f_i(X)=\displaystyle \prod_{k \in C_{\ell_i}^{(q)}}(X-\eta^k)$ for $0 \leq i \leq s-1,$ where $C_{\ell_i}^{(q)}~( 0 \leq i \leq s-1)$  is the $q$-cyclotomic coset modulo $n$ containing the integer $\ell_i$ and $\eta$ is a primitive $n$th root of unity over $\mathbb{F}_{q}.$
 Therefore for $0 \leq i \leq s-1,$ if $\mathcal{K}_{i}$ is the (minimal) ideal of $\mathcal{R}_n^{(q)}$ generated by the polynomial $ (X^n-1)/f_i(X),$  then it is easy to see that $\mathcal{R}_n^{(q)} = \mathcal{K}_0 \oplus \mathcal{K}_1 \oplus \cdots \oplus \mathcal{K}_{s-1},$ where  $\mathcal{K}_{i}\mathcal{K}_{j} =\{0\}$ for all $i\neq j$ and $\mathcal{K}_i \simeq \mathbb{F}_{q^{d_i}}$ with $d_i$ as the cardinality of $C_{\ell_i}^{(q)}.$
Next to write $\mathcal{R}_n^{(q^t)}$ as a direct sum of minimal ideals, we further factorize the polynomials $f_0(X),f_1(X),\cdots,f_{s-1}(X)$ into monic irreducible polynomials over $\mathbb{F}_{q^t}.$ To do this, by Lemma 1 of Huffman \cite{huff},    we see that for $0 \leq i \leq s-1,$ $C_{\ell_i}^{(q)}=C_{\ell_i}^{(q^t)}\cup C_{\ell_i q}^{(q^t)}\cup \cdots\cup C_{\ell_iq^{g_i-1}}^{(q^t)}$ with $g_i=\gcd(d_i,t),$  which led to the following factorization of $f_i(X)$ over $\mathbb{F}_{q^t}$: $f_i(X)= M_{i,0}(X)M_{i,1}(X)\cdots M_{i,g_i-1}(X),$ where $M_{i,j}(X)=\displaystyle \prod_{k \in C_{\ell_iq^j}^{(q^t)}}(X-\eta^k)$ for $0 \leq j \leq g_i-1.$ Therefore, if for each $i$ and $j,$  $\mathcal{I}_{i,j}$ is the (minimal) ideal of $\mathcal{R}_n^{(q^t)}$ generated by $ (X^n-1)/M_{i,j}(X),$ then one can show that $\mathcal{R}_n^{(q^t)} =\displaystyle \bigoplus_{i=0}^{s-1} \bigoplus_{j=0}^{g_i-1} \mathcal{I}_{i,j},$ where $\mathcal{I}_{i,j}\mathcal{I}_{i^{\prime},j^{\prime}} =\{0\}$ for all $(i,j)\neq (i^{\prime},j^{\prime})$ and $\mathcal{I}_{i,j} \simeq \mathbb{F}_{q^{tD_i}}$ with $D_i= \frac{d_i}{g_i}.$ Further, it is easy to see that $\mathcal{J}_i=\mathcal{I}_{i,0} \oplus \mathcal{I}_{i,1} \oplus \cdots \oplus \mathcal{I}_{i,g_i-1}$ is a vector space of dimension $t$ over $\mathcal{K}_{i}$ for $0\leq i\leq s-1.$

In order to study the containment $\mathcal{R}_{n}^{(q)}\subset \mathcal{R}_{n}^{(q^t)},$ Huffman \cite{huff} defined ring automorphisms $\tau_{q^u,v} : \mathcal{R}_{n}^{(q^r)} \rightarrow \mathcal{R}_{n}^{(q^r)}$ as $\tau_{q^u,v}\biggl(\sum\limits_{i=0}^{n-1}a_iX^i\biggr)=\sum\limits_{i=0}^{n-1}a_i^{q^u}X^{vi}$ for any integer $r \geq 1,$ where $u,v$ are  integers satisfying $0 \leq u \leq r,$ $1 \leq v \leq n$ and $\gcd(v,n)=1.$ When $r=t,$ by Lemma 2 of Huffman \cite{huff}, we see that $\tau_{q^u,v}$ permutes minimal ideals $\mathcal{I}_{i,j}$'s of the ring $\mathcal{R}_{n}^{(q^t)}.$
It is easy to see that for every $i ~(1\leq i \leq s-1),$ there exists a unique integer $i'~(1\leq i'\leq s-1)$ satisfying $C_{-\ell_i}^{(q)}=C_{\ell_{i'}}^{(q)}.$ This gives rise to a permutation $\mu$ of the set $\{0,1,2,\cdots,s-1\},$ defined by $C_{-\ell_i}^{(q)}=C_{\ell_{\mu(i)}}^{(q)}$ for $0 \leq i \leq s-1.$ Note that $\mu(0)=0$ and $\mu(\mu(i))=i$ for $0 \leq i \leq s-1.$ That is, $\mu$ is either the identity permutation or is a product of transpositions. Note that when $n$ is even, there exists an integer $i^{\#}$ satisfying $0 \leq i^{\#} \leq s-1$ and $C_{\ell_{i^{\#}}}^{(q)}=C_{\frac{n}{2}}^{(q)}=\{\frac{n}{2}\},$ as $q$ is odd. From this, we see that  $C_{-\ell_{i^{\#}}}^{(q)}=C_{\ell_{i^{\#}}}^{(q)},$ which implies that $\mu(i^{\#})=i^{\#}.$

Now an $\mathbb{F}_{q}$-linear $\mathbb{F}_{q^t}$-code $\mathcal{C}$ of length $n$ is defined as an $\mathbb{F}_{q}$-linear subspace of $\mathbb{F}_{q^t}^n.$ Further, the code $\mathcal{C}$ is said to be cyclic if $(c_0,c_1,\cdots,c_{n-1}) \in \mathcal{C}$ implies that $(c_{n-1},c_0,c_1,\cdots,c_{n-2}) \in \mathcal{C}.$
Throughout this paper, we shall identify each vector  $a=(a_0,a_1,\cdots,a_{n-1}) \in \mathbb{F}_{q^t}^{n}$ with the representative $a(X)=\sum\limits_{i=0}^{n-1}a_iX^i$ of the coset  $a(X)+\left<X^n-1\right> \in \mathcal{R}_n^{(q^t)},$ and perform addition and multiplication of its representatives modulo $X^n-1.$ Under this identification, one can easily observe that the cyclic shift $\sigma(a)=(a_{n-1},a_0,a_1,\cdots,a_{n-2})$ of $a \in \mathbb{F}_{q^t}^n$ is identified with $Xa(X) \in \mathcal{R}_{n}^{(q^t)}.$ In view of this, every cyclic $\mathbb{F}_{q}$-linear $\mathbb{F}_{q^t}$-code of length $n$ can be viewed as an $\mathcal{R}_n^{(q)}$-submodule of $\mathcal{R}_n^{(q^t)}.$  Huffman \cite{huff} studied dual codes of cyclic $\mathbb{F}_{q}$-linear $\mathbb{F}_{q^t}$-codes of length $n$ with respect to ordinary and Hermitian trace bilinear forms on $\mathbb{F}_{q^t}^n,$ which are  defined as follows:

Let $ \text{Tr}_{q,t}:\mathbb{F}_{q^t}\rightarrow \mathbb{F}_q $ be the trace map defined as $\text{Tr}_{q,t}(\alpha)=\sum\limits_{j=0}^{t-1}\alpha^{q^{j}}$ for each $\alpha \in \mathbb{F}_{q^t}.$ It is well-known that $\text{Tr}_{q,t}$ is an $\mathbb{F}_q$-linear, surjective map with kernel of size $q^{t-1}$ (see \cite[Th. 2.23]{lidl}).
Then for any integer $t \geq 2,$ the ordinary trace inner product on $\mathbb{F}_{q^t}^n$ is a map $(\cdot,\cdot)_{0}: \mathbb{F}_{q^t}^n \times \mathbb{F}_{q^t}^n \rightarrow \mathbb{F}_{q},$ defined as $( a,b )_{0} = \sum\limits_{j=0}^{n-1}\text{Tr}_{q,t}(a_jb_j) \text{ ~for ~all~ }a=(a_0,a_1, \cdots ,a_{n-1}),~ b =(b_0,b_1,\ldots ,b_{n-1})\in \mathbb{F}_{q^t}^n,$
and the map $\left[\cdot,\cdot\right]_{0} : \mathcal{R}_{n}^{(q^t)} \times \mathcal{R}_{n}^{(q^t)} \rightarrow \mathcal{R}_{n}^{(q)}$ is  defined as
$\left[a(X),b(X)\right]_{0}= \sum\limits_{w=0}^{t-1} \tau_{q^w,1}\left(a(X)\tau_{1,-1}(b(X))\right)\text{ ~for~ all~} a(X),~b(X) \in \mathcal{R}_{n}^{(q^t)}.$ The Hermitian trace inner product on $\mathbb{F}_{q^t}^n$ is defined only for even integers $t \geq 2,$ which can be written as $t=2^am,$ where $ a\geq 1$ and $m$ is odd. It is easy to see that there exists an element $\gamma \in \mathbb{F}_{q^{2^a}}$ satisfying $\gamma+\gamma^{q^{2^{a-1}}}=0.$
 Then the Hermitian trace inner product  on $\mathbb{F}_{q^t}^n$ is a map $(\cdot,\cdot)_{\gamma}: \mathbb{F}_{q^t}^n \times \mathbb{F}_{q^t}^n \rightarrow \mathbb{F}_{q},$  defined as $(a, b)_{\gamma} = \sum\limits_{j=0}^{n-1}\text{Tr}_{q,t}(\gamma a_j b_j^{q^{t/2}})\text{ ~for~all~ }a=(a_0,a_1, \cdots ,a_{n-1}),~ b =(b_0,b_1,\ldots ,b_{n-1}) \in  \mathbb{F}_{q^t}^n,$
 and the map  $\left[\cdot,\cdot\right]_{\gamma} : \mathcal{R}_{n}^{(q^t)} \times \mathcal{R}_{n}^{(q^t)} \rightarrow \mathcal{R}_{n}^{(q)}$ is defined as
 $\left[a(X),b(X)\right]_{\gamma}= \sum\limits_{w=0}^{t-1} \tau_{q^w,1}\left(\gamma a(X)\tau_{q^{t/2},-1}(b(X))\right)\text{ ~for~all~ } a(X),~b(X) \in \mathcal{R}_{n}^{(q^t)}.$
 By Lemma 6 of Huffman \cite{huff}, we see that $\left[\cdot,\cdot\right]_{0}$ is a non-degenerate, reflexive  and Hermitian $\tau_{1,-1}$-sesquilinear form on $\mathcal{R}_{n}^{(q^t)},$ while $\left[\cdot,\cdot\right]_{\gamma}$ is a non-degenerate, reflexive and skew-Hermitian $\tau_{1,-1}$-sesquilinear form on $\mathcal{R}_{n}^{(q^t)}.$ Sharma and Kaur \cite{sh} studied dual codes of cyclic $\mathbb{F}_{q}$-linear $\mathbb{F}_{q^t}$-codes with respect to the bilinear form $\ast$ on $\mathbb{F}_{q^t}^{n},$ which is defined only  for  integers $t \geq 2$ satisfying $t \not \equiv 1(\text{mod }p).$ In order to define the form $\ast$, we see that for any integer $t \geq 2$ satisfying $t \not\equiv 1(\text{mod }p),$   the map $ \phi :\mathbb{F}_{q^t} \rightarrow \mathbb{F}_{q^t} ,$ defined as $\phi(\alpha) = \alpha^q+\alpha^{q^2}+\cdots+\alpha^{q^{t-1}}\text{ for each } \alpha\in \mathbb{F}_{q^t},$
 is an $\mathbb{F}_q$-linear vector space automorphism. Now the $\ast$-bilinear form on $\mathbb{F}_{q^t}^n$ is a map $(\cdot,\cdot)_{\ast}: \mathbb{F}_{q^t}^n \times \mathbb{F}_{q^t}^n \rightarrow \mathbb{F}_{q},$ defined as $( a,b )_{\ast} = \sum\limits_{j=0}^{n-1}\text{Tr}_{q,t}(a_j\phi(b_j)) \text{ ~for~all~ } a=(a_0,a_1, \cdots ,a_{n-1}), ~ b =(b_0,b_1,\ldots ,b_{n-1}) \text{ ~in~}  \mathbb{F}_{q^t}^n,$ and  the map $\left[\cdot,\cdot\right]_{\ast} : \mathcal{R}_{n}^{(q^t)} \times \mathcal{R}_{n}^{(q^t)} \rightarrow \mathcal{R}_{n}^{(q)}$ is defined as
 $\left[ a(X),b(X) \right]_{\ast} = \displaystyle \sum_{u=0}^{t-1} \tau_{q^u,1}\biggl( a(X)\displaystyle\sum_{w=1}^{t-1}\tau_{q^w,-1}\bigl(b(X)\bigr) \biggr) \text{ ~for~all~ } a(X),~b(X) \in \mathcal{R}_{n}^{(q^t)}.$  By Lemma 3.3 of Sharma and Kaur \cite{sh}, we see that $\left[\cdot,\cdot\right]_{\ast}$ is a  non-degenerate, reflexive and Hermitian $\tau_{1,-1}$-sesquilinear form on $\mathcal{R}_n^{(q^t)}.$

Next for $\delta \in \{\ast, 0,\gamma \},$  throughout this paper, let $\mathbb{T}_{\delta}$ be defined as (i) the set of all integers $t \geq 2$ satisfying $t \not \equiv 1(\text{mod }p)$ when $\delta=\ast,$ (ii)  the set of all integers $t \geq 2$ when $\delta=0,$ and (iii) the set of all even integers $t \geq 2$ when $\delta=\gamma.$
Now for each $\delta \in \{\ast,0,\gamma \}$ with $t \in \mathbb{T}_{\delta},$ the $\delta$-dual code of $\mathcal{C}$ is defined as $\mathcal{C}^{\perp_{\delta}} = \{v \in  \mathbb{F}_{q^t}^n : \left( v,c \right)_{\delta} =0 \text{ for all } c \in \mathcal{C}\}.$ It is easy to see that the dual code $\mathcal{C}^{\perp_{\delta}}$ is also an $\mathbb{F}_{q}$-linear $\mathbb{F}_{q^t}$-code of length $n.$ We also note that if the code $\mathcal{C}$ is cyclic, then its dual code $ \mathcal{C}^{\perp_{\delta}} $ is also cyclic. In fact, if $\mathcal{C} \subseteq \mathcal{R}_{n}^{(q^t)}$ is any cyclic $\mathbb{F}_{q}$-linear $\mathbb{F}_{q^t}$-code, then one can easily view its dual code $\mathcal{C}^{\perp_{\delta}} \subseteq \mathcal{R}_{n}^{(q^t)}$ as the dual code of $\mathcal{C}$ with respect to the form $\left[\cdot,\cdot\right]_{\delta}$ on $\mathcal{R}_{n}^{(q^t)}.$ Further, by Theorem 4 of Huffman \cite{huff}, we see that $\mathcal{C}= \mathcal{C}_0 \oplus \mathcal{C}_1 \oplus \cdots \oplus\mathcal{C}_{s-1}$ and $ \mathcal{C}^{\perp_{\delta}}= \mathcal{C}_0^{(\delta)} \oplus \mathcal{C}_1^{(\delta)} \oplus \cdots \oplus \mathcal{C}_{s-1}^{(\delta)},$ where $\mathcal{C}_i= \mathcal{C} \cap \mathcal{J}_i$ and $ \mathcal{C}_i^{(\delta)} = \mathcal{C}^{\perp_{\delta}} \cap \mathcal{J}_i $ are  $\mathcal{K}_i$-subspaces of $\mathcal{J}_i$ for $0 \leq i\leq s-1.$  In addition, by Theorem 7 of Huffman \cite{huff} and Theorem 4.1  of Sharma and Kaur \cite{sh}, we have $\mathcal{C}_{\mu(i)}^{(\delta)} = \{a(X) \in \mathcal{J}_{\mu(i)}: \left[ a(X),c(X) \right]_{\delta} = 0 \text{ for all } c(X) \in \mathcal{C}_i \}$ and $\text{dim}_{\mathcal{K}_{\mu(i)}}\mathcal{C}_{\mu(i)}^{(\delta)}= t- \text{dim}_{\mathcal{K}_i}\mathcal{C}_i$ for $0 \leq i \leq s-1,$ (throughout this paper, $\text{dim}_{F}V$ denotes the dimension of a finite-dimensional vector space $V$ over the field $F$). Now for $\delta \in \{\ast,0,\gamma\},$ the code $\mathcal{C}$ is said to be $\delta$-complementary-dual if it satisfies $\mathcal{C} \cap \mathcal{C}^{\perp_{\delta}}=\{0\}.$ For more details, one may refer to Huffman \cite{huff} and Sharma and Kaur \cite{sh}.
\\\noindent  Henceforth,  we shall follow the same notations and terminology as in Section \ref{prem}. \vspace{-5mm}
\section{Determination of the number of $\delta$-complementary-dual cyclic $\mathbb{F}_{q}$-linear $\mathbb{F}_{q^t}$-codes}\label{comple}
 \vspace{-2mm}  Throughout this paper, let $n$ be a positive integer, $q$ be a power of the prime $p$ with $\gcd(n,q)=1$ and  $t \in \mathbb{T}_{\delta}$ for $\delta \in \{\ast,0,\gamma\}.$ In this section,  we shall count all the $\delta$-complementary-dual cyclic $\mathbb{F}_{q}$-linear $\mathbb{F}_{q^t}$-codes of length $n$ for each $\delta \in \{\ast, 0,\gamma\}.$
 For this, we first recall that the permutation $\mu$ of the set $\{0,1,2, \cdots, s-1\}$ is either the identity or is a product of transpositions, and it satisfies $\mu(0)=0$ and $\mu(i^{\#})=i^{\#}$ (provided $n$ is even). Let $\mathfrak{I}=\{0\}$ if $n$  is odd and $\mathfrak{I}=\{0,i^{\#}\}$ if $n$ is even. Let $\mathfrak{F}$ denote the set consisting of all the fixed points of $\mu$  excluding the points of $\mathfrak{I}$  and  $\mathfrak{M}$ denote the set consisting of exactly one element from each of the transpositions in $\mu.$ Throughout this paper, we shall denote the  $q$-binomial coefficient or Gaussian binomial coefficient by ${\cdot \brack \cdot }_q,$ which is   defined as ${a \brack b }_q=\prod\limits_{i=0}^{b-1}\left(\frac{q^a-q^i}{q^b-q^i}\right),$ where $a,b$ are integers satisfying $1 \leq b \leq a.$  From now onwards, we shall denote the number of isotropic vectors and hyperbolic pairs in a reflexive and non-degenerate formed space of dimension $n$ and Witt index $m$  by $I_{m,n-2m}$ and $H_{m,n-2m},$ respectively.

 In the following theorem, we enumerate all the $\delta$-complementary-dual cyclic $\mathbb{F}_{q}$-linear $\mathbb{F}_{q^t}$-codes of length $n$ for each $\delta \in \{\ast,0,\gamma\}.$
\begin{theo}\label{complecount}
  Let $q$ be a power of the prime $p$ and $n$ be a positive integer with $\gcd(n,q)=1.$  Then for $\delta \in \{\ast,0,\gamma\}$ with $t \in \mathbb{T}_{\delta},$ the number  $N$  of distinct $\delta$-complementary-dual cyclic $\mathbb{F}_q$-linear $\mathbb{F}_{q^t}$-codes of length $n$   is given by \vspace{-2mm}$$N=\mathfrak{R}^{\gcd(n,2)}\displaystyle\prod \limits_{i \in \mathfrak{F}} \left\{ 2+\sum\limits_{r=1}^{t-1}q^{\frac{r(t-r)d_i}{2}} \prod\limits_{j=0}^{r-1}\biggl(\frac{q^{(t-j)\frac{d_i}{2}}-(-1)^{t-j}}{q^{(r-j)\frac{d_i}{2}}-(-1)^{r-j}} \biggr) \right\} \prod \limits_{h \in \mathfrak{M}}\left\{2+\sum\limits_{\ell=1}^{t-1} q^{\ell(t-\ell)d_h} {t \brack \ell}_{q^{d_h}}\right\},\vspace{-2mm}$$ where $\mathfrak{R}$ equals
\begin{itemize}\vspace{-2mm}\item $2+\displaystyle \sum\limits_{\stackrel{k=2} {k\equiv 0(\text{mod }2) }}^{t-1} q^{\frac{k(t-k)}{2}}{t/2 \brack k/2}_{q^2}$ when either $\delta=\ast$ and both $t,q$ are even or   $\delta=\gamma$ and $t$ is even;

\vspace{-2mm}\item $\displaystyle 2+\sum\limits_{\stackrel{k=1}{k\equiv 0(\text{mod }2) }}^{t-1} q^{\frac{k(t-k+1)}{2}}{ (t-1)/2 \brack k/2}_{q^2} +  \sum\limits_{\stackrel{k=1}{k\equiv 1(\text{mod }2) } }^{t-1}q^{\frac{(t-k)(k+1)}{2}}  { (t-1)/2  \brack  (k-1)/2 }_{q^2}\vspace{2mm}$ when $\delta\in \{\ast,0\}$  and both $t,q$ are odd;

\vspace{-2mm}\item $\displaystyle 2+\sum\limits_{\stackrel{k=1}{k\equiv 0(\text{mod }2) }}^{t-1} q^{\frac{k(t-k)}{2}} {t/2 \brack k/2}_{q^2} +  \sum\limits_{\stackrel{k=1}{k\equiv 1(\text{mod }2) }}^{t-1} q^{\frac{tk-k^2-1}{2}}(q^{\frac{t}{2}}+1) {(t-2)/2 \brack (k-1)/2}_{q^2}\vspace{2mm}$ when $\delta\in \{\ast,0\}$ with either $t$ even and $q \equiv 1(\text{mod }4)$ or  $t \equiv 0(\text{mod }4)$ and $q \equiv 3(\text{mod }4).$
\vspace{-2mm}\item $\displaystyle 2+\sum\limits_{\stackrel{k=1}{k\equiv 0(\text{mod }2) }}^{t-1} q^{\frac{k(t-k)}{2}} {t/2 \brack k/2}_{q^2} + \sum\limits_{\stackrel{k=1}{k\equiv 1(\text{mod }2) }}^{t-1} q^{\frac{tk-k^2-1}{2}} (q^{\frac{t}{2}}-1){(t-2)/2 \brack (k-1)/2}_{q^2}\vspace{2mm}$ when $\delta\in \{\ast,0\},$ $t \equiv 2(\text{mod }4)$ and $q \equiv 3(\text{mod }4).$

\vspace{-2mm}\item $2+\displaystyle \sum\limits_{\stackrel{k=1}{k\equiv 0(\text{mod }2) }}^{t-1}q^{\frac{k(t-k+1)}{2}}{(t-1)/2 \brack k/2}_{q^2} + \sum\limits_{\stackrel{k=1}{k\equiv 1(\text{mod }2) }}^{t-1}q^{\frac{(t-k)(k+1)}{2}} {(t-1)/2 \brack (k-1)/2}_{q^2}\vspace{2mm}$ when  $\delta =0,$  $q$ is even and $t$ is odd;
\vspace{-2mm}\item $\displaystyle 2+  \sum\limits_{\stackrel{k=1}{k\equiv 0(\text{mod }2) }}^{t-1}q^{\frac{tk-k^2-2}{2}}\Big\{(q^{k}+q-1){(t-2)/2 \brack k/2}_{q^2} +(q^{t-k+1}-q^{t-k}+1){(t-2)/2 \brack (k-2)/2}_{q^2} \Big\}+\\ \sum\limits_{\stackrel{k=1}{k\equiv 1(\text{mod }2) }}^{t-1} q^{\frac{tk-k^2+t-1}{2}} {(t-2)/2 \brack (k-1)/2}_{q^2} \vspace{2mm}$ when $\delta=0$ and both $q, t$ are even.
\end{itemize}
\end{theo}
To prove the above theorem,  we first observe the following:
\begin{prop}\label{completheo} Let   $q$ be a power of the prime $p$ and $n$ be a positive integer  coprime to $q.$  Then for $\delta \in \{\ast,0,\gamma\}$ with $t \in \mathbb{T}_{\delta},$  the following hold:
\begin{enumerate}\vspace{-2mm}\item[(a)] Let $\mathcal{C}\subseteq \mathcal{R}_{n}^{(q^t)}$ be a cyclic $\mathbb{F}_{q}$-linear $\mathbb{F}_{q^t}$-code of length $n.$ Let us write $ \mathcal{C}= \mathcal{C}_0 \oplus \mathcal{C}_1 \oplus \cdots \oplus \mathcal{C}_{s-1} $ and $ \mathcal{C}^{\perp_{\delta}} = \mathcal{C}_0^{(\delta)} \oplus \mathcal{C}_1^{(\delta)}\oplus \cdots \oplus \mathcal{C}_{s-1}^{(\delta)},$ where $\mathcal{C}_i = \mathcal{C} \cap \mathcal{J}_i $ and $ \mathcal{C}_i^{(\delta)}  = \mathcal{C}^{\perp_{\delta}} \cap \mathcal{J}_i$ for $0 \leq i \leq s-1.$ Then the code $\mathcal{C} $ is ${\delta}$-complementary-dual if and only if $\mathcal{C}_{i} \cap \mathcal{C}_{i}^{(\delta)}=\{0\}$ for $0\leq i \leq s-1.$
 \vspace{-2mm}\item[(b)] For each $i \in \mathfrak{F} \cup \mathfrak{I},$ let $N_i$ denote the number of  $\mathcal{K}_i$-subspaces $\mathcal{C}_{i}$ of $\mathcal{J}_i$ satisfying $\mathcal{C}_{i} \cap \mathcal{C}_{i}^{(\delta)} =\{0\}.$ For each $h \in \mathfrak{M},$ let $N_h$ denote the number of pairs $(\mathcal{C}_{h},\mathcal{C}_{\mu(h)})$ with  $\mathcal{C}_{h}$ as a $\mathcal{K}_{h}$-subspace  of $\mathcal{J}_{h}$ and $\mathcal{C}_{\mu(h)}$ as a $\mathcal{K}_{\mu(h)}$-subspace  of $\mathcal{J}_{\mu(h)}$ satisfying $\mathcal{C}_{h} \cap \mathcal{C}_{h}^{(\delta)} =\{0\}$ and $\mathcal{C}_{\mu(h)} \cap \mathcal{C}_{\mu(h)}^{(\delta)} =\{0\}.$ Then the total number of distinct $\delta$-complementary-dual cyclic $\mathbb{F}_{q}$-linear $\mathbb{F}_{q^t}$-codes of length $n$ is given by $\text{N}=\prod\limits_{i \in \mathfrak{F}\cup \mathfrak{I}}N_i\prod\limits_{h \in \mathfrak{M}}N_h.$
\end{enumerate}
 \vspace{-5mm}\end{prop}
 \vspace{-2mm}\begin{proof}
Proof is trivial.
\vspace{-2mm}\end{proof}
In view of the above proposition, we see that to enumerate all $\delta$-complementary-dual cyclic $\mathbb{F}_{q}$-linear $\mathbb{F}_{q^t}$-codes of length $n$ for each $\delta \in \{\ast,0,\gamma\},$ we need to determine   the following:
\begin{itemize}
\vspace{-2mm}\item For each $i \in \mathfrak{F}\cup \mathfrak{I},$ the number $N_i$ of distinct $\mathcal{K}_{i}$-subspaces $\mathcal{C}_{i}$ of $\mathcal{J}_{i}$ satisfying $\mathcal{C}_{i} \cap \mathcal{C}_{i}^{(\delta)}=\{0\}.$\vspace{-2mm} \item For each $h \in \mathfrak{M},$ the number $N_h$ of pairs $\left(\mathcal{C}_{h},\mathcal{C}_{\mu(h)}\right)$ with $\mathcal{C}_{h}$ as a $\mathcal{K}_{h}$-subspace of $\mathcal{J}_{h}$ and $\mathcal{C}_{\mu(h)}$ as a $\mathcal{K}_{\mu(h)}$-subspace of $\mathcal{J}_{\mu(h)}$ satisfying $\mathcal{C}_{h} \cap \mathcal{C}_{h}^{(\delta)}=\{0\}$ and $\mathcal{C}_{\mu(h)} \cap \mathcal{C}_{\mu(h)}^{(\delta)}=\{0\}.$
\end{itemize}

First of all, we shall determine the number $N_i$ for each $i \in \mathfrak{F}\cup \mathfrak{I}.$ For this, we need the following lemma:
\begin{lem} \label{non-deg} Let $i \in \mathfrak{F}\cup \mathfrak{I}$ be fixed. For each $\delta \in \{\ast, 0,\gamma\},$ let $\left[\cdot,\cdot\right]_{\delta}\restriction _{\mathcal{J}_i \times \mathcal{J}_i}$ denote the restriction of the $\tau_{1,-1}$-sesquilinear form $\left[\cdot,\cdot\right]_{\delta}$ to $\mathcal{J}_i \times \mathcal{J}_i.$ Then the following hold.
\begin{enumerate}\vspace{-2mm}\item[(a)] For each $\delta \in \{\ast,0,\gamma\},$ the form $\left[\cdot,\cdot\right]_{\delta}\restriction _{\mathcal{J}_i \times \mathcal{J}_i}$ is reflexive and non-degenerate.
\vspace{-2mm}\item[(b)] For $\delta \in \{\ast,0\},$ the form $\left[\cdot,\cdot\right]_{\delta}\restriction _{\mathcal{J}_i \times \mathcal{J}_i}$ is Hermitian when $i \in \mathfrak{F}$ and is symmetric when $i \in \mathfrak{I}.$ For $\delta=\gamma,$ the form $\left[\cdot,\cdot\right]_{\gamma}\restriction _{\mathcal{J}_i \times \mathcal{J}_i}$  is skew-Hermitian when $i \in \mathfrak{F}$ and is alternating when $i \in \mathfrak{I}.$ \end{enumerate}
\vspace{-5mm}\end{lem}
\vspace{-2mm}\begin{proof}
It follows from Lemmas 6 and 13 of Huffman \cite{huff} and Lemma 4.3 of Sharma and Kaur \cite{sh}.
\vspace{-2mm}\end{proof}
Next we make the following observation.
\begin{rem} \label{non-deg1}  Let $\mathcal{C}$ be a cyclic $\mathbb{F}_{q}$-linear $\mathbb{F}_{q^t}$-code of length $n.$ For each $i \in \mathfrak{F}\cup \mathfrak{I}$ and $\delta \in \{\ast, 0,\gamma\},$ if $\mathcal{C}_i=\mathcal{C}\cap \mathcal{J}_{i},$ then by Lemma \ref{completheo}, the $\mathcal{K}_{i}$-subspace $\mathcal{C}_{i}^{(\delta)}=\mathcal{C}^{\perp_{\delta}} \cap \mathcal{J}_{i}=\{a(X)\in \mathcal{J}_{i}: \left[a(X),c(X)\right]_{\delta}=0\text{ for all }c(X) \in \mathcal{C}_{i}\}$ can be viewed as the orthogonal complement of $\mathcal{C}_i$ with respect to the restricted form  $\left[\cdot,\cdot\right]_{\delta}\restriction _{\mathcal{J}_i \times \mathcal{J}_i}.$ From this, it follows that a $\mathcal{K}_{i}$-subspace $\mathcal{C}_{i}$ of $\mathcal{J}_{i}$ satisfies $\mathcal{C}_{i} \cap \mathcal{C}_{i}^{(\delta)}=\{0\}$ if and only if $\mathcal{C}_i$ is a non-degenerate $\mathcal{K}_{i}$-subspace of $\left(\mathcal{J}_{i},\left[\cdot,\cdot\right]_{\delta}\restriction _{\mathcal{J}_i \times \mathcal{J}_i}\right).$ Therefore the number $N_i$ equals the number of non-degenerate $\mathcal{K}_{i}$-subspaces of $\left(\mathcal{J}_{i},\left[\cdot,\cdot\right]_{\delta}\restriction _{\mathcal{J}_i \times \mathcal{J}_i}\right) $ for each $i \in \mathfrak{F}\cup \mathfrak{I}$ and $\delta \in \{\ast, 0,\gamma\}.$ \end{rem}
\vspace{-4mm}\subsection{Determination of the number $N_i$ when $i\in \mathfrak{F}$}\label{F}
Throughout this section, we assume that $i \in \mathfrak{F}.$ Here we recall that $\mathcal{J}_i$ is a $t$-dimensional vector space over $\mathcal{K}_i \simeq \mathbb{F}_{q^{d_i}}.$ Further, by Lemma 10(i) of Huffman \cite{huff}, we see that the integer $d_i$ is even.
Then in the following proposition, we determine the number $N_i$ for each  $\delta \in \{\ast,0,\gamma\}.$
\begin{prop}\label{complet}
Let $i \in \mathfrak{F}$ be fixed. Then for each $\delta\in \{\ast,0,\gamma\},$ the number $N_i$ of distinct $\mathcal{K}_i$-subspaces $\mathcal{C}_i$ of  $\mathcal{J}_i$ satisfying $\mathcal{C}_i \cap \mathcal{C}_i^{(\delta)} =\{0\}$ is given by \vspace{-2mm}$N_i=2+\displaystyle \sum\limits_{k=1}^{t-1} q^{\frac{k(t-k)d_i}{2}} \prod\limits_{j=0}^{k-1}\left(\frac{q^{\frac{(t-j)d_i}{2}} -(-1)^{t-j}} {q^{\frac{(k-j)d_i}{2}}-(-1)^{k-j}}\right).$
\end{prop}
\begin{proof}
 To prove this, for each integer $k~(0 \leq k \leq t),$ let $N_{i,k}$ denote the number of $k$-dimensional $\mathcal{K}_i$-subspaces $\mathcal{C}_i$ of  $\mathcal{J}_i$ satisfying $\mathcal{C}_i \cap \mathcal{C}_i^{(\delta)} =\{0\}.$ Then we have $N_i=\sum\limits_{k=0}^{t}N_{i,k}.$

When $k=0,$ we have $\mathcal{C}_{i}=\{0\},$ which gives $\mathcal{C}_{i}^{(\delta)}=\mathcal{J}_{i}$ and thus $\mathcal{C}_{i} \cap \mathcal{C}_{i}^{(\delta)}=\{0\}$ holds. This implies that $N_{i,0}=1.$ Moreover, by Lemma \ref{non-deg}(a), we see that $\left(\mathcal{J}_{i},\left[\cdot,\cdot\right]_{\delta}\restriction _{\mathcal{J}_i \times \mathcal{J}_i}\right)$  is a $t$-dimensional reflexive and non-degenerate space over $\mathcal{K}_i$. This implies that when $\mathcal{C}_i=\mathcal{J}_i,$ we have $\mathcal{C}_i^{(\delta)}=\{0\},$ and thus $\mathcal{C}_i \cap \mathcal{C}_i^{(\delta)} =\{0\}$ holds. This  gives $N_{i,t}=1.$ So from now onwards, we assume that $1 \leq k \leq t-1$ and we assert that  $ N_{i,k}=\displaystyle q^{\frac{k(t-k)d_i}{2}} \prod\limits_{j=0}^{k-1}\left(\frac{q^{\frac{(t-j)d_i}{2}} -(-1)^{t-j}} {q^{\frac{(k-j)d_i}{2}}-(-1)^{k-j}}\right)$ for $1 \leq k \leq t-1.$ To prove  this assertion, by Remark \ref{non-deg1}, we note that the number $N_{i,k}$ is equal to the number of  $k$-dimensional non-degenerate $\mathcal{K}_{i}$-subspaces of $\mathcal{J}_{i}$ for $1 \leq k \leq t-1.$ Now we shall consider the following two cases separately: \textbf{I. } $\delta \in \{\ast,0\}$ and \textbf{II. } $\delta=\gamma.$
\\\textbf{Case I.} Let $\delta \in \{\ast,0\}.$ Here by Lemma \ref{non-deg}, we see that $\left[\cdot, \cdot\right]_{\delta}\restriction_{\mathcal{J}_{i} \times \mathcal{J}_{i}}$ is a reflexive, non-degenerate and Hermitian  form, i.e., $\left(\mathcal{J}_i, \left[\cdot, \cdot\right]_{\delta}\restriction_{\mathcal{J}_{i} \times \mathcal{J}_{i}}\right)$ is a unitary space over $\mathcal{K}_{i}.$

First  let $k$ be odd. By \cite[p. 116]{tay}, we see that any $k$-dimensional non-degenerate  (and hence unitary) $\mathcal{K}_{i}$-subspace of $\mathcal{J}_i$  has a Witt decomposition of the form $ \langle a_1(X),b_1(X)\rangle \perp \cdots \perp \langle a_{\frac{k-1}{2}}(X),b_{\frac{k-1}{2}}(X)\rangle \perp \langle w(X)\rangle,$ where $\big(a_{\ell}(X),b_{\ell}(X)\big)$ is a hyperbolic pair in $\mathcal{J}_{i}$  for $1 \leq \ell \leq \frac{k-1}{2}$ and $w(X)$ is an anisotropic vector in  $\mathcal{J}_i.$ Now in order to determine the number of $k$-dimensional non-degenerate (and hence unitary) $\mathcal{K}_{i}$-subspaces of $\mathcal{J}_i,$ we shall count all Witt bases of the type $\{a_1(X),b_1(X),\cdots,a_{\frac{k-1}{2}}(X),b_{\frac{k-1}{2}}(X),w(X)\}$ in $\mathcal{J}_i.$ For this, we note that $(a_1(X),b_1(X))$ is a hyperbolic pair in $\mathcal{J}_i,$ so by Corollary 10.6 of \cite{tay}, it has $H_{\mathfrak{r}_1,t-2\mathfrak{r}_1}$ choices, where $\mathfrak{r}_1$ is the Witt index of $\mathcal{J}_i.$ Now since $\langle a_1(X),b_1(X)\rangle$ is a non-degenerate $\mathcal{K}_i$-subspace of $\mathcal{J}_i,$ working as in Proposition 2.9 of \cite{grove}, we can write $\mathcal{J}_i=\langle a_1(X),b_1(X)\rangle \perp \langle a_1(X),b_1(X)\rangle^{\perp_{\delta}},$ where $\langle a_1(X),b_1(X)\rangle^{\perp_{\delta}} $ is a $(t-2)$-dimensional non-degenerate $\mathcal{K}_i$-subspace of $\mathcal{J}_{i}.$ Next we need to choose the second hyperbolic pair $(a_2(X),b_2(X))$ from the unitary $\mathcal{K}_i$-subspace $\langle a_1(X),b_1(X)\rangle^{\perp_{\delta}}$ of $\mathcal{J}_{i}.$  By Corollary 10.6 of \cite{tay}, we see that $(a_2(X),b_2(X))$ has $H_{\mathfrak{r}_2,t-2-2\mathfrak{r}_2}$ choices, where $\mathfrak{r}_2$ is the Witt index of $\langle a_1(X),b_1(X)\rangle^{\perp_{\delta}}.$ 
Continuing like this, we see that the $(\frac{k-1}{2})$th hyperbolic pair  $(a_{\frac{k-1}{2}}(X),b_{\frac{k-1}{2}}(X))$ has to be chosen from the $(t-k+3)$-dimensional unitary  $\mathcal{K}_i$-subspace $\langle a_1(X),b_1(X), \cdots , a_{\frac{k-3}{2}}(X),b_{\frac{k-3}{2}}(X)\rangle^{\perp}$ of $\mathcal{J}_i,$ so  by Corollary 10.6 of \cite{tay}, the hyperbolic pair $(a_{\frac{k-1}{2}}(X),b_{\frac{k-1}{2}}(X))$ has $H_{\mathfrak{r}_{\frac{k-1}{2}},t-k+3-2\mathfrak{r}_{\frac{k-1}{2}}}$ choices,
where $\mathfrak{r}_{\frac{k-1}{2}}$ is the Witt index of $\langle a_1(X),b_1(X), \cdots, a_{\frac{k-3}{2}},b_{\frac{k-3}{2}} \rangle^{\perp_{\delta}}.$ Next we need to choose an anisotropic vector $w(X)$ from the $(t-k+1)$-dimensional unitary  $\mathcal{K}_i$-subspace $\langle a_1(X),b_1(X), \cdots, a_{\frac{k-1}{2}},b_{\frac{k-1}{2}}\rangle^{\perp_{\delta}}$ of $\mathcal{J}_i.$ By Lemma 10.4 of \cite{tay}, we see that $w(X)$ has $q^{(t-k+1)d_i}-1-I_{\mathfrak{r}_{\frac{k+1}{2}},t-k+1-2\mathfrak{r}_{\frac{k+1}{2}}}$ choices, where $\mathfrak{r}_{\frac{k+1}{2}}$ is the Witt index of  $\langle a_1(X),b_1(X), \cdots, a_{\frac{k-1}{2}}, b_{\frac{k-1}{2}}\rangle^{\perp_{\delta}}.$ Therefore the number of choices for a Witt basis of cardinality $k$ in $\mathcal{J}_{i}$  is given by \vspace{-4mm}$$\mathfrak{U}_{k,t}=
H_{\mathfrak{r}_{1},t-2\mathfrak{r}_1}H_{\mathfrak{r}_2,
t-2-2\mathfrak{r}_2} \cdots H_{\mathfrak{r}_{\frac{k-1}{2}},t-k+3-2\mathfrak{r}_{\frac{k-1}{2}}}\bigl( q^{(t-k+1)d_i}-1-I_{\mathfrak{r}_{\frac{k+1}{2}}, t-k+1-2\mathfrak{r}_{\frac{k+1}{2}}}\bigr).\vspace{-3mm}$$ By \cite[p. 116]{tay}, for $1 \leq \ell \leq \frac{k+1}{2},$ we note that \vspace{-4mm}$$t-2(\ell-1)-2\mathfrak{r}_{\ell}=\left\{\begin{array}{ll} 0 &\text{if }t \text{ is even;}\\ 1 & \text{if } t \text{ is odd.}\end{array}\right.\vspace{-5mm}$$
Now using Lemma 10.4 and Corollary 10.6 of \cite{tay}, we see that  $\mathfrak{U}_{k,t} =q^{\frac{(2t-k-1)kd_i}{4}} (q^{\frac{d_i}{2}}-1) \prod\limits_{j=0}^{k-1}(q^{\frac{(t-j)d_i}{2}}-(-1)^{t-j}).$ Next working in a similar manner as above and using Lemma 10.4 and Corollary 10.6 of \cite{tay}, we see that the number of Witt bases of  a $k$-dimensional unitary $\mathcal{K}_i$-subspace of $\mathcal{J}_i$ is given by $\mathfrak{U}_{k,k}= H_{\frac{k-1}{2},1}H_{\frac{k-3}{2},1}\cdots H_{1,1}(q^{d_i}-1)=q^{\frac{k(k-1)d_i}{4}} (q^{\frac{d_i}{2}}-1) \prod\limits_{j=0}^{k-1} (q^{\frac{(k-j)d_i}{2}}-(-1)^{k-j}).$ Therefore the number of distinct $k$-dimensional non-degenerate (and hence unitary) $\mathcal{K}_{i}$-subspaces  of $\mathcal{J}_i$ is given by \vspace{-2mm}\begin{equation*}N_{i,k}=\frac{\mathfrak{U}_{k,t}}{\mathfrak{U}_{k,k}}=q^{\frac{k(t-k)d_i}{2}} \prod\limits_{j=0}^{k-1} \left(\frac{q^{\frac{(t-j)d_i}{2}}-(-1)^{t-j}} {q^{\frac{(k-j)d_i}{2}}-(-1)^{k-j}}\right).\vspace{-2mm}\end{equation*}

Next let $k$ be even. Here by \cite[p. 116]{tay}, we see that any $k$-dimensional non-degenerate (and hence unitary) $\mathcal{K}_{i}$-subspace of $\mathcal{J}_i$  has a Witt decomposition of the form $ \langle a_1(X),b_1(X)\rangle \perp \cdots \perp \langle a_{\frac{k}{2}}(X),b_{\frac{k}{2}}(X)\rangle,$ where $\big(a_{\ell}(X),b_{\ell}(X)\big)$ is a hyperbolic pair in $\mathcal{J}_i$ for $1 \leq \ell \leq \frac{k}{2}.$ Now working in a similar manner as in the previous case,  the  assertion follows.
\\ \textbf{Case II.} Let $\delta=\gamma.$ By Lemma \ref{non-deg}, we see that $\left[\cdot,\cdot\right]_{\gamma}\restriction_{\mathcal{J}_{i} \times \mathcal{J}_{i}}$ is a reflexive, non-degenerate and skew-Hermitian form. In this case, we will first transform the skew-Hermitian form $\left[\cdot,\cdot\right]_{\gamma} \restriction_{\mathcal{J}_{i}  \times \mathcal{J}_{i}}$ to a Hermitian form  $\left[\cdot,\cdot\right]_{\gamma}^{\prime}$ on $\mathcal{J}_{i}$ as follows:

  By Lemma 11(iii) of Huffman \cite{huff}, we note that $\tau_{1,-1}$ is not the identity map on $\mathcal{K}_i,$ so there exists $\mathfrak{u}(X)\in \mathcal{K}_i$ such that $\tau_{1,-1}\bigl(\mathfrak{u}(X)\bigr)\neq \mathfrak{u}(X).$ Then we define a map $\left[\cdot,\cdot\right]_{\gamma}^{\prime}: \mathcal{J}_{i}\times \mathcal{J}_{i}\rightarrow \mathcal{K}_{i}$  as $\left[c(X),d(X)\right]_{\gamma}^{\prime}= \mathfrak{v}(X)\left[c(X),d(X)\right]_{\gamma}$ for all $c(X),d(X) \in \mathcal{J}_{i},$ where $\mathfrak{v}(X)=\mathfrak{u}(X)-\tau_{1,-1}\bigl(\mathfrak{u}(X)\bigr) \neq 0.$ It is easy to see that $\left[\cdot,\cdot\right]_{\gamma}^{\prime} $ is a  reflexive, non-degenerate and Hermitian form on $\mathcal{J}_i.$  Next we observe that if two vectors in $\mathcal{J}_{i}$ are orthogonal with respect to $ \left[\cdot,\cdot\right]_{\gamma}\restriction_{\mathcal{J}_{i} \times \mathcal{J}_{i}},$ then those two vectors are also  orthogonal with respect to $\left[\cdot,\cdot\right]_{\gamma}'$ and vice versa. This implies that any non-degenerate $\mathcal{K}_{i}$-subspace of $\mathcal{J}_{i}$ with respect to $\left[\cdot,\cdot\right]_{\gamma}\restriction_{\mathcal{J}_{i} \times \mathcal{J}_{i}}$ is also non-degenerate with respect to $\left[\cdot,\cdot\right]_{\gamma}'$ and vice versa.
 So the number $N_i$ equals the number of non-degenerate $\mathcal{K}_{i}$-subspaces of $\mathcal{J}_{i}$ with respect to $\left[\cdot,\cdot\right]_{\gamma}'.$ Now as $(\mathcal{J}_i,\left[\cdot,\cdot \right]_{\gamma}')$  is a unitary space over $\mathcal{K}_i$ having dimension $t,$ working in a similar manner as in case I, the assertion follows. This proves the proposition.
\end{proof}
\vspace{-4mm}\subsection{Determination of the number $N_i$ when $i \in \mathfrak{I}$}\label{0}
\noindent In this section, we will determine the number $N_i$ for each $i \in \mathfrak{I}$ and $\delta \in \{\ast, 0,\gamma\}.$ Here we recall that $\mathcal{K}_{i} \simeq \mathbb{F}_{q},$ as $d_i=1$ for each $i \in \mathfrak{I}.$

In the following proposition, we suppose that either $\delta =\gamma$ or $\delta=\ast$ and $q$ is even, and determine the number $N_i$ of $\mathcal{K}_i$-subspaces $\mathcal{C}_i$ of $\mathcal{J}_i$ satisfying $\mathcal{C}_i \cap \mathcal{C}_i^{(\delta)}=\{0\}$ for each $i \in \mathfrak{I}.$

\begin{prop}\label{symp} Suppose that either $\delta=\gamma$ or $\delta=\ast$ and $q$ is even. For $i \in \mathfrak{I},$ the number $N_i$ of distinct $\mathcal{K}_i$-subspaces $\mathcal{C}_i$ of $\mathcal{J}_i$ satisfying $\mathcal{C}_i \cap \mathcal{C}_i^{(\delta)}=\{0\}$ is given by $N_i=2+\displaystyle \sum\limits_{\stackrel{k=1}{k\equiv 0(\text{mod }2) }}^{t-1}q^{\frac{k(t-k)}{2}}{t/2 \brack k/2}_{q^2} .$
\vspace{-2mm}\end{prop}
\vspace{-2mm}\begin{proof}
 In order to prove this, for each integer $k~(0 \leq k \leq t),$ let $N_{i,k}$ denote the number of $k$-dimensional $\mathcal{K}_i$-subspaces $\mathcal{C}_i$ of  $\mathcal{J}_i$ satisfying $\mathcal{C}_i \cap \mathcal{C}_i^{(\delta)} =\{0\}.$ Then we have $N_i=\sum\limits_{k=0}^{t}N_{i,k}.$

First we observe that  $N_{i,0}=N_{i,t}=1.$  So from now onwards, throughout this proof, we assume that $1 \leq k \leq t-1.$ Here we assert that for $1 \leq k \leq t-1,$ the number $N_{i,k}=0$ when $k$ is odd and  $N_{i,k}=q^{\frac{k(t-k)}{2}}{t/2 \brack k/2}_{q^2}$ when $k$ is even.
To prove this assertion, by Remark \ref{non-deg1}, we note that the number $N_{i,k}$ is equal to the number of  $k$-dimensional non-degenerate $\mathcal{K}_{i}$-subspaces  of $\mathcal{J}_{i}$ for $1 \leq k \leq t-1.$
Now we shall distinguish the cases, I. $\delta=\gamma$ and II. $\delta=\ast$ and $q$ is even.
\\\textbf{Case I.} Let $\delta=\gamma.$ Here by Lemma \ref{non-deg}, we see that $\left[\cdot,\cdot\right]_{\gamma}\restriction_{\mathcal{J}_i \times \mathcal{J}_i}$ is a reflexive, non-degenerate and alternating form, i.e., $(\mathcal{J}_i, \left[\cdot,\cdot\right]_{\gamma}\restriction_{\mathcal{J}_i \times \mathcal{J}_i})$ is a symplectic space over $\mathcal{K}_i.$ In view of this, we see that the number $N_{i,k}$ equals the number of  distinct $k$-dimensional non-degenerate (and hence symplectic) $\mathcal{K}_{i}$-subspaces of $\mathcal{J}_i.$ By \cite[p. 69]{tay}, we see that $N_{i,k}=0$ if $k$ is odd. So we need to determine the number $N_{i,k}$ when $k \geq 2$ is an even integer. For this, by \cite[p. 69]{tay}, we see that when $k$ is even, any $k$-dimensional non-degenerate (and hence symplectic) $\mathcal{K}_i$-subspace of $\mathcal{J}_i$ has a Witt decomposition of the form $\langle a_1(X),b_1(X)\rangle \perp \cdots \perp \langle a_{\frac{k}{2}}(X),b_{\frac{k}{2}}(X)\rangle,$ where each $(a_{\ell}(X),b_{\ell}(X))$ is a hyperbolic pair in $\mathcal{J}_i.$ Further, working in a similar manner as in Proposition \ref{complet} and using \cite[pp. 69-70]{tay}, we get $N_{i,k}=q^{\frac{k(t-k)}{2}}{t/2 \brack k/2}_{q^2}$ for each even integer $k.$
\\\textbf{Case II.} Let $\delta=\ast$ and $q$ be even. In this case,  $t$ is even and $n$ is odd, and hence $i=0$ only. By Huffman \cite[p. 264]{huff}, we see that  $\mathcal{J}_0 = \{ aM(X): a \in \mathbb{F}_{q^t} \} \simeq \mathbb{F}_{q^t}$ and $ \mathcal{K}_0 = \{ rM(X): r \in \mathbb{F}_{q} \} \simeq \mathbb{F}_{q},$  where $M(X)=1+X+X^2+\cdots+X^{n-1}.$ We next observe that $ \left[aM(X),bM(X) \right]_{\ast}= n\text{Tr}_{q,t}\bigl(a \phi(b)\bigr) M(X)=n(a,b)_{\ast}M(X)$ for all $aM(X),bM(X) \in \mathcal{J}_0.$ Since $M(X)\neq 0,$ we observe that if two vectors $aM(X),b M(X) \in \mathcal{J}_{0}$ are orthogonal with respect to $\left[\cdot,\cdot\right]_{\ast}\restriction_{\mathcal{J}_0\times \mathcal{J}_0},$ then the corresponding vectors $a,b \in \mathbb{F}_{q^t}$ are orthogonal with respect to $(\cdot,\cdot)_{\ast}$  on $\mathbb{F}_{q^t}$ and vice versa. So the number of distinct $k$-dimensional non-degenerate $\mathcal{K}_0$-subspaces of $\mathcal{J}_0$ with respect to $\left[\cdot,\cdot \right]_{\ast}\restriction_{\mathcal{J}_0\times \mathcal{J}_0}$ is equal to the number of distinct $k$-dimensional non-degenerate $\mathbb{F}_q$-subspaces of $\mathbb{F}_{q^t}$ with respect to $(\cdot,\cdot)_{\ast}$ on $\mathbb{F}_{q^t}.$ Further, by Lemma 3.2 of Sharma and Kaur \cite{sh}, we see that when $q$ is even, $( \cdot , \cdot)_{\ast} $ is a reflexive, non-degenerate and alternating form on  $ \mathbb{F}_{q^t},$ i.e., $(\mathbb{F}_{q^t},(\cdot,\cdot)_{\ast})$  is a symplectic space over $\mathbb{F}_q$ having dimension $t.$ Now working in a similar manner as in case I, we obtain $N_{0,k}=q^{\frac{k(t-k)}{2}}{t/2 \brack k/2}_{q^2}$ when $k$ is even and $N_{0,k}=0$ when $k$ is odd. This proves the proposition.
\end{proof}
In the following proposition, we suppose that $\delta \in \{\ast,0\}$ and $q$ is odd, and determine the number of $\mathcal{K}_i$-subspaces $\mathcal{C}_i$ of $\mathcal{J}_i$ satisfying $\mathcal{C}_i \cap \mathcal{C}_i^{(\delta)}=\{0\}$ for each $i \in \mathfrak{I}.$
\begin{prop} \label{div} Let  $\delta \in \{\ast,0\}$ and $q$ be an odd prime power. For $i \in \mathfrak{I},$ the number $N_i$ of distinct $\mathcal{K}_i$-subspaces $\mathcal{C}_i$ of $\mathcal{J}_i$ satisfying $\mathcal{C}_i \cap \mathcal{C}_i^{(\delta)}=\{0\}$ is given by the following:
\begin{enumerate}
\vspace{-2mm}\item[(a)] $N_i =\displaystyle 2+ \sum\limits_{\stackrel{k=1}{k\equiv 0(\text{mod }2) } }^{t-1} q^{\frac{k(t-k+1)}{2}}{(t-1)/2 \brack k/2}_{q^2}+\sum\limits_{\stackrel{k=1}{k\equiv 1(\text{mod }2) } }^{t-1}q^{\frac{(t-k)(k+1)}{2}}  { (t-1)/2  \brack  (k-1)/2 }_{q^2}$ if $t$ is odd.
\vspace{-2mm}\item[(b)] $N_i =\displaystyle 2+\sum\limits_{\stackrel{k=1}{k\equiv 0(\text{mod }2) }}^{t-1} q^{\frac{k(t-k)}{2}} {t/2 \brack k/2}_{q^2} + \sum\limits_{\stackrel{k=1}{k\equiv 1(\text{mod }2) }}^{t-1} q^{\frac{tk-k^2-1}{2}}(q^{\frac{t}{2}}+1) {(t-2)/2 \brack (k-1)/2}_{q^2}\vspace{2mm}$ if either $t$ is even and $q\equiv 1(\text{mod }4)$ or $t \equiv 0(\text{mod }4)$ and $q \equiv 3(\text{mod }4).$
\vspace{-2mm}\item[(c)] $N_i= \displaystyle 2+\sum\limits_{\stackrel{k=1}{k\equiv 0(\text{mod }2) }}^{t-1} q^{\frac{k(t-k)}{2}} {t/2 \brack k/2}_{q^2} + \sum\limits_{\stackrel{k=1}{k\equiv 1(\text{mod }2) }}^{t-1} q^{\frac{tk-k^2-1}{2}} (q^{\frac{t}{2}}-1){(t-2)/2 \brack (k-1)/2}_{q^2}\vspace{2mm}$ if $t \equiv 2(\text{mod }4)$ and $q \equiv 3(\text{mod }4).$
    \end{enumerate}
    \end{prop}
To prove this proposition, we see, by  Lemma \ref{non-deg}, that when $\delta \in \{\ast,0\}$ and $q$ is an odd prime power,   $\left[\cdot, \cdot\right]_{\delta}\restriction_{\mathcal{J}_{i} \times \mathcal{J}_{i}}$ is a reflexive, non-degenerate and symmetric  form, i.e., $\left(\mathcal{J}_i, \left[\cdot, \cdot\right]_{\delta}\restriction_{\mathcal{J}_{i} \times \mathcal{J}_{i}}\right)$ is an orthogonal space over $\mathcal{K}_{i}.$ Furthermore, since $q$ is odd, the map $Q_i: \mathcal{J}_{i} \rightarrow \mathcal{K}_{i},$ defined by $Q_i(u(X))=\frac{1}{2}\left[u(X),u(X)\right]_{\delta}$ for each $u(X) \in \mathcal{J}_{i},$ is a quadratic map, and  hence $\left(\mathcal{J}_i,Q_i\right)$ is a non-degenerate quadratic space over $\mathcal{K}_{i}.$ Now to prove Proposition \ref{div}, we shall first prove the following lemma:
\begin{lem} \label{isometry}
Let $\delta \in \{\ast,0\},$ $q$ be an odd prime power and $k\geq 1$ be an odd integer. Let $U$ be a $k$-dimensional non-degenerate quadratic $\mathcal{K}_i$-subspace of $\mathcal{J}_{i}$ having a Witt decomposition of the form $\langle m_1(X),n_1(X)\rangle \perp  \langle m_2(X),n_2(X)\rangle \perp\cdots \perp \langle m_{\frac{k-1}{2}}(X),n_{\frac{k-1}{2}}(X)\rangle \perp \langle r(X)\rangle,$ where $(m_{\ell}(X),n_{\ell}(X))$ is a hyperbolic pair in $U$ for $1 \leq \ell \leq \frac{k-1}{2}$ and $r(X)$ is a non-singular vector in $U.$ Then the following hold.
\begin{itemize}
\vspace{-2mm}\item[(a)] The number $\mathsf{L}_k$ of non-singular vectors $u(X) \in U$ with $\left<u(X)\right>$ isometric to  $\langle r(X) \rangle$ is given by $\mathsf{L}_k=\frac{q^{\frac{k-1}{2}}(q-1) (q^{\frac{k-1}{2}}+1)}{2}.$
\vspace{-2mm}\item[(b)] The number $\mathsf{M}_k$ of non-singular vectors $u(X) \in U$ with $\left<u(X)\right>$  not isometric to $\langle r(X) \rangle$ is given by $\mathsf{M}_k=\frac{q^{\frac{k-1}{2}} (q-1)(q^{\frac{k-1}{2}}-1)}{2}.$
\end{itemize}
\end{lem}
\begin{proof}
\begin{itemize}
\vspace{-2mm}\item[(a)] First of all, we observe that for each odd integer $k \geq 1,$ the number $\mathsf{L}_{k}$ is equal to the number of pairs $\left(u(X),\theta(X)\right),$ where $u(X)$ is a non-singular vector in $U$ and $\theta(X)$ is a non-zero square in $\mathcal{K}_{i}$ satisfying $\left[u(X),u(X)\right]_{\delta}=\theta(X)\left[r(X),r(X)\right]_{\delta}.$  Now to determine the number $\mathsf{L}_{k},$ we will apply induction on an odd integer $k \geq 1.$

When $k=1,$ we have $U=\left<r(X)\right>.$ From this,  it follows trivially that $\mathsf{L}_1=q-1,$ i.e., the result holds for $k=1.$
Next  we assume that $k\geq 3$ is an odd integer and the result holds for all $(k-2)$-dimensional non-degenerate quadratic $\mathcal{K}_{i}$-subspaces of $\mathcal{J}_{i}.$ Now let $U$ be a $k$-dimensional non-degenerate quadratic $\mathcal{K}_{i}$-subspace of $\mathcal{J}_{i}.$  As $k \geq 3,$ by Lemma 7.3 and Theorem 11.2 of \cite{tay}, there exists a hyperbolic pair in $U,$ say $(e(X),f(X)).$ Further by Proposition 2.9 of \cite{grove}, we write $U =\langle e(X),f(X)\rangle \perp  \langle e(X),f(X)\rangle^{\perp_{\delta}},$ where $\langle e(X),f(X)\rangle^{\perp_{\delta}}$ is a $(k-2)$-dimensional non-degenerate  quadratic $\mathcal{K}_i$-subspace of $U.$ In view of this, any element $u(X)\in U$ can be written as $u(X)=\lambda(X)e(X)+\kappa(X)f(X)+w(X),$ where $\lambda(X),\kappa(X) \in \mathcal{K}_i$ and $w(X)\in \langle e(X),f(X)\rangle^{\perp_{\delta}}.$ Now we shall count all pairs $(u(X),\theta(X))$ with $u(X)$ as a non-singular vector in $U$ and $\theta(X)$ as a non-zero square in $\mathcal{K}_{i}$ satisfying $\left[u(X),u(X)\right]_{\delta}=\theta(X)\left[r(X),r(X)\right]_{\delta}.$ For this, we see that $\left[u(X),u(X)\right]_{\delta}=\theta(X)\left[r(X),r(X)\right]_{\delta}$  if and only if \vspace{-2mm}\begin{equation}\label{p2}2\lambda(X)\kappa(X)+\left[w(X),w(X)\right]_{\delta}=\theta(X)\left[r(X),r(X)\right]_{\delta}.\vspace{-2mm}\end{equation}  Here also, we shall distinguish  the following two cases: \textbf{I.} $\lambda(X)\kappa(X)=0$ and \textbf{II.} $\lambda(X)\kappa(X)\neq 0.$
\\\textbf{I.} Let $\lambda(X)\kappa(X)=0.$ Then  by \eqref{p2}, we have $\left[w(X),w(X)\right]_{\delta}=\theta(X)\left[r(X),r(X)\right]_{\delta},$ that is, the $\mathcal{K}_i$-subspaces $\langle w(X)\rangle$ and $\langle r(X)\rangle $ are isometric. As $w(X) \in \left<e(X),f(X)\right>^{\perp},$  by applying the induction hypothesis, we see that the pair $(w(X),\theta(X))$ has precisely $\mathsf{L}_{k-2}$ relevant choices. Moreover, the pair $(\lambda(X),\kappa(X) )$ has $2q-1$ choices. Thus in this case, there are  precisely $(2q-1)\mathsf{L}_{k-2}$ number of pairs $(u(X),\theta(X)),$ where $u(X)$ is a non-singular vector of $U$ and $\theta(X)$ is a non-zero square in $\mathcal{K}_{i}$ satisfying $\left[u(X),u(X)\right]_{\delta}=\theta(X)\left[r(X),r(X)\right]_{\delta}.$
\\\textbf{II.} Next let $\lambda(X)\kappa(X)\neq0.$ Then by \eqref{p2}, we get $\left[w(X),w(X)\right]_{\delta}\neq \theta(X)\left[r(X),r(X)\right]_{\delta},$ that is, the $\mathcal{K}_i$-subspace $\langle w(X)\rangle$ is not isometric to $\langle r(X)\rangle.$ By applying the induction hypothesis, we see that the number of pairs $(w(X),\theta(X))$ with $\theta(X)$ as a non-zero square in $\mathcal{K}_{i}$ and $w(X) \in \left<e(X),f(X)\right>^{\perp}$ satisfying $\left[w(X),w(X)\right]_{\delta}\neq \theta(X)\left[r(X),r(X)\right]_{\delta}$ is given by $\frac{q^{k-2}(q-1)}{2}-\mathsf{L}_{k-2}.$  In view of \eqref{p2},  we see that $\lambda(X)$ is fixed for a given choice of $\theta(X),w(X)$ and $\kappa(X).$ Also as $\kappa(X)\neq 0,$ it has $q-1$ choices. Thus in this case, there are precisely $(q-1)\bigg(\frac{q^{k-2}(q-1)}{2}-\mathsf{L}_{k-2}\bigg)$ number of pairs $(u(X),\theta(X))$ with $u(X)$ as a non-singular vector of $U$ and $\theta(X)$ as a non-zero square in $\mathcal{K}_{i}$ satisfying $\left[u(X),u(X)\right]_{\delta}=\theta(X)\left[r(X),r(X)\right]_{\delta}.$\\
Therefore from the cases \textbf{I} and \textbf{II} above, we see that the number $\mathsf{L}_{k}$ of pairs $(u(X),\theta(X))$ with $u(X)$ as a non-singular vector in $U$ and $\theta(X)$ as a non-zero square in $\mathcal{K}_{i}$ satisfying $\left[u(X),u(X)\right]_{\delta}=\theta(X)\left[r(X),r(X)\right]_{\delta},$ is given by $\mathsf{L}_{k}=(2q-1)\mathsf{L}_{k-2}+(q-1)\bigg(\frac{q^{k-2} (q-1)}{2}-\mathsf{L}_{k-2}\bigg).$ Further, after a simple computation and on substituting the value of $\mathsf{L}_{1}=q-1,$ we obtain $ \mathsf{L}_{k}=\frac{q^{\frac{k-1}{2}}(q-1)(q^{\frac{k-1}{2}}+1)}{2},$  which proves (a).
\vspace{-2mm}\item[(b)] To prove this, we observe that  for each odd integer $k \geq 1,$ the number $\mathsf{L}_{k}+\mathsf{M}_{k}$ is equal to the number of non-singular vectors in $U.$ Furthermore,  by Theorem 11.5 of \cite{tay}, we see that the number of non-singular vectors in $U$ is given by $q^k-1-I_{\frac{k-1}{2},1}=q^k-q^{k-1},$ where $I_{\frac{k-1}{2},1}$ denotes the number of isotropic vectors in $U$ having Witt index $\frac{k-1}{2}.$ From this and using part (a), we obtain $\mathsf{M}_{k}=q^k-q^{k-1}-\mathsf{L}_{k}= \frac{q^{\frac{k-1}{2}}(q-1)(q^{\frac{k-1}{2}}-1)}{2},$ which proves (b).
\end{itemize}
\vspace{-4mm}\end{proof}
\noindent\textbf{Proof of Proposition \ref{div}.}
To prove this, for each integer $k~(0 \leq k \leq t),$ let $N_{i,k}$ denote the number of $k$-dimensional $\mathcal{K}_i$-subspaces $\mathcal{C}_i$ of  $\mathcal{J}_i$ satisfying $\mathcal{C}_i \cap \mathcal{C}_i^{(\delta)} =\{0\}.$ Then we have \vspace{-3mm}\begin{equation}\label{1} \hspace{-4mm} N_i=\sum\limits_{k=0}^{t}N_{i,k}.\vspace{-3mm}\end{equation}
Now we proceed to determine the number $N_{i,k}$ for each integer $k,$ $0 \leq k \leq t.$ Towards this, it is easy to observe that $N_{i,0}=N_{i,t}=1.$  So we assume $1\leq k \leq t-1$ from now onwards. Here by Remark \ref{non-deg1}, we note that the number $N_{i,k}$ is equal to the number of  $k$-dimensional non-degenerate $\mathcal{K}_{i}$-subspaces of $\mathcal{J}_{i}$ for $1 \leq k \leq t-1.$ Besides this, we recall that $\left(\mathcal{J}_i,Q_i\right)$ is a non-degenerate quadratic space over $\mathcal{K}_{i} \simeq \mathbb{F}_{q},$ where the quadratic map $Q_i: \mathcal{J}_{i} \rightarrow \mathcal{K}_{i}$ is defined as $Q_i(u(X))=\frac{1}{2}\left[u(X),u(X)\right]_{\delta}$ for each $u(X) \in \mathcal{J}_{i}.$
Then we observe the following: \begin{enumerate} \vspace{-2mm}\item[($\square$)] The discriminant of every 2-dimensional anisotropic $\mathcal{K}_{i}$-subspace  of $\mathcal{J}_{i}$ is a non-square in $\mathcal{K}_{i}.$ As a consequence, every 2-dimensional anisotropic $\mathcal{K}_{i}$-subspace of $\mathcal{J}_{i}$ is  non-degenerate.
\vspace{-2mm}\item[($\blacksquare$)] Every 2-dimensional anisotropic $\mathcal{K}_{i}$-subspace of $\mathcal{J}_{i}$ has an orthogonal basis, i.e., it has a basis of the type $\{u(X),v(X)\}$ with $\left[u(X),v(X)\right]_{\delta}=0.$ \vspace{-1mm}
\end{enumerate}
Furthermore, by Proposition 1 of Drees et al. \cite{dree}, we see that the determinant of $\left(\mathcal{J}_{i},Q_i\right)$ is a non-square in $\mathcal{K}_{i}.$ From this, by \cite[p. 138]{tay} and assertion ($\square$), it is easy to observe (see Huffman \cite[p. 279]{huff}) that the Witt index $\mathfrak{w}$ of $\left(\mathcal{J}_i,Q_i\right)$ is given by
\vspace{-2mm} \begin{equation} \label{2}\mathfrak{w}=\left\{\begin{array}{cl}  \frac{t-1}{2} & \text{if }t \text{ is odd}; \\ \frac{t-2}{2} &\text{if either } t \text{ is even and }q \equiv 1~(\text{mod }4) \text{ or } t \equiv 0~(\text{mod }4) \text{ and }  q \equiv 3~(\text{mod }4);\\ \frac{t}{2} & \text{if } t \equiv 2~(\text{mod }4) \text{ and } q \equiv 3~(\text{mod }4).\end{array}\right. \vspace{-2mm}\end{equation}

Now in order to determine the number $N_{i,k},$  by \cite[p. 138]{tay}, we see that any $k$-dimensional non-degenerate quadratic $\mathcal{K}_i$-subspace of $\mathcal{J}_i$ has a Witt decomposition of the form $\langle a_1(X),b_1(X)\rangle \perp \cdots \perp \langle a_{\mathfrak{m}_k}(X),b_{\mathfrak{m}_k}(X)\rangle \perp W_k,$ where each $(a_{\ell}(X),b_{\ell}(X))$ is a hyperbolic pair in $\mathcal{J}_i,$  $\mathfrak{m}_k$ is its Witt index and $W_k$ is an anisotropic $\mathcal{K}_i$-subspace of $\mathcal{J}_i$ with $\text{dim}_{\mathcal{K}_{i}}W_k =k-2\mathfrak{m}_{k}\leq 2.$ By \cite[p. 138]{tay}, we also see that $\mathfrak{m}_k=\frac{k-1}{2}$ when $k$ is odd and  $\mathfrak{m}_k$ is either $\frac{k-2}{2}$ or $\frac{k}{2}$  (depending upon the isometry class of $\mathcal{J}_{i})$ when $k$ is even.  We shall first determine the number $\mathfrak{Q}_{\mathfrak{m}_k,\mathfrak{w}}$ of Witt bases of the type $\{a_1(X),b_1(X),  \cdots, a_{\mathfrak{m}_k}(X),b_{\mathfrak{m}_k}(X)\} \cup \beta_{W_k}$ with  $\beta_{W_k}$ as a basis of the anisotropic $\mathcal{K}_i$-subspace $W_k$ of $\mathcal{J}_i,$ where the basis $\beta_{W_k}$ is orthogonal when $\text{dim}_{\mathcal{K}_{i}}W_k=2$ (or equivalently, $\mathfrak{m}_k=\frac{k-2}{2}$).
For this, we shall distinguish the following two cases:  \textbf{I. } $k$ is odd and \textbf{II. } $k$ is even.
\\\textbf{I.} Let $k$ be odd. Here, by \cite[p. 138]{tay}, we have $\mathfrak{m}_k=\frac{k-1}{2},$ which implies that  $\text{dim}_{\mathcal{K}_i}W_k=1.$ So in this case, any $k$-dimensional non-degenerate quadratic $\mathcal{K}_{i}$-subspace of $\mathcal{J}_{i}$ has a Witt decomposition of the form $\langle a_1(X),b_1(X)\rangle \perp \cdots \perp \langle a_{\frac{k-1}{2}}(X), b_{\frac{k-1}{2}}(X)\rangle \perp \langle w(X) \rangle,$ where $(a_{\ell}(X),b_{\ell}(X))$ is a hyperbolic pair in $\mathcal{J}_i$ for $1 \leq \ell \leq \frac{k-1}{2}$ and $w(X)$ is a non-singular vector of $\mathcal{J}_i.$  We first note that as $(a_1(X),b_1(X))$ is a hyperbolic pair in $\mathcal{J}_i,$ by \cite[p. 141]{tay}, it has $H_{\mathfrak{w},t-2\mathfrak{w}}$ choices, where $\mathfrak{w}$ is the Witt index of $\mathcal{J}_i.$ Further by Proposition 2.9 of \cite{grove}, we write $\mathcal{J}_i=\langle a_1(X),b_1(X)\rangle \perp \langle a_1(X),b_1(X)\rangle^{\perp_{\delta}},$ where  $\langle a_1(X),b_1(X)\rangle^{\perp_{\delta}}$ is a $(t-2)$-dimensional non-degenerate quadratic  $\mathcal{K}_i$-subspace of $\mathcal{J}_i.$ Now we choose the second hyperbolic pair $(a_2(X),b_2(X))$ from the quadratic space  $\langle a_1(X),b_1(X)\rangle^{\perp_{\delta}}.$ By applying Witt's cancellation theorem, we see that the space $\langle a_1(X),b_1(X)\rangle^{\perp_{\delta}}$ has Witt index $\mathfrak{w}-1.$ By  \cite[p. 141]{tay}, again, we see that the second hyperbolic pair $(a_2(X),b_2(X))$ has  $H_{\mathfrak{w}-1,t-2\mathfrak{w}}$ choices.
Continuing like this, we see that  the $(\frac{k-1}{2})$th hyperbolic pair  $(a_{\frac{k-1}{2}}(X),b_{\frac{k-1}{2}}(X))$ can be chosen in $H_{\mathfrak{w}-\frac{(k-3)}{2}, t-2\mathfrak{w}}$ number of ways.
Next to count all possible choices for the basis $\beta_{W_k}$ of a 1-dimensional anisotropic $\mathcal{K}_{i}$-subspace $W_k$ of $\langle a_1(X),b_1(X), \cdots, a_{\frac{k-1}{2}}(X),b_{\frac{k-1}{2}}(X) \rangle^{\perp_{\delta}},$ we need to choose a non-singular vector $w(X)$ from the $(t-k+1)$-dimensional non-degenerate quadratic $\mathcal{K}_{i}$-subspace $ \langle a_1(X), b_1(X), \cdots, a_{\frac{k-1}{2}}(X),b_{\frac{k-1}{2}}(X) \rangle^{\perp_{\delta}}$ of $\mathcal{J}_{i}.$  By applying Witt's cancellation theorem, we see that the Witt index of $\langle a_1(X), b_1(X), \hspace{-0.2mm}\cdots\hspace{-0.3mm},\hspace{-0.2mm} a_{\frac{k-1}{2}}(X), b_{\frac{k-1}{2}}(X) \rangle^{\perp_{\delta}}$ is $\mathfrak{w}-\frac{(k-1)}{2}.$ Now by Theorem 11.5 of \cite{tay}, we see that the number of non-singular vectors in $\langle a_1(X), b_1(X), \cdots, a_{\frac{k-1}{2}}(X),b_{\frac{k-1}{2}}(X) \rangle^{\perp_{\delta}}$ is given by $q^{t-k+1}-1-I_{\mathfrak{w}-\frac{(k-1)}{2},t-k+1-2(\mathfrak{w}-\frac{(k-1)}{2})}.$
From this, it follows that when $k$ is odd, the number of Witt bases of the form $\{a_1(X),b_1(X), \cdots, a_{\frac{k-1}{2}}(X),b_{\frac{k-1}{2}}(X)\}\cup\beta_{W_k}$ and having cardinality $k$ in $\mathcal{J}_i$ is given by \vspace{-3mm}\begin{equation}\label{5}\displaystyle \mathfrak{Q}_{\frac{k-1}{2},\mathfrak{w}}=H_{\mathfrak{w},t-2\mathfrak{w}} H_{\mathfrak{w}-1,t-2\mathfrak{w}} \cdots H_{\mathfrak{w}-\frac{(k-1)}{2}+1,t-2\mathfrak{w}}\left(q^{t-k+1}-1-I_{\mathfrak{w}-\frac{(k-1)}{2},t-2\mathfrak{w}}\right). \vspace{-1mm}\end{equation}
This, by \cite[pp. 140-141]{tay}, gives
\vspace{-2mm}\begin{equation}\label{Q1}
\hspace{-2mm}\mathfrak{Q}_{\frac{k-1}{2},\mathfrak{w}}=\left\{
  \begin{array}{ll}
    q^{\frac{2tk-(k+1)^2}{4}}(q^{\frac{t}{2}}-1) (q-1)\prod\limits_{j=1}^{(k-1)/2}(q^{t-2j}-1) & \hbox{if $\mathfrak{w}=\frac{t}{2}$ and $k$ is odd; } \\
    q^{\frac{2t(k+1)-k(k+4)+1}{4}}(q-1)\prod\limits_{j=0}^{(k-3)/2}(q^{t-2j-1}-1) & \hbox{if  $\mathfrak{w}=\frac{t-1}{2}$ and $k$ is odd;} \\
    q^{\frac{2tk-(k+1)^2}{4}}(q^{\frac{t}{2}}+1)(q-1)\prod\limits_{j=1}^{(k-1)/2}(q^{t-2j}-1) & \hbox{ if $\mathfrak{w}=\frac{t-2}{2}$ and $k$ is odd}.
   \end{array}
   \right.
   \vspace{-2mm} \end{equation}
    Next when $k$ is odd, let $\mathfrak{Q}_{\frac{k-1}{2}}$ denote the number of Witt bases of a $k$-dimensional non-degenerate quadratic $\mathcal{K}_i$-subspace of $\mathcal{J}_{i}.$  Here  working in a similar manner as above and by \cite[pp. 138-141]{tay}, we obtain
 \vspace{-3.4mm}\begin{equation}\small\label{6}
    \mathfrak{Q}_{\frac{k-1}{2}}=H_{\frac{k-1}{2},1}H_{\frac{k-3}{2},1}\cdots H_{1,1} (q-1)=q^{\frac{(k-1)^2}{4}}(q-1)\prod\limits_{j=0}^{(k-3)/2}(q^{k-2j-1}-1).
 \vspace{-1mm}\end{equation}
Therefore when $k$ is odd,   using \eqref{Q1} and \eqref{6},  we obtain
 \vspace{-2mm}\begin{equation}\label{n1}\displaystyle
N_{i,k}=\frac{\mathfrak{Q}_{\frac{k-1}{2},\mathfrak{w}}}{\mathfrak{Q}_{\frac{k-1}{2}}}=  \left\{
  \begin{array}{ll}
  \displaystyle  q^{\frac{tk-k^2-1}{2}} (q^{\frac{t}{2}}-1){(t-2)/2 \brack (k-1)/2}_{q^2} &  \hbox{if $\mathfrak{w}=\frac{t}{2}$ and $k$ is odd;} \vspace{1mm}\\
   \displaystyle q^{\frac{(t-k)(k+1)}{2}} {(t-1)/2 \brack (k-1)/2}_{q^2} & \hbox{if $\mathfrak{w}=\frac{t-1}{2}$ and $k$ is odd;}\vspace{1mm}\\

    \displaystyle q^{\frac{tk-k^2-1}{2}}(q^{\frac{t}{2}}+1) {(t-2)/2 \brack (k-1)/2}_{q^2} & \hbox{if $\mathfrak{w}=\frac{t-2}{2}$ and $k$ is odd.}

  \end{array}
   \right.
    \vspace{-2mm}\end{equation}
\textbf{II.} Next let $k$ be even. Here by \cite[p. 138]{tay}, we see that $\mathfrak{m}_k$ is either $\frac{k-2}{2}$ or $\frac{k}{2}.$ For each even integer $k~(1\leq k \leq t-1),$ let $R_{i,k}$ denote the number of $k$-dimensional non-degenerate quadratic $\mathcal{K}_{i}$-subspaces of $\mathcal{J}_{i}$ having Witt index as $\mathfrak{m}_{k}=\frac{k-2}{2}$ and $S_{i,k}$ denote the number of $k$-dimensional non-degenerate quadratic $\mathcal{K}_i$-subspaces  of  $\mathcal{J}_i$ having Witt index as $\mathfrak{m}_k=\frac{k}{2}.$ In view of this, we have
 \vspace{-2mm}\begin{equation}\label{r} N_{i,k} =R_{i,k}+S_{i,k}.\vspace{-2mm} \end{equation}
Now we shall distinguish the following two cases: \textbf{A. } $\mathfrak{m}_k=\frac{k-2}{2}$ and \textbf{B. }  $\mathfrak{m}_k=\frac{k}{2}.$
\\\textbf{A.} First let $\mathfrak{m}_k=\frac{k-2}{2}.$ In this case, by \cite[p. 138]{tay}, we see that any $k$-dimensional non-degenerate quadratic $\mathcal{K}_{i}$-subspace of $\mathcal{J}_{i}$ has a  Witt decomposition of the form $\langle a_1(X),b_1(X)\rangle \perp \cdots \perp \langle a_{\frac{k-2}{2}}(X),b_{\frac{k-2}{2}}(X)\rangle \perp  W_k ,$ where each $(a_{\ell}(X),b_{\ell}(X))$ is a hyperbolic pair in $\mathcal{J}_i$ and $W_k$ is  a 2-dimensional anisotropic $\mathcal{K}_i$-subspace of $\mathcal{J}_i.$
We note that as $(a_1(X),b_1(X))$ is a hyperbolic pair in $\mathcal{J}_i,$ by \cite[p. 141]{tay}, it has $H_{\mathfrak{w},t-2\mathfrak{w}}$ choices, where $\mathfrak{w}$ is the Witt index of $\mathcal{J}_i.$ Further by Proposition 2.9 of \cite{grove}, we write $\mathcal{J}_i=\langle a_1(X),b_1(X)\rangle \perp \langle a_1(X),b_1(X)\rangle^{\perp_{\delta}},$ where  $\langle a_1(X),b_1(X)\rangle^{\perp_{\delta}}$ is a $(t-2)$-dimensional non-degenerate quadratic  $\mathcal{K}_i$-subspace of $\mathcal{J}_i.$ 
Continuing like this, we see that  the $(\frac{k-2}{2})$th hyperbolic pair  $(a_{\frac{k-2}{2}}(X),b_{\frac{k-2}{2}}(X))$ can be chosen in $H_{\mathfrak{w}-\frac{(k-4)}{2}, t-2\mathfrak{w}}$ number of ways. In order to determine the number $\mathfrak{Q}_{\frac{k-2}{2},\mathfrak{w}},$ we need to determine the number $\mathbb{A}_{k,\mathfrak{w}}$ of orthogonal bases of all 2-dimensional anisotropic $\mathcal{K}_{i}$-subspaces of the $(t-k+2)$-dimensional non-degenerate quadratic $\mathcal{K}_{i}$-subspace $\langle a_1(X),b_1(X), \cdots, a_{\frac{k-2}{2}}(X),b_{\frac{k-2}{2}}(X) \rangle^{\perp_{\delta}}$ of $\mathcal{J}_{i}.$
For this, by applying Witt's cancellation theorem, we see that the space $\langle a_1(X), b_1(X), \cdots, a_{\frac{k-2}{2}}(X),b_{\frac{k-2}{2}}(X) \rangle^{\perp_{\delta}}$ has Witt index as $\mathfrak{w}-\frac{(k-2)}{2}.$ Now we assert that
\vspace{-2mm}\begin{equation}\label{9}
\hspace{-1mm}\mathbb{A}_{k,\mathfrak{w}}=\left\{
  \begin{array}{ll}
    \frac{q^{t-k}(q-1)^2(q^{\frac{t-k}{2}}-1)(q^{\frac{t-k+2}{2}}-1)}{2} & \hbox{if $\mathfrak{w}=\frac{t}{2};$ }\vspace{1mm} \\
    \frac{q^{t-k}(q-1)^2(q^{t-k+1}-1)}{2} &  \hbox{if $\mathfrak{w}=\frac{t-1}{2};$} \vspace{1mm}\\
     \frac{q^{t-k}(q-1)^2(q^{\frac{t-k}{2}}+1)(q^{\frac{t-k+2}{2}}+1)}{2} & \hbox{ if $\mathfrak{w}=\frac{t-2}{2}$}.
   \end{array}
   \right.
    \vspace{-2mm}\end{equation}

To prove assertion \eqref{9}, we will consider the following three cases separately: \textbf{(i)} $\mathfrak{w}=\frac{t}{2},$ \textbf{(ii)} $\mathfrak{w}=\frac{t-1}{2}$ and \textbf{(iii)} $\mathfrak{w}=\frac{t-2}{2}.$ \\
\textbf{(i)} First let $\mathfrak{w}=\frac{t}{2}.$ Here we assert the following:
\begin{enumerate}
\vspace{-2mm}\item[(a)] The number of orthogonal bases of the type $\{u(X),v(X)\}$ with $u(X),v(X)$ as non-singular vectors of  $\langle a_1(X), b_1(X), \cdots, a_{\frac{k-2}{2}}(X),b_{\frac{k-2}{2}}(X) \rangle^{\perp_{\delta}},$ is given by $q^{\frac{3(t-k)}{2}}(q-1)^2(q^{\frac{t-k+2}{2}}-1).$
\vspace{-2mm}\item[(b)] The number of orthogonal bases of the type $\{u(X),v(X)\}$ with $u(X),v(X)$ as non-singular vectors of  $\langle a_1(X), b_1(X), \cdots, a_{\frac{k-2}{2}}(X),b_{\frac{k-2}{2}}(X) \rangle^{\perp_{\delta}}$ and $\left<u(X),v(X)\right>$ containing a singular vector, is given by $\frac{q^{t-k}(q-1)^2(q^{\frac{t-k+2}{2}}-1) (q^{\frac{t-k}{2}}+1)}{2}.$
    \end{enumerate}

To prove (a), we need to count orthogonal bases  of the type $\{u(X),v(X)\}$ with $u(X),v(X)$ as  non-singular vectors of $\langle a_1(X), b_1(X),\cdots,a_{\frac{k-2}{2}}(X),b_{\frac{k-2}{2}}(X) \rangle^{\perp_{\delta}}.$  For this, we see that as  $\mathfrak{w}=\frac{t}{2},$ by Witt's cancellation theorem, the Witt index of $\langle a_1(X), b_1(X), \cdots, a_{\frac{k-2}{2}}(X),b_{\frac{k-2}{2}}(X) \rangle^{\perp_{\delta}}$ is $\frac{t}{2}-\frac{(k-2)}{2}=\frac{t-k+2}{2}.$ Now as  $u(X)$ is a non-singular vector  in  $\langle a_1(X), b_1(X), \cdots, a_{\frac{k-2}{2}}(X), b_{\frac{k-2}{2}}(X) \rangle^{\perp_{\delta}},$  by Theorem 11.5 of \cite{tay}, we see that $u(X)$  has $q^{t-k+2}- 1-I_{\frac{(t-k+2)}{2},0}$ choices. Further, by applying Proposition 2.9 of \cite{grove}, we write $\langle a_1(X), b_1(X), \cdots, a_{\frac{k-2}{2}}(X), b_{\frac{k-2}{2}}(X) \rangle^{\perp_{\delta}} = \langle u(X)\rangle \perp \langle u(X)\rangle^{\perp_{\delta}},$ where $\langle u(X) \rangle^{\perp_{\delta}}$ is a $(t-k+1)$-dimensional non-degenerate quadratic $\mathcal{K}_{i}$-subspace of $\langle a_1(X), b_1(X), \cdots, a_{\frac{k-2}{2}}(X), b_{\frac{k-2}{2}}(X) \rangle^{\perp_{\delta}}.$  Now we will  choose another non-singular vector $v(X)$ from the space $\langle u(X) \rangle^{\perp_{\delta}}.$ As $u(X)$ is non-singular, we see that $u(X) \not\in \langle u(X)\rangle^{\perp_{\delta}}.$ Moreover, since $t-k+1$ is odd, by \cite[p. 138]{tay}, we see that the Witt index of $\langle u(X)\rangle^{\perp_{\delta}}$ is $\frac{t-k}{2}.$ Then we see that $v(X)$ has $ q^{t-k+1}-1-I_{\frac{t-k}{2},1}$ choices. From this and using Theorem 11.5 of \cite{tay}, we see that the number of choices for the basis $\{u(X),v(X)\}$ with $u(X),v(X)$ as  non-singular vectors of $\langle a_1(X), b_1(X),\cdots,a_{\frac{k-2}{2}}(X),b_{\frac{k-2}{2}}(X) \rangle^{\perp_{\delta}}$ and satisfying $\left[u(X),v(X)\right]_{\delta}=0,$ is given by $(q^{t-k+2}-1- I_{\frac{(t-k+2)}{2},0})(q^{t-k+1}-1-I_{\frac{(t-k)}{2},1})=q^{\frac{3(t-k)}{2}}(q-1)^2(q^{\frac{t-k+2}{2}}-1).$

To prove (b), we see that if there exists a singular vector, say $r_1(X),$  in a 2-dimensional non-degenerate $\mathcal{K}_{i}$-subspace $\langle u(X),v(X)\rangle$ of $\langle a_1(X), b_1(X),\cdots,a_{\frac{k-2}{2}}(X),b_{\frac{k-2}{2}}(X) \rangle^{\perp_{\delta}}$ satisfying  $Q_i\big(u(X)\big) \neq 0,~Q_i\big(v(X)\big) \\\neq 0$ and $\left[u(X),v(X)\right]_{\delta}=0,$ then by Lemma 7.3 of \cite{tay}, we have $\langle u(X),v(X)\rangle\hspace{-0.4mm}=\hspace{-0.4mm}\langle r_1(X),r_2(X) \rangle,$ where $(r_1(X),r_2(X))$ is a hyperbolic pair in the $\mathcal{K}_{i}$-subspace $\langle a_1(X), b_1(X), \cdots,a_{\frac{k-2}{2}}(X),b_{\frac{k-2}{2}}(X) \rangle^{\perp_{\delta}}$ of $\mathcal{J}_i.$ By \cite[p. 138]{tay}, we see that the space $\langle a_1(X), b_1(X), \cdots,a_{\frac{k-2}{2}}(X),b_{\frac{k-2}{2}}(X) \rangle^{\perp_{\delta}}$  contains $H_{\frac{t-k+2}{2},0}$ hyperbolic pairs. On the other hand, by \cite[p. 138]{tay} again, we see that there are precisely $H_{1,0}$ distinct hyperbolic pairs in a 2-dimensional non-degenerate quadratic $\mathcal{K}_{i}$-subspace of $\mathcal{J}_{i}$ containing a singular vector. Thus the number of distinct 2-dimensional non-degenerate $\mathcal{K}_i$-subspaces of $\langle a_1(X), b_1(X), \cdots,a_{\frac{k-2}{2}}(X),b_{\frac{k-2}{2}}(X) \rangle^{\perp_{\delta}}$ containing a singular vector is given by $\frac{H_{\frac{t-k+2}{2},0}}{H_{1,0}},$ which equals $\frac{q^{t-k} (q^{\frac{t-k+2}{2}}-1) (q^{\frac{t-k}{2}}+1)}{2(q-1)}$ by \cite[p. 141]{tay}. Further, we observe that within each such non-degenerate quadratic $\mathcal{K}_i$-subspace $\langle u(X),v(X)\rangle=\left<r_1(X),r_2(X)\right>$ of $\langle a_1(X), b_1(X), \cdots,a_{\frac{k-2}{2}}(X),b_{\frac{k-2}{2}}(X) \rangle^{\perp_{\delta}}$  containing a singular  vector, there are precisely $(q-1)^3$ distinct choices for a basis of the type $\{u(X),v(X)\}$ satisfying $Q_i\big(u(X)\big) \neq 0,~Q_i\big(v(X)\big) \neq 0$ and $\left[u(X),v(X)\right]_{\delta}=0.$   Thus there are precisely $\Delta_k=\frac{q^{t-k}(q-1)^3 (q^{\frac{t-k+2}{2}}-1) (q^{\frac{t-k}{2}}+1)}{2(q-1)}$ choices for a basis of the type $\{u(X),v(X)\}$ with $u(X),v(X) \in \langle a_1(X), b_1(X), \cdots,a_{\frac{k-2}{2}}(X),b_{\frac{k-2}{2}}(X) \rangle^{\perp_{\delta}}$ satisfying  $Q_i\big(u(X)\big) \neq 0,~Q_i\big(v(X)\big) \neq 0,$ $\left[u(X),v(X)\right]_{\delta}=0$ and $\langle u(X),v(X)\rangle$ containing a singular vector.

Now from assertions  (a) and (b), it follows that  the number of orthogonal bases  of 2-dimensional anisotropic $\mathcal{K}_{i}$-subspaces of $\langle a_1(X), b_1(X), \cdots, a_{\frac{k-2}{2}}(X),b_{\frac{k-2}{2}}(X) \rangle^{\perp_{\delta}}$ is given by  $ \mathbb{A}_{k,\frac{t}{2}}=    q^{\frac{3(t-k)}{2}}(q-1)^2(q^{\frac{t-k+2}{2}}\\ -1)- \Delta_k =  \frac{q^{t-k}(q-1)^2(q^{\frac{t-k+2}{2}}-1) (q^{\frac{t-k}{2}}-1)}{2},$ which proves assertion \eqref{9} when $\mathfrak{w}=\frac{t}{2}.$\\
\textbf{(ii)} Next let $\mathfrak{w}=\frac{t-1}{2}.$  Here we assert  the following:
\begin{enumerate}
\vspace{-2mm}\item[(c)] The number of orthogonal bases of the type $\{u(X),v(X)\}$ with $u(X),v(X)$ as non-singular vectors of  $\langle a_1(X), b_1(X), \cdots, a_{\frac{k-2}{2}}(X),b_{\frac{k-2}{2}}(X) \rangle^{\perp_{\delta}}$ is given by $q^{t-k}(q-1)^2(q^{t-k+1}-1).$
\vspace{-2mm}\item[(d)] The number of orthogonal bases of the type $\{u(X),v(X)\}$ with $u(X),v(X)$ as non-singular vectors of  $\langle a_1(X), b_1(X), \cdots, a_{\frac{k-2}{2}}(X),b_{\frac{k-2}{2}}(X) \rangle^{\perp_{\delta}}$ and $\left<u(X),v(X)\right>$ containing a singular vector, is given by $\frac{q^{t-k}(q-1)^2 (q^{t-k+1}-1)}{2}.$
    \vspace{-2mm}\end{enumerate}
To prove this assertion, as $\mathfrak{w}=\frac{t-1}{2},$ by \cite[p. 138]{tay}, we  write $\mathcal{J}_i=\langle r_1(X),s_1(X) \rangle \perp \cdots \perp \langle r_{\frac{t-1}{2}}(X),s_{\frac{t-1}{2}}(X)\rangle \perp \langle \eta(X)\rangle,$ where  $(r_{\ell}(X),s_{\ell}(X))$ is a hyperbolic pair in $\mathcal{J}_i$ for $1 \leq \ell \leq \frac{t-1}{2}$ and $\eta(X)$ is a non-singular vector of $\mathcal{J}_i.$

In order to prove (c), by applying Witt's cancellation theorem, we see that the Witt index of the $\mathcal{K}_i$-subspace $\langle a_1(X), b_1(X),  \cdots, a_{\frac{k-2}{2}}(X),b_{\frac{k-2}{2}}(X) \rangle^{\perp_{\delta}}$ of $\mathcal{J}_i$ is $\frac{(t-1)}{2}-\frac{(k-2)}{2}=\frac{t-k+1}{2}.$ Now we need to choose a non-singular vector $u(X)$ from the space $\langle a_1(X), b_1(X), \cdots, a_{\frac{k-2}{2}}(X), b_{\frac{k-2}{2}}(X) \rangle^{\perp_{\delta}}.$ Here the following two cases arise: \begin{itemize}\vspace{-2mm}\item  the non-singular vector $u(X)\in \langle a_1(X), b_1(X),\cdots,a_{\frac{k-2}{2}}(X), b_{\frac{k-2}{2}}(X) \rangle^{\perp_{\delta}}$ is  such that the $\mathcal{K}_i$-subspace $\langle u(X)\rangle$ is isometric to  $\langle \eta(X)\rangle,$ and \vspace{-2mm}\item the non-singular vector $u(X)\in \langle a_1(X), b_1(X),\cdots,a_{\frac{k-2}{2}}(X), b_{\frac{k-2}{2}}(X) \rangle^{\perp_{\delta}}$ is such that the $\mathcal{K}_i$-subspace $\langle u(X)\rangle$ is not isometric to $\langle \eta(X)\rangle.$ \vspace{-2mm}\end{itemize}

First of all, we choose a non-singular vector $u(X)$  from $\langle a_1(X), b_1(X),\cdots,a_{\frac{k-2}{2}}(X), b_{\frac{k-2}{2}}(X) \rangle^{\perp_{\delta}}$ such that $\langle u(X)\rangle $  is isometric to  $\langle \eta(X)\rangle.$ In this case, by Lemma \ref{isometry}(a), we see that  $u(X)$ has $\mathsf{L}_{t-k+2}$ choices.
Next for each such choice of $u(X),$ by applying Proposition 2.9 of \cite{grove},  we write $\langle a_1(X), b_1(X), \cdots, a_{\frac{k-2}{2}}(X), \\ b_{\frac{k-2}{2}}(X) \rangle^{\perp_{\delta}} = \langle u(X)\rangle \perp \langle u(X)\rangle^{\perp_{\delta}},$ where $\langle u(X) \rangle^{\perp_{\delta}}$ is a $(t-k+1)$-dimensional non-degenerate quadratic $\mathcal{K}_{i}$-subspace of $\langle a_1(X), b_1(X), \cdots, a_{\frac{k-2}{2}}(X), b_{\frac{k-2}{2}}(X) \rangle^{\perp_{\delta}}.$ Here as  $\langle u(X)\rangle$  and $\langle \eta(X)\rangle$ are isometric, by Witt's cancellation theorem, we see that the Witt index of $\langle u(X) \rangle^{\perp_{\delta}}$ is $\frac{t-k+1}{2}.$ Now we will  choose another non-singular vector $v(X)$ from the space $\langle u(X) \rangle^{\perp_{\delta}}.$ As $u(X)$ is non-singular, we see that $u(X) \not\in \langle u(X)\rangle^{\perp_{\delta}}.$  By \cite[p. 138]{tay}, we observe that $v(X)$ has $ q^{t-k+1}-1-I_{\frac{(t-k+1)}{2},0}$ relevant choices. Therefore by Theorem 11.5 of \cite{tay} and Lemma \ref{isometry}(a), the number of choices for the basis $\{u(X),v(X)\}$ with $\left<u(X)\right>$ isometric to $\left<\eta(X)\right>$ and satisfying $Q_i\big(u(X)\big) \neq 0,~Q_i\big(v(X)\big) \neq 0$ and $\left[u(X),v(X)\right]_{\delta}=0,$ is given by $\mathsf{L}_{t-k+2}(q^{t-k+1}-1-I_{\frac{(t-k+1)}{2},0})= \frac{q^{t-k}(q-1)^2(q^{t-k+1}-1)}{2}.$

 Next we choose a non-singular vector $u(X)$ from $\langle a_1(X), b_1(X),\cdots,a_{\frac{k-2}{2}}(X), b_{\frac{k-2}{2}}(X) \rangle^{\perp_{\delta}}$ such that $\langle u(X)\rangle$ is not isometric to $\langle \eta(X)\rangle.$ Here by Lemma \ref{isometry}(b), such a non-singular vector $u(X)$ has $\mathsf{M}_{t-k+2}$ choices. Now by applying Proposition 2.9 of \cite{grove}, we write $\langle a_1(X), b_1(X), \cdots, a_{\frac{k-2}{2}}(X), b_{\frac{k-2}{2}}(X) \rangle^{\perp_{\delta}} = \langle u(X)\rangle \perp \langle u(X)\rangle^{\perp_{\delta}},$ where $\langle u(X) \rangle^{\perp_{\delta}}$ is a $(t-k+1)$-dimensional non-degenerate quadratic $\mathcal{K}_{i}$-subspace of $\langle a_1(X), b_1(X),  \cdots, a_{\frac{k-2}{2}}(X), b_{\frac{k-2}{2}}(X) \rangle^{\perp_{\delta}}.$ Here as the $\mathcal{K}_i$-subspaces $\langle u(X)\rangle$ and $\langle \eta(X)\rangle$ of $\mathcal{J}_i$ are not isometric, we see, by Witt's cancellation theorem,  that the Witt index of $\langle u(X) \rangle^{\perp_{\delta}}$ is $\frac{t-k-1}{2}.$ Now we will  choose another non-singular vector $v(X)$ from the space $\langle u(X) \rangle^{\perp_{\delta}}.$ As $u(X)$ is non-singular, we see that $u(X) \not\in \langle u(X)\rangle^{\perp_{\delta}}.$  By \cite[p. 138]{tay}, we see that $v(X)$ has $ q^{t-k+1}-1-I_{\frac{(t-k-1)}{2},2}$ choices. Therefore, by Theorem 11.5 of \cite{tay} and Lemma \ref{isometry}(b), the number of choices for the basis $\{u(X),v(X)\}$ with $\left<u(X)\right>$ not isometric to $\left<\eta(X)\right>$ and  satisfying $Q_i\big(u(X)\big) \neq 0,~Q_i\big(v(X)\big) \neq 0$ and  $\left[u(X),v(X)\right]_{\delta}=0,$ is given by $\mathsf{M}_{t-k+2}(q^{t-k+1}-1-I_{\frac{(t-k-1)}{2},2})=\frac{q^{t-k}(q-1)^2(q^{t-k+1}-1)}{2}.$

From the above, we see that the number of choices for the basis $\{u(X),v(X)\}$ with $u(X),v(X) \in \hspace{-0.5mm}\langle a_1(X), b_1(X),\cdots,a_{\frac{k-2}{2}}(X), b_{\frac{k-2}{2}}(X) \rangle^{\perp_{\delta}} $  satisfying $Q_i\big(u(X)\big) \neq 0,~Q_i\big(v(X)\big) \neq 0$ and $\left[u(X),v(X)\right]_{\delta}=0,$ is given by $\frac{q^{t-k}(q-1)^2(q^{t-k+1}-1)}{2}+  \frac{q^{t-k}(q-1)^2(q^{t-k+1}-1)}{2}= q^{t-k}(q-1)^2(q^{t-k+1}-1),$ which proves (c).\\
Next working in a similar way as in assertion (b) of case (i), assertion (d)   follows.

Now from assertions (c) and (d), we see that  the number of orthogonal bases  of 2-dimensional anisotropic $\mathcal{K}_{i}$-subspaces of $\langle a_1(X), b_1(X), \cdots, a_{\frac{k-2}{2}}(X),b_{\frac{k-2}{2}}(X) \rangle^{\perp_{\delta}}$ is given by  $ \mathbb{A}_{k,\frac{t-1}{2}}=    q^{t-k}(q-1)^2(q^{t-k+1}-1)- \frac{q^{t-k}(q-1)^2(q^{t-k+1}-1)}{2}= \frac{q^{t-k}(q-1)^2(q^{t-k+1}-1)}{2},$ which proves assertion \eqref{9} when $\mathfrak{w}=\frac{t-1}{2}.$
\\\textbf{(iii)} Next suppose that $\mathfrak{w}=\frac{t-2}{2}.$ Here working in a similar manner as in case (i), we observe the following:
\begin{enumerate}
\vspace{-2mm}\item[(e)] The number of orthogonal bases of the type $\{u(X),v(X)\}$ with $u(X),v(X)$ as non-singular vectors of  $\langle a_1(X), b_1(X), \cdots, a_{\frac{k-2}{2}}(X),b_{\frac{k-2}{2}}(X) \rangle^{\perp_{\delta}}$ is given by $q^{\frac{3(t-k)}{2}}(q-1)^2(q^{\frac{t-k+2}{2}}+1).$
\vspace{-2mm}\item[(f)] The number of orthogonal bases of the type $\{u(X),v(X)\}$ with $u(X),v(X)$ as non-singular vectors of  $\langle a_1(X), b_1(X), \cdots, a_{\frac{k-2}{2}}(X),b_{\frac{k-2}{2}}(X) \rangle^{\perp_{\delta}}$ and $\left<u(X),v(X)\right>$ containing a singular vector, is given by $\frac{q^{t-k}(q-1)^2 (q^{\frac{t-k+2}{2}}+1) (q^{\frac{t-k}{2}}-1)}{2}.$
   \vspace{-2mm} \end{enumerate}

Now from (e) and (f), it follows that  the number of orthogonal bases  of 2-dimensional anisotropic $\mathcal{K}_{i}$-subspaces of $\langle a_1(X), b_1(X), \cdots, a_{\frac{k-2}{2}}(X),b_{\frac{k-2}{2}}(X) \rangle^{\perp_{\delta}}$ is given by  $ \mathbb{A}_{k,\frac{t-2}{2}}=    q^{\frac{3(t-k)}{2}}(q-1)^2(q^{\frac{t-k+2}{2}}+1)- \frac{q^{t-k}(q-1)^2 (q^{\frac{t-k+2}{2}}+1) (q^{\frac{t-k}{2}}-1)}{2} =  \frac{q^{t-k}(q-1)^2(q^{\frac{t-k+2}{2}}+1) (q^{\frac{t-k}{2}}+1)}{2},$ which proves  assertion \eqref{9} when $\mathfrak{w}=\frac{t-2}{2}.$

Next we see that the number of Witt bases of the form $\{a_1(X),b_1(X), \cdots, a_{\frac{k-2}{2}}(X),b_{\frac{k-2}{2}}(X)\} \cup \beta_{W_k}$ with $\beta_{W_k}$ as an orthogonal basis of a 2-dimensional anisotropic $\mathcal{K}_{i}$-subspace $W_k$ of $\langle a_1(X), b_1(X), \cdots, \\a_{\frac{k-2}{2}}(X),b_{\frac{k-2}{2}}(X) \rangle^{\perp_{\delta}},$ is given by
 \vspace{-3.5mm}\begin{eqnarray} \label{10}\mathfrak{Q}_{\frac{k-2}{2},\mathfrak{w}}= H_{\mathfrak{w},t-2\mathfrak{w}} H_{\mathfrak{w}-1,t-2\mathfrak{w}} \cdots H_{\mathfrak{w}-\frac{(k-2)}{2}+1,t-2\mathfrak{w}}\mathbb{A}_{k,\mathfrak{w}}. \vspace{-8mm}\end{eqnarray}
From this, using \eqref{9}, \eqref{10} and  \cite[p. 141]{tay}, we get
 \vspace{-2mm}\begin{equation}\label{11}
\hspace{-2mm}\mathfrak{Q}_{\frac{k-2}{2},\mathfrak{w}}=\left\{
  \begin{array}{ll}
    \frac{q^{\frac{k(2t-k-2)}{4}} (q^{\frac{t}{2}}-1)(q^{\frac{t-k}{2}}-1)(q-1)^2}{2}\prod\limits_{j=1}^{(k-2)/2}(q^{t-2j}-1) & \hbox{if $\mathfrak{w}=\frac{t}{2}$ and $k$ is even; } \\
    \frac{q^{\frac{k(2t-k-2)}{4}}(q-1)^2}{2} \prod \limits_{j=0}^{(k-2)/2}(q^{t-2j-1}-1) & \hbox{if $\mathfrak{w}=\frac{t-1}{2}$ and $k$ is even;}\\
    \frac{q^{\frac{k(2t-k-2)}{4}} (q^{\frac{t}{2}}+1)(q^{\frac{t-k}{2}}+1)(q-1)^2}{2}\prod\limits_{j=1}^{(k-2)/2}(q^{t-2j}-1) & \hbox{if $\mathfrak{w}=\frac{t-2}{2}$ and $k$ is even.} \\
    \end{array}
   \right.
    \vspace{-2mm}\end{equation}
Next  let $\mathfrak{Q}_{\frac{k-2}{2}}$ denote the number of Witt bases of a $k$-dimensional non-degenerate quadratic $\mathcal{K}_i$-subspace of $\mathcal{J}_{i}$ of the form $\{a_1(X),b_1(X),\cdots,a_{\frac{k-2}{2}}(X),b_{\frac{k-2}{2}}(X)\}\cup \beta_{W_k},$ where $\beta_{W_k}$ is an orthogonal basis of the 2-dimensional anisotropic $\mathcal{K}_{i}$-subspace $W_k$ of $\mathcal{J}_i.$  Here also, working in a similar manner as above and by \cite[pp. 138-141]{tay}, we obtain
 \vspace{-5mm}\begin{equation}\small\label{12}\mathfrak{Q}_{\frac{k-2}{2}}=
   H_{\frac{k-2}{2},2}H_{\frac{k-4}{2},2}  \cdots H_{1,2}(q^2-1)(q-1)=q^{\frac{k(k-2)}{4}}(q^{\frac{k}{2}}+1)(q-1)^2 \prod\limits_{j=1}^{(k-2)/2} (q^{k-2j}-1).
 \vspace{-2mm}\end{equation}
 Therefore the number $R_{i,k}$ of $k$-dimensional non-degenerate quadratic $\mathcal{K}_i$-subspaces of $\mathcal{J}_i$ having Witt index as $\mathfrak{m}_k=\frac{k-2}{2},$ is given by $R_{i,k}= \frac{\mathfrak{Q}_{\frac{k-2}{2},\mathfrak{w}}}{\mathfrak{Q}_{\frac{k-2}{2}}}.$ From this and using \eqref{11} and \eqref{12}, we obtain
 \vspace{-3mm}\begin{equation}\label{Q3}\displaystyle
R_{i,k}=  \left\{
  \begin{array}{ll}
  \displaystyle \frac{q^{\frac{k(t-k)}{2}}(q^{\frac{k}{2}}-1) (q^{\frac{t-k}{2}}-1)}{2(q^{\frac{t}{2}}+1)}{t/2 \brack k/2}_{q^2} & \hbox{if $\mathfrak{w}=\frac{t}{2}$ and $k$ is even;}\vspace{1mm} \\
   \displaystyle \frac{q^{\frac{k(t-k)}{2}}(q^{\frac{k}{2}}-1)}{2} {(t-1)/2 \brack k/2}_{q^2} & \hbox{if $\mathfrak{w}=\frac{t-1}{2}$ and $k$ is even;}\vspace{1mm}\\
  \displaystyle \frac{q^{\frac{k(t-k)}{2}}(q^{\frac{k}{2}}-1) (q^{\frac{t-k}{2}}+1)}{2(q^{\frac{t}{2}}-1)}{t/2 \brack k/2}_{q^2} & \hbox{if $\mathfrak{w}=\frac{t-2}{2}$ and $k$ is even}.
   \end{array}
   \right.
    \vspace{-2mm}\end{equation}
\\\textbf{B.} Next let $\mathfrak{m}_k=\frac{k}{2}.$  Here by \cite[p. 138]{tay}, any $k$-dimensional non-degenerate quadratic $\mathcal{K}_{i}$-subspace of $\mathcal{J}_{i}$ has a Witt decomposition of the form $\langle a_1(X),b_1(X)\rangle \perp \cdots \perp \langle a_{\frac{k}{2}}(X),b_{\frac{k}{2}}(X)\rangle,$ where $(a_{\ell}(X),b_{\ell}(X))$ is a hyperbolic pair in $\mathcal{J}_i$ for $1 \leq \ell \leq \frac{k}{2}.$ Now working as in case \textbf{A}, we see that the number of Witt bases of the form $\{a_1(X),b_1(X), \cdots, a_{\frac{k}{2}}(X),b_{\frac{k}{2}}(X)\}$  in $\mathcal{J}_i$ is given by
\vspace{-2mm}\begin{equation}\label{13}
\mathfrak{Q}_{\frac{k}{2},\mathfrak{w}}=
H_{\mathfrak{w},t-2\mathfrak{w}}  H_{\mathfrak{w}-1,t-2\mathfrak{w}} \cdots H_{\mathfrak{w}-\frac{(k-2)}{2},t-2\mathfrak{w}}. \vspace{-1mm}\end{equation}
This, by \cite[p. 141]{tay}, gives
\vspace{-1mm}\begin{equation}\label{15}
\hspace{-2mm}\mathfrak{Q}_{\frac{k}{2},\mathfrak{w}}=\left\{
  \begin{array}{ll}
    q^{\frac{k(2t-k-2)}{4}} (q^{\frac{t}{2}}-1) (q^{\frac{t-k}{2}}+1) \prod\limits_{j=1}^{(k-2)/2} (q^{t-2j}-1) & \hbox{if $\mathfrak{w}=\frac{t}{2}$ and $k$ is even; } \\
     q^{\frac{k(2t-k-2)}{4}} \prod\limits_{j=0}^{(k-2)/2} (q^{t-2j-1}-1) & \hbox{if $\mathfrak{w}=\frac{t-1}{2}$ and $k$ is even;} \\
     q^{\frac{k(2t-k-2)}{4}} (q^{\frac{t}{2}}+1) (q^{\frac{t-k}{2}}-1) \prod\limits_{j=1}^{(k-2)/2} (q^{t-2j}-1) & \hbox{if $\mathfrak{w}=\frac{t-2}{2}$ and $k$ is even}.
   \end{array}
   \right.
    \vspace{-2mm}\end{equation}
Next when $\mathfrak{m}_k=\frac{k}{2},$ let $\mathfrak{Q}_{\frac{k}{2}}$ denote the number of Witt bases of a $k$-dimensional non-degenerate quadratic $\mathcal{K}_i$-subspace of $\mathcal{J}_{i}$ having Witt index as $\frac{k}{2}.$ Here  working in a similar manner as above and by \cite[pp. 138-141]{tay}, we obtain
\vspace{-7mm}\begin{equation}\label{16}
   \mathfrak{Q}_{\frac{k}{2}}= H_{\frac{k}{2},0}H_{\frac{k-2}{2},0} \cdots H_{1,0}=2q^{\frac{k(k-2)}{4}}(q^{\frac{k}{2}}-1)\prod\limits_{j=1}^{(k-2)/2}(q^{k-2j}-1).
   \vspace{-5mm} \end{equation}
Therefore the number $S_{i,k}$ of $k$-dimensional non-degenerate $\mathcal{K}_i$-subspaces of $\mathcal{J}_i$ having Witt index as $\frac{k}{2},$ is given by $S_{i,k}= \frac{\mathfrak{Q}_{\frac{k}{2},\mathfrak{w}}}{\mathfrak{Q}_{\frac{k}{2}}}.$ From this and using \eqref{15} and \eqref{16}, we obtain
\vspace{-3mm}\begin{equation}\label{Q4}
\hspace{-2mm}S_{i,k}= \left\{
  \begin{array}{ll}
    \displaystyle \frac{q^{\frac{k(t-k)}{2}}(q^{\frac{k}{2}}+1) (q^{\frac{t-k}{2}}+1)}{2(q^{\frac{t}{2}}+1)} {t/2 \brack k/2}_{q^2} & \hbox{if $\mathfrak{w}=\frac{t}{2}$ and $k$ is even;  } \\
     \displaystyle \frac{q^{\frac{k(t-k)}{2}}(q^{\frac{k}{2}}+1)}{2}{(t-1)/2 \brack k/2}_{q^2} & \hbox{if $\mathfrak{w}=\frac{t-1}{2}$ and $k$ is even;} \\
     \displaystyle \frac{q^{\frac{k(t-k)}{2}}(q^{\frac{k}{2}}+1)(q^{\frac{t-k}{2}}-1)}{2(q^{\frac{t}{2}}-1)} {t/2 \brack k/2}_{q^2} & \hbox{if $\mathfrak{w}=\frac{t-2}{2}$ and $k$ is even}.
   \end{array}
   \right.
    \vspace{-1mm}\end{equation}

    Now on substituting the values of $R_{i,k}$ and $S_{i,k}$ from \eqref{Q3} and \eqref{Q4} in \eqref{r}, we obtain
    \vspace{-2mm}\begin{equation}\label{n2}
\hspace{-2mm}N_{i,k}=\left\{
  \begin{array}{ll}
    \displaystyle q^{\frac{k(t-k)}{2}}{t/2 \brack k/2}_{q^2} & \hbox{if $\mathfrak{w}=\frac{t}{2}$ and $k$ is even;  } \vspace{1mm}\\
     \displaystyle q^{\frac{k(t-k+1)}{2}}{(t-1)/2 \brack k/2}_{q^2} & \hbox{if $\mathfrak{w}=\frac{t-1}{2}$ and $k$ is even;} \vspace{1mm} \\
     \displaystyle q^{\frac{k(t-k)}{2}}{t/2 \brack k/2}_{q^2} & \hbox{if $\mathfrak{w}=\frac{t-2}{2}$ and $k$ is even}.
   \end{array}
   \right.
    \vspace{-2mm}\end{equation}

Further, on substituting the values of $N_{i,k}$ from \eqref{n1} and \eqref{n2} (accordingly as $k$ is even or odd) in \eqref{1} and using equation \eqref{2}, the desired result follows. $\hfill\Box$

  By closely looking at the proof of Proposition \ref{div}, we deduce the following divisibility results:
\begin{cor} Let $q$ be an odd prime power and $\lambda, \mu$ be positive integers satisfying $\mu \leq \lambda.$ Then we have the following:
\begin{itemize}
\vspace{-2mm}\item[(a)] If $q \equiv 3(\text{mod }4)$ and  $\lambda$ is odd,   then both $\displaystyle  \frac{(q^{\mu}+1) (q^{\lambda-\mu}+1)}{2(q^{\lambda}+1)} {\lambda \brack \mu}_{q^2}$ and $\displaystyle  \frac{(q^{\mu}-1) (q^{\lambda-\mu}-1)}{2(q^{\lambda}+1)} {\lambda \brack \mu}_{q^2}$ are integers.
\vspace{-2mm}\item[(b)] If either $q \equiv 3(\text{mod }4)$ and $\lambda$ is even or $q \equiv 1(\text{mod }4),$ then both
$\displaystyle  \frac{(q^{\mu}+1) (q^{\lambda-\mu}-1)}{2(q^{\lambda}-1)} {\lambda \brack \mu}_{q^2}$ and\\ $\displaystyle \frac{(q^{\mu}-1)(q^{\lambda-\mu}+1)} {2(q^{\lambda}-1)}{\lambda \brack \mu}_{q^2}$ are integers.
\end{itemize}
\end{cor}
\begin{proof}
\begin{itemize}
\vspace{-2mm}\item[(a)]  Let $i \in \mathfrak{I},$  $\delta \in \{\ast,0\},$ $q \equiv 3(\text{mod }4)$ and  $\dim_{\mathcal{K}_{i}}\mathcal{J}_{i}=2\lambda,$ where  $\lambda$ is an odd integer. Then by \eqref{2}, we see that the Witt index of $\mathcal{J}_{i}$ is $\lambda.$ Further, by \eqref{Q3}, we see that the number of $(2\mu)$-dimensional non-degenerate quadratic $\mathcal{K}_{i}$-subspaces of $\mathcal{J}_{i}$ having Witt index as $(\mu-1)$ is given by $R_{i,2\mu}=q^{2\mu(\lambda-\mu)} \frac{(q^{\mu}-1) (q^{\lambda-\mu}-1)}{2(q^{\lambda}+1)} {\lambda \brack \mu}_{q^2},$ from which it follows that $ \frac{(q^{\mu}-1) (q^{\lambda-\mu}-1)}{2(q^{\lambda}+1)} {\lambda \brack \mu}_{q^2}$ is an integer.
Moroever, using \eqref{Q4}, we see that the number of $(2\mu)$-dimensional non-degenerate quadratic $\mathcal{K}_{i}$-subspaces of $\mathcal{J}_{i}$ having Witt index as $\mu$ is given by $S_{i,2\mu}=q^{2\mu(\lambda-\mu)}  \frac{(q^{\mu}+1) (q^{\lambda-\mu}+1)}{2(q^{\lambda}+1)} {\lambda \brack \mu}_{q^2},$ which implies that $ \frac{(q^{\mu}+1) (q^{\lambda-\mu}+1)}{2(q^{\lambda}+1)} {\lambda \brack \mu}_{q^2}$ is an integer.
\vspace{-2mm}\item[(b)] Let $i \in \mathfrak{I},$ $\delta \in \{\ast,0\}$ and $\dim_{\mathcal{K}_{i}}\mathcal{J}_{i}=2\lambda,$ where $\lambda$ is an even integer when $q \equiv 3(\text{mod }4)$ and $\lambda \geq 1$ is any integer when $q \equiv 1(\text{mod }4).$ Here by \eqref{2}, we see that the Witt index of $\mathcal{J}_{i}$ is $(\lambda-1).$ Further, by \eqref{Q3}, we see that the number of $(2\mu)$-dimensional non-degenerate quadratic $\mathcal{K}_{i}$-subspaces of $\mathcal{J}_{i}$ having Witt index as $(\mu-1)$ is given by $R_{i,2\mu}= q^{2\mu(\lambda-\mu)}\frac{(q^{\mu}-1)(q^{\lambda-\mu}+1)} {2(q^{\lambda}-1)}{\lambda \brack \mu}_{q^2},$ from which it follows that $\frac{(q^{\mu}-1)(q^{\lambda-\mu}+1)} {2(q^{\lambda}-1)}{\lambda \brack \mu}_{q^2}$ is an integer.
Moroever, using \eqref{Q4}, we see that the number of $(2\mu)$-dimensional non-degenerate quadratic $\mathcal{K}_{i}$-subspaces of $\mathcal{J}_{i}$ having Witt index as $\mu$ is given by $S_{i,2\mu}=q^{2\mu(\lambda-\mu)} \frac{(q^{\mu}+1) (q^{\lambda-\mu}-1)}{2(q^{\lambda}-1)} {\lambda \brack \mu}_{q^2},$ which implies that $\frac{(q^{\mu}+1) (q^{\lambda-\mu}-1)}{2(q^{\lambda}-1)} {\lambda \brack \mu}_{q^2}$ is an integer.
\end{itemize}
\vspace{-2mm}\end{proof}
\vspace{-2mm}Next we shall consider the case $\delta=0$ and $q$ is even. As $\gcd(n,q)=1,$  $n$ must be odd, which implies that $\mathfrak{I}=\{0\},$ i.e., $i=0$ is the only choice.  In the following proposition, we determine the number $N_0$ of $\mathcal{K}_0$-subspaces $\mathcal{C}_{0}$ of $\mathcal{J}_{0}$ satisfying $\mathcal{C}_{0}\cap \mathcal{C}_{0}^{(0)}=\{0\}.$
\begin{prop}\label{quad2}
Let $\delta=0$ and $q$ be an even prime power.  Then the number $N_0$ of $\mathcal{K}_0$-subspaces $\mathcal{C}_0$ of $\mathcal{J}_0$ satisfying $\mathcal{C}_0 \cap \mathcal{C}_0^{(0)}=\{0\}$ is given by
\begin{enumerate}
\vspace{-2mm}\item[(a)] $\displaystyle N_0= 2+\sum\limits_{\stackrel{k=1}{k\equiv 0(\text{mod }2) }}^{t-1}\hspace{-2mm}q^{\frac{k(t-k+1)}{2}}  {(t-1)/2 \brack k/2}_{q^2} + \sum\limits_{\stackrel{k=1}{k\equiv 1(\text{mod }2) }}^{t-1}\hspace{-2mm}q^{\frac{(t-k)(k+1)}{2}} {(t-1)/2 \brack (k-1)/2}_{q^2}$ if $t$ is odd.
\vspace{-2mm}\item[(b)] $\displaystyle  N_0= 2+  \sum\limits_{\stackrel{k=1}{k\equiv 0(\text{mod }2) }}^{t-1}q^{\frac{tk-k^2-2}{2}}\Big\{(q^{k}+q-1){(t-2)/2 \brack k/2}_{q^2} +(q^{t-k+1}-q^{t-k}+1){(t-2)/2 \brack (k-2)/2}_{q^2} \Big\} + \sum\limits_{\stackrel{k=1}{k\equiv 1(\text{mod }2) }}^{t-1} q^{\frac{tk-k^2+t-1}{2}} {(t-2)/2 \brack (k-1)/2}_{q^2}  $ \vspace{-2mm} if $t$ is even.
\end{enumerate}
\end{prop}
\begin{proof}
To prove this, for each integer $k~(0 \leq k \leq t),$ let $N_{0,k}$ denote the number of $k$-dimensional $\mathcal{K}_0$-subspaces $\mathcal{C}_0$ of  $\mathcal{J}_0$ satisfying $\mathcal{C}_0 \cap \mathcal{C}_0^{(0)} =\{0\}.$ Then we have $N_0=\sum\limits_{k=0}^{t}N_{0,k}.$ For this, we first observe that $N_{0,0}=N_{0,t}=1.$ So from now onwards, we assume that $1 \leq k \leq t-1.$ By Remark \ref{non-deg1}, we note that the number $N_{0,k}$ is equal to the number of  $k$-dimensional non-degenerate $\mathcal{K}_{0}$-subspaces of $\mathcal{J}_{0}$ for $1 \leq k \leq t-1.$ To determine the numbers $N_{0,k}$ for $1 \leq k \leq t-1,$ by Huffman \cite[p. 264]{huff}, we see that  $\mathcal{J}_0 = \{ aM(X): a \in \mathbb{F}_{q^t} \} \simeq \mathbb{F}_{q^t}$ and $ \mathcal{K}_0 = \{ rM(X): r \in \mathbb{F}_{q} \} \simeq \mathbb{F}_{q},$ where $M(X)=1+X+X^2+\cdots+X^{n-1} \neq 0.$ In view of this, it is easy to observe that $ \left[aM(X),bM(X) \right]_{0}= n\text{Tr}_{q,t}(ab) M(X)=n(a,b)_0M(X)$ for all $a,b \in \mathbb{F}_{q^t}.$ From this, we observe that if two vectors  $aM(X),bM(X) \in \mathcal{J}_{0}$ are orthogonal with respect to $\left[\cdot,\cdot\right]_{0}\restriction_{\mathcal{J}_0\times \mathcal{J}_0},$ then the corresponding vectors $a,b \in \mathbb{F}_{q^t}$ are also orthogonal with respect to  ordinary trace bilinear form $(\cdot,\cdot)_{0}$ on $\mathbb{F}_{q^t}$ and vice versa. So for each integer $k~(1 \leq k \leq t-1),$ the number $N_{0,k}$ of $k$-dimensional non-degenerate $\mathcal{K}_0$-subspaces of $\mathcal{J}_0$ with respect to $\left[\cdot,\cdot \right]_{0}\restriction _{\mathcal{J}_{0}\times \mathcal{J}_{0}}$ is equal to the number of $k$-dimensional non-degenerate $\mathbb{F}_q$-subspaces of $\mathbb{F}_{q^t}$ with respect to $(\cdot,\cdot)_{0}$ on $\mathbb{F}_{q^t}.$

In order to determine the number of $k$-dimensional non-degenerate $\mathbb{F}_q$-subspaces of $\mathbb{F}_{q^t}$ with respect to $(\cdot,\cdot)_{0},$ let $\mathcal{V}_{0}= \{x\in \mathbb{F}_{q^t}:\text{Tr}_{q,t}(x)=0 \},$ i.e., $\mathcal{V}_{0}$ equals  kernel of the trace map $\text{Tr}_{q,t}.$ Then it is well-known that $\mathcal{V}_0$ is an $\mathbb{F}_{q}$-subspace of $\mathbb{F}_{q^t}$ having dimension $t-1.$ Now we will distinguish the following two cases:  $t$ is odd and  $t$ is even.
\\ {\bf (a)} Let $t$ be odd.  Here we see that $1 \not \in \mathcal{V}_0$ and $(v,1)_0=0$ for all $v \in \mathcal{V}_0,$ which implies that $\mathbb{F}_{q^t}=\mathcal{V}_0 \perp \langle 1\rangle.$ Furthermore, by Huffman \cite[p. 280]{huff}, we see that $( \cdot , \cdot)_{0} $  is a reflexive, non-degenerate and alternating form on $\mathcal{V}_0,$ i.e., $(\mathcal{V}_0, (\cdot,\cdot)_{0}\restriction_{\mathcal{V}_0 \times \mathcal{V}_0})$ is a symplectic space over $\mathbb{F}_q.$

We now proceed to count all $k$-dimensional non-degenerate $\mathbb{F}_q$-subspaces of $\mathbb{F}_{q^t}$ with respect to $(\cdot , \cdot)_{0}.$ To do so, we first observe that any $k$-dimensional $\mathbb{F}_{q}$-subspace of $\mathbb{F}_{q^t}$ is \textbf{I.} either  contained in $\mathcal{V}_0,$ or \textbf{II.}  contained in $\mathbb{F}_{q^t}$ but not in $\mathcal{V}_{0},$ i.e., it is of the type $\langle x_1,x_2, \cdots, x_{k-1},1+x_k \rangle$ with  $x_j \in \mathcal{V}_0 \setminus \{0\}$ for $1\leq j \leq k-1$ and $x_k \in \mathcal{V}_0.$
\\\textbf{I.~~}   As $(\mathcal{V}_0, (\cdot,\cdot)_{0}\restriction_{\mathcal{V}_0 \times \mathcal{V}_0})$ is a symplectic space, by \cite[p. 69]{tay},   we see that a $k$-dimensional  non-degenerate (and hence symplectic) $\mathbb{F}_{q}$-subspace of $\mathcal{V}_0$ exists if and only if $k$ is even. Further,  working in a similar manner as in Proposition \ref{symp} and using the fact that  $\text{dim}_{\mathbb{F}_{q}}\mathcal{V}_{0}=t-1,$  we see that the number of $k$-dimensional  non-degenerate (and hence symplectic) $\mathbb{F}_{q}$-subspaces of $\mathcal{V}_0$  is given by  $q^{\frac{k(t-k-1)}{2}} {(t-1)/2 \brack k/2}_{q^2}$ when $k$ is even.
\\\textbf{II.~~} We will next count all $k$-dimensional non-degenerate $\mathbb{F}_{q}$-subspaces of $\mathbb{F}_{q^t}$ of the type $\langle x_1, x_2,\cdots, x_{k-1},1+x_k \rangle,$ where  $x_j \in \mathcal{V}_0 \setminus \{0\}$ for $1\leq j \leq k-1$ and $x_k \in \mathcal{V}_{0}.$  

First let $k$ be even. Here when $x_k=0,$  by  applying Theorem 5.1.1 of \cite{szy},  we see that  the $k$-dimensional $\mathbb{F}_q$-subspace $\langle x_1,x_2, \cdots, x_{k-1},1\rangle$  of $\mathbb{F}_{q^t}$ is  degenerate. Next let $x_k \neq 0.$ In this case, by Theorem 5.1.1 of \cite{szy} again, we note that $\langle x_1, x_2,\cdots, x_{k-1},1+x_k\rangle$ is a $k$-dimensional non-degenerate $\mathbb{F}_q$-subspace of $\mathbb{F}_{q^t}$ if and only if $\langle x_1, x_2,\cdots, x_{k-1},x_k\rangle$ is a  $k$-dimensional non-degenerate $\mathbb{F}_q$-subspace of $\mathcal{V}_0.$ Now working similarly as in Proposition \ref{symp} again, we see that the number of distinct $k$-dimensional non-degenerate $\mathbb{F}_{q}$-subspaces of $\mathcal{V}_0$ is given by $q^{\frac{k(t-k-1)}{2}}{(t-1)/2 \brack k/2}_{q^2}.$  
Further, we observe that  given any $k$-dimensional non-degenerate $\mathbb{F}_{q}$-subspace of $\mathcal{V}_0,$ there are precisely $q^k-1$  distinct $k$-dimensional non-degenerate $\mathbb{F}_{q}$-subspaces of $\mathbb{F}_{q^t}$ of the type $\left<x_1,x_2,\cdots,x_{k-1},1+x_k\right>$ with $x_j$'s in $\mathcal{V}_{0}\setminus \{0\}.$ 
Therefore 
when $k$ is even, the number of distinct $k$-dimensional non-degenerate $\mathbb{F}_q$-subspaces of $\mathbb{F}_{q^t}$ of the type $\left<x_1,x_2,\cdots,x_{k-1},1+x_k\right>$ with $x_j$'s in $\mathcal{V}_{0}\setminus \{0\},$ is given by $ q^{\frac{k(t-k-1)}{2}}(q^k-1){(t-1)/2 \brack k/2}_{q^2}.$

Next let $k$ be odd. Here  using Theorem 5.1.1 of \cite{szy},  we see that  when $x_k =0,$ $\langle x_1,x_2, \cdots, x_{k-1},1\rangle$ is a $k$-dimensional non-degenerate $\mathbb{F}_q$-subspace of $\mathbb{F}_{q^t}$ if and only if $\langle x_1, x_2,\cdots, x_{k-1}\rangle$  is a $(k-1)$-dimensional non-degenerate $\mathbb{F}_q$-subspace of $\mathcal{V}_0.$  Now as $(\mathcal{V}_0, (\cdot,\cdot)_{0}\restriction_{\mathcal{V}_0 \times \mathcal{V}_0})$ is a symplectic space over $\mathbb{F}_q,$ working similarly as in Proposition \ref{symp}, we see that the number of $(k-1)$-dimensional non-degenerate $\mathbb{F}_q$-subspaces of $\mathcal{V}_0$ is $q^{\frac{(k-1)(t-k)}{2}}{(t-1)/2 \brack (k-1)/2}_{q^2},$ which is equal to the number of distinct $k$-dimensional non-degenerate $\mathbb{F}_q$-subspaces of $\mathbb{F}_{q^t}$ of the type $\langle x_1,x_2,\cdots,x_{k-1},1\rangle$ with $x_j \in \mathcal{V}_{0}\setminus \{0\}$ for $1 \leq j \leq k-1.$ Next let $x_k$ be non-zero. In this case,  applying Theorem 5.1.1 of \cite{szy} again, we see  that  $\langle x_1,x_2,\cdots,x_{k-1},1+x_k\rangle $ is a $k$-dimensional non-degenerate $\mathbb{F}_q$-subspace of $\mathbb{F}_{q^t}$ if and only if $\langle x_1,x_2,\cdots,x_{k-1}\rangle$ is a $(k-1)$-dimensional  non-degenerate $\mathbb{F}_q$-subspace of $\mathcal{V}_0.$ Further, working in a similar way as in Proposition \ref{symp} again, we see that the number of $(k-1)$-dimensional non-degenerate $\mathbb{F}_q$-subspaces of $\mathcal{V}_0$ is $q^{\frac{(k-1)(t-k)}{2}}{(t-1)/2 \brack (k-1)/2}_{q^2}.$ Further, for a given $(k-1)$-dimensional non-degenerate $\mathbb{F}_q$-subspace $\left<x_1,x_2,\cdots,x_{k-1}\right>$ of $\mathcal{V}_{0},$ we need to count all choices for $x_k \in \mathcal{V}_{0}\setminus \{0\}$ that give rise to distinct $k$-dimensional non-degenerate $\mathbb{F}_{q}$-subspaces $\left<x_1,x_2,\cdots,x_{k-1},1+x_k\right>$ of $\mathbb{F}_{q^t}.$ To do this, using Proposition 2.9 of \cite{grove}, we write $\mathcal{V}_0 = \langle x_1,x_2,\cdots,x_{k-1}\rangle \perp \langle x_1,x_2,\cdots,x_{k-1}\rangle^{\perp_{0}},$ where the $\mathbb{F}_q$-subspace  $\langle x_1,x_2,\cdots,x_{k-1}\rangle^{\perp_{0}}$ of $\mathcal{V}_0$ has dimension $t-k.$ Therefore each $x_k\in \mathcal{V}_{0}$ can be uniquely written as $x_k=w+\widetilde{w},$ where $w \in \langle x_1,x_2,\cdots,x_{k-1}\rangle$ and $\widetilde{w} \in \langle x_1,x_2,\cdots,x_{k-1}\rangle^{\perp_{0}}.$ It is easy to see that the subspace  $\langle x_1,x_2, \cdots, x_{k-1},1+w+\widetilde{w}\rangle = \langle x_1, x_2,\cdots, x_{k-1},1+\widetilde{w}\rangle,$ where $\widetilde{w} \in \langle x_1,x_2,\cdots,x_{k-1}\rangle^{\perp_{0}}.$ We further observe that  each $x_k \in \langle x_1,x_2,\cdots,x_{k-1}\rangle^{\perp_{\delta}}\setminus \{0\}$ gives rise to a distinct $k$-dimensional non-degenerate $\mathbb{F}_q$-subspace $\left<x_1,x_2,\cdots,x_{k-1},1+x_k\right>$ of $\mathbb{F}_{q^t},$ and thus there are precisely $q^{t-k}-1$ relevant choices for $x_k.$ Therefore when $k$ is odd, the number of distinct $k$-dimensional non-degenerate $\mathbb{F}_q$-subspaces of $\mathbb{F}_{q^t}$ of the type $\langle x_1,x_2,\cdots,x_{k-1},1+x_k\rangle$ with $x_j \in \mathcal{V}_0\setminus\{0\}$ for $1\leq j\leq k,$ is given by $q^{\frac{(k-1)(t-k)}{2}}(q^{t-k}-1){(t-1)/2 \brack (k-1)/2}_{q^2}.$

On combining all the above cases, when $t$ is odd, for $1 \leq k \leq t-1,$ we get
   \vspace{-2mm} $$
N_{0,k}=\left\{
  \begin{array}{ll}
   q^{\frac{k(t-k+1)}{2}} {(t-1)/2 \brack k/2}_{q^2} & \hbox{if $k$ is even;} \\
    q^{\frac{(t-k)(k+1)}{2}} {(t-1)/2 \brack (k-1)/2}_{q^2}  & \hbox{if $k$ is odd.}

   \end{array}
   \right.
  \vspace{-1mm} $$
\textbf{(b) }  Let $t$ be even. Here we have $\text{Tr}_{q,t}(1)=0,$ which implies that $1 \in \mathcal{V}_0.$ Since the trace map $\text{Tr}_{q,t}$ is non-zero, there exists $\alpha \in \mathbb{F}_{q^t}$ such that $\text{Tr}_{q,t}(\alpha) \neq 0.$ Without any loss of generality, we  assume that  $\text{Tr}_{q,t}(\alpha) =1.$ Now we define $\mathcal{V}_1=\mathcal{V}_0 \cap \alpha^{-1}\mathcal{V}_0.$ It is clear that $1 \not \in \mathcal{V}_1.$ By  Huffman \cite[p. 281]{huff}, we see that $\mathcal{V}_1$ is an $\mathbb{F}_q$-subspace of $\mathbb{F}_{q^t}$ having dimension $t-2$ and $\mathcal{V}_0=\mathcal{V}_1 \oplus \langle 1\rangle,$ which implies that $\mathbb{F}_{q^t}=\mathcal{V}_1 \oplus \langle1\rangle \oplus \langle \alpha \rangle.$ Further, we observe that $(\cdot , \cdot)_{0}\restriction_{\mathcal{V}_1\times \mathcal{V}_1}$ is a reflexive, non-degenerate and alternating  form, that is, $(\mathcal{V}_1, (\cdot , \cdot)_{0}\restriction_{\mathcal{V}_1 \times \mathcal{V}_1})$ is a symplectic space over $\mathbb{F}_q.$
  In this case, for $1 \leq k \leq t-1,$ we see that  any $k$-dimensional $\mathbb{F}_{q}$-subspace of $\mathbb{F}_{q^t}$ is \textbf{III.}
  either contained in $\mathcal{V}_1,$ or \textbf{ IV.}  contained in $\mathcal{V}_1 \oplus \langle 1\rangle$ but not in $\mathcal{V}_1,$ or \textbf{ V.}  contained in $\mathcal{V}_1 \oplus \langle \alpha \rangle$ but not  in $\mathcal{V}_1,$ or \textbf{VI.}  contained in $\mathbb{F}_{q^t}=\mathcal{V}_1 \oplus \langle1\rangle \oplus \langle \alpha \rangle$ but not in any of the spaces $\mathcal{V}_1,$ $\mathcal{V}_1 \oplus \langle 1\rangle$ and $\mathcal{V}_1 \oplus \langle \alpha\rangle.$
\\\textbf{III.~~} To begin with, we shall first enumerate all $k$-dimensional non-degenerate $\mathbb{F}_{q}$-subspaces of $\mathbb{F}_{q^t}$ that are  contained in the symplectic space $\mathcal{V}_1.$ Here we must have $1 \leq k \leq t-2.$ When $k$ is odd, by \cite[p. 69]{tay}, we see that there does not exist any $k$-dimensional non-degenerate (and hence symplectic) $\mathbb{F}_q$-subspace of $\mathcal{V}_{1}.$  When $k$ is even, working as in Proposition \ref{symp}, it is easy to see that the number of $k$-dimensional non-degenerate $\mathbb{F}_q$-subspaces of $\mathcal{V}_1$ is given by $q^{\frac{k(t-k-2)}{2}}{(t-2)/2 \brack k/2}_{q^2}.$
\\\textbf{IV.~~} Next we proceed to enumerate all $k$-dimensional non-degenerate $\mathbb{F}_{q}$-subspaces of $\mathbb{F}_{q^t}$ that are contained in $\mathcal{V}_{1} \oplus \langle 1\rangle$ but not in $\mathcal{V}_{1}.$ For this, we observe that any such $k$-dimensional $\mathbb{F}_q$-subspace of $\mathbb{F}_{q^t}$ is of the type $\langle x_1,x_2, \cdots, x_{k-1},1+x_k\rangle,$ where  $x_j \in \mathcal{V}_1 \setminus \{0\}$ for $1\leq j \leq k-1$ and $x_k \in \mathcal{V}_1.$ Here by applying Theorem 5.1.1 of \cite{szy}, it is easy to observe that  the $k$-dimensional $\mathbb{F}_q$-subspace $\langle x_1, x_2,\cdots, x_{k-1},1+x_k\rangle$ of $\mathbb{F}_{q^t}$ is degenerate when either $x_k=0$ or $x_k\neq 0$ and $k$ is odd.  Further, when $k$ is even and $x_k\neq 0,$  the $k$-dimensional $\mathbb{F}_q$-subspace $\langle x_1,x_2, \cdots, x_{k-1},1+x_k\rangle$ of $\mathbb{F}_{q^t}$ is non-degenerate if and only if the $k$-dimensional $\mathbb{F}_q$-subspace $ \langle x_1,x_2, \cdots, x_{k} \rangle $ of $\mathcal{V}_1$ is non-degenerate. Now when $k$ is even, working in a similar manner as in case II of part (a),    we see that the number of distinct $k$-dimensional non-degenerate $\mathbb{F}_q$-subspaces of $\mathbb{F}_{q^t}$ of the type $\langle x_1,x_2, \cdots, x_{k-1},1+x_k\rangle$ with $x_j \in \mathcal{V}_1\setminus\{0\}$ for $1\leq j \leq k,$ is given by $ q^{\frac{k(t-k-2)}{2}}(q^k-1){(t-2)/2 \brack k/2}_{q^2}.$
\\ {\bf{V.~~}}  Next we shall enumerate all $k$-dimensional non-degenerate $\mathbb{F}_{q}$-subspaces of $\mathbb{F}_{q^t}$ that are contained in $\mathcal{V}_{1} \oplus \langle \alpha\rangle$ but not in $\mathcal{V}_{1}.$ To do so, we observe that any such $k$-dimensional $\mathbb{F}_q$-subspace of $\mathbb{F}_{q^t}$ is of the type $\langle x_1, x_2,\cdots, x_{k-1},\alpha+x_k\rangle,$ where  $x_j \in \mathcal{V}_1 \setminus \{0\}$ for $1\leq j \leq k-1$ and $x_k \in \mathcal{V}_1.$

First let $k$ be even. Here  when $x_k=0,$ by applying Theorem 5.1.1 of \cite{szy}, it is easy to observe that the $k$-dimensional $\mathbb{F}_q$-subspace $\langle x_1,x_2, \cdots, x_{k-1},\alpha\rangle$ of $\mathbb{F}_{q^t}$ is degenerate. When $x_k\neq 0,$ the $k$-dimensional $\mathbb{F}_q$-subspace $\langle x_1,x_2, \cdots, x_{k-1},\alpha +x_k \rangle$ of $\mathbb{F}_{q^t}$ is non-degenerate if and only if the $k$-dimensional $\mathbb{F}_q$-subspace $\langle x_1,x_2, \cdots, x_{k}\rangle$  of $\mathcal{V}_1$ is non-degenerate. Now working as in case II of part (a), we see that when $k$ is even, the  number of distinct $k$-dimensional non-degenerate $\mathbb{F}_q$-subspaces of $\mathbb{F}_{q^t}$ of the type $\langle x_1,x_2,\cdots,x_{k-1},\alpha+x_k\rangle$ with $x_j$'s in $\mathcal{V}_{1}\setminus \{0\},$ is given by $ q^{\frac{k(t-k-2)}{2}}(q^k-1){(t-2)/2 \brack k/2}_{q^2}.$

Next let $k$ be odd. Here using Theorem 5.1.1 of \cite{szy}, we see that when $x_k=0,$ the $k$-dimensional $\mathbb{F}_q$-subspace $\langle x_1,x_2, \cdots, x_{k-1},\alpha\rangle$ of $\mathbb{F}_{q^t}$ is non-degenerate if and only if the $(k-1)$-dimensional $\mathbb{F}_q$-subspace $\langle x_1,x_2, \cdots, x_{k-1} \rangle $ of $\mathcal{V}_1$ is non-degenerate. Now as $\left(\mathcal{V}_{1},\left(\cdot,\cdot\right)_{0}\restriction_{\mathcal{V}_{1}
\times\mathcal{V}_{1}}\right)$ is  a symplectic space, working similarly as in Proposition \ref{symp}, we see that the number of  such spaces is given by $q^{\frac{(k-1)(t-k-1)}{2}}{(t-2)/2 \brack (k-1)/2}_{q^2}.$ When $x_k\neq 0,$ by applying Theorem 5.1.1 of \cite{szy}, we see that $k$-dimensional $\mathbb{F}_q$-subspace  $\langle x_1, x_2,\cdots, x_{k-1},\alpha + x_k\rangle$ of $\mathbb{F}_{q^t}$ is non-degenerate if and only if the $(k-1)$-dimensional $\mathbb{F}_q$-subspace $\langle x_1,x_2, \cdots, x_{k-1}\rangle$ of $\mathcal{V}_1$ is non-degenerate. Now working as in case II of part (a), we see that the number of  distinct $k$-dimensional non-degenerate $\mathbb{F}_q$-subspaces of the type $\langle x_1,x_2, \cdots, x_{k-1},\alpha +x_k\rangle$  with $x_j$'s in $\mathcal{V}_{1}\setminus \{0\},$  is given by $q^{\frac{(k-1)(t-k-1)}{2}}(q^{t-k-1}-1){(t-2)/2 \brack (k-1)/2}_{q^2}.$
\\ \textbf{VI.~~} We next proceed to enumerate all $k$-dimensional non-degenerate $\mathbb{F}_q$-subspaces of $\mathbb{F}_{q^t}$ that are not contained in any of the subspaces $\mathcal{V}_1,$ $\mathcal{V}_1 \oplus \langle 1\rangle$ and $\mathcal{V}_1 \oplus \langle \alpha\rangle$ of $\mathbb{F}_{q^t}.$ Towards this, we first observe that any such $k$-dimensional $\mathbb{F}_q$-subspace of $\mathbb{F}_{q^t}$ is of the following two types:
\vspace{-2mm}\begin{itemize} \item  $\langle x_1, x_2,\cdots, x_{k-1},1+\lambda \alpha +x_k\rangle,$ where $\lambda \in \mathbb{F}_{q} \setminus \{0\},$ $x_j \in \mathcal{V}_1 \setminus \{0\}$ for $1\leq j \leq k-1$ and $x_k \in \mathcal{V}_1.$ \vspace{-2mm}\item $\langle x_1,x_2, \cdots, x_{k-2},1+x_{k-1},\alpha+x_k\rangle,$ where  $k \geq 2,$ $x_j \in \mathcal{V}_1 \setminus \{0\}$ for $1\leq j \leq k-2$ and $x_{k-1},x_k \in \mathcal{V}_1.$\end{itemize}

\noindent To begin with, we shall first count all $k$-dimensional non-degenerate $\mathbb{F}_q$-subspaces of the type $\langle x_1, x_2,\cdots, \\x_{k-1},1+\lambda \alpha +x_k\rangle,$ where $\lambda \in \mathbb{F}_{q} \setminus \{0\},$ $x_j \in \mathcal{V}_1 \setminus \{0\}$ for $1\leq j \leq k-1$ and $x_k \in \mathcal{V}_1.$

First let $k$ be even. When $x_k=0,$ by applying Theorem 5.1.1 of \cite{szy},  we see that the $k$-dimensional $\mathbb{F}_q$-subspace $\langle x_1,x_2, \cdots, x_{k-1},1+\lambda \alpha\rangle$ of $\mathbb{F}_{q^t}$ is degenerate for each $\lambda \in \mathbb{F}_{q}\setminus \{0\}.$ When $x_k\neq 0,$ by applying Theorem 5.1.1 of \cite{szy} again, we see that  for each $\lambda \in \mathbb{F}_{q}\setminus \{0\},$ the $k$-dimensional $\mathbb{F}_q$-subspace $\langle x_1, x_2,\cdots, x_{k-1},1+\lambda \alpha +x_k\rangle$ of $\mathbb{F}_{q^t}$ is non-degenerate if and only if the $k$-dimensional $\mathbb{F}_q$-subspace $\langle x_1,x_2, \cdots, x_{k}\rangle$  of $\mathcal{V}_1$ is non-degenerate. Now as $\left(\mathcal{V}_{1},\left(\cdot,\cdot\right)_{0}\restriction_{\mathcal{V}_{1} \times\mathcal{V}_{1}}\right)$ is  a symplectic space, working similarly as in Proposition \ref{symp}, we see that the number of $k$-dimensional non-degenerate $\mathbb{F}_{q}$-subspaces of $\mathcal{V}_1$ is given by $q^{\frac{k(t-k-2)}{2}}{(t-2)/2 \brack k/2}_{q^2}.$ Further, it is easy to observe that each $k$-dimensional non-degenerate $\mathbb{F}_{q}$-subspace  of $\mathcal{V}_1$ gives rise to precisely $(q^k-1)(q-1)$ distinct $k$-dimensional $\mathbb{F}_q$-subspaces of $\mathbb{F}_{q^t}$ of the type  $\langle x_1,x_2,\cdots,x_{k-1},1+\lambda\alpha+x_k\rangle$ with $\lambda \in \mathbb{F}_q\setminus \{0\}$ and $x_j \in \mathcal{V}_1 \setminus \{0\}$ for $1\leq j \leq k.$ Therefore when $k$ is even, the number of distinct $k$-dimensional non-degenerate $\mathbb{F}_q$-subspaces of $\mathbb{F}_{q^t}$ of the type $\langle x_1,x_2, \cdots, x_{k-1},1+\lambda \alpha +x_k\rangle$ with $x_j \in \mathcal{V}_1\setminus \{0\}$ for $1\leq j\leq k$ and $\lambda \in \mathbb{F}_q\setminus \{0\},$ is given by $ q^{\frac{k(t-k-2)}{2}}(q^k-1)(q-1){(t-2)/2 \brack k/2}_{q^2}.$

Next let $k$ be odd. Here we observe, by Theorem 5.1.1 of \cite{szy},  that when $x_k=0,$ the $k$-dimensional $\mathbb{F}_q$-subspace $\langle x_1, x_2,\cdots, x_{k-1},1+\lambda \alpha\rangle$ of $\mathbb{F}_{q^t}$ is non-degenerate if and only if the $(k-1)$-dimensional $\mathbb{F}_q$-subspace $\langle x_1,x_2, \cdots, x_{k-1} \rangle $ of $\mathcal{V}_1$ is non-degenerate for each $\lambda \in \mathbb{F}_{q}\setminus \{0\}.$ Now working as in Proposition \ref{symp}, we see that the number of $(k-1)$-dimensional non-degenerate $\mathbb{F}_{q}$-subspaces of the symplectic space $\mathcal{V}_{1}$ is given by $q^{\frac{(k-1)(t-k-1)}{2}} {(t-2)/2 \brack (k-1)/2}_{q^2}.$  Further, we observe that   each $(k-1)$-dimensional non-degenerate $\mathbb{F}_{q}$-subspace of $\mathcal{V}_1$ gives rise to precisely $(q-1)$  distinct $\mathbb{F}_q$-subspaces of the type $ \langle x_1,x_2, \cdots, x_{k-1},1+\lambda\alpha\rangle$ with $x_j$'s in $\mathcal{V}_{1} \setminus \{0\}$ and $\lambda \in \mathbb{F}_{q}\setminus \{0\}.$  Therefore when $x_k=0,$ the number of distinct $k$-dimensional non-degenerate $\mathbb{F}_q$-subspaces of $\mathbb{F}_{q^t}$ of the type $\langle x_1, x_2,\cdots, x_{k-1},1+\lambda\alpha\rangle$ with $x_j$ in $\mathcal{V}_{1} \setminus \{0\}$ for $1\leq j \leq k-1$ and $\lambda \in \mathbb{F}_{q}\setminus \{0\},$  is given by $q^{\frac{(k-1)(t-k-1)}{2}} (q-1){(t-2)/2 \brack (k-1)/2}_{q^2}$ when $k$ is odd.
 Next let $x_k$ be non-zero. Here using  Theorem 5.1.1 of \cite{szy},  we see that for each $\lambda \in \mathbb{F}_{q}\setminus \{0\},$ the $k$-dimensional $\mathbb{F}_q$-subspace $\langle x_1,x_2, \cdots, x_{k-1},1+\lambda \alpha +x_k\rangle$ of $\mathbb{F}_{q^t}$ is non-degenerate if and only if the $(k-1)$-dimensional $\mathbb{F}_q$-subspace $\langle x_1,x_2, \cdots, x_{k-1}\rangle$ of $\mathcal{V}_1$ is non-degenerate. Now  as $\left(\mathcal{V}_{1},\left(\cdot,\cdot\right)_{0}\restriction_{\mathcal{V}_{1}
\times\mathcal{V}_{1}}\right)$ is  a symplectic space, working similarly as in  Proposition \ref{symp}, we see that the number of $(k-1)$-dimensional non-degenerate $\mathbb{F}_{q}$-subspaces of $\mathcal{V}_1$ is given by $q^{\frac{(k-1)(t-k-1)}{2}}{(t-2)/2 \brack (k-1)/2}_{q^2}.$ Further, for a given $(k-1)$-dimensional non-degenerate $\mathbb{F}_q$-subspace $\langle x_1,x_2,\cdots,x_{k-1}\rangle$  of $\mathcal{V}_1,$ we need to count all choices for  $x_k \in \mathcal{V}_1\setminus \{0\}$ that give rise to distinct $k$-dimensional non-degenerate $\mathbb{F}_q$-subspaces $\left<x_1,x_2,\cdots,x_{k-1},1+\lambda\alpha+x_k\right>$ of $\mathbb{F}_{q^t}$ for each $\lambda \in \mathbb{F}_{q} \setminus \{0\}.$ To determine the number of choices for such an element $x_k \in \mathcal{V}_{1},$ we note that given any $(k-1)$-dimensional non-degenerate  $\mathbb{F}_{q}$-subspace $\langle x_1,x_2,\cdots,x_{k-1}\rangle$ of $\mathcal{V}_1,$  by Proposition 2.9 of \cite{grove}, we can write $\mathcal{V}_1 =\langle x_1,x_2,\cdots,x_{k-1}\rangle \perp \langle x_1,x_2,\cdots,x_{k-1}\rangle ^{\perp_{0}},$ where $\langle x_1,x_2,\cdots,x_{k-1}\rangle ^{\perp_{0}}$ is a non-degenerate $\mathbb{F}_q$-subspace of $\mathcal{V}_1$ having dimension $t-k-1.$ In view of this, each $x_k\in \mathcal{V}_{1}$ can be uniquely written as $x_k=v+\widetilde{v},$ where $v\in \langle x_1,x_2,\cdots,x_{k-1}\rangle$  and $\widetilde{v} \in \langle x_1,x_2,\cdots,x_{k-1}\rangle ^{\perp_{0}}.$ It is easy to see that the subspace  $\langle x_1, x_2,\cdots, x_{k-1},1+\lambda\alpha+v+\widetilde{v}\rangle = \langle x_1,x_2, \cdots, x_{k-1},1+\lambda\alpha+\widetilde{v}\rangle,$ where $\widetilde{v} \in \langle x_1,x_2,\cdots,x_{k-1}\rangle^{\perp_{0}}.$ We further observe that  for  every $\lambda \in \mathbb{F}_{q}\setminus \{0\}$, each non-zero $x_k \in \langle x_1,x_2,\cdots,x_{k-1}\rangle ^{\perp_{0}}$ gives rise to a distinct $k$-dimensional non-degenerate $\mathbb{F}_q$-subspace $\langle x_1,x_2,\cdots,x_{k-1},1+\lambda\alpha+x_k\rangle$ of $\mathbb{F}_{q^t},$ and thus there are precisely $q^{t-k-1}-1$ relevant choices for $x_k.$ Moreover, each choice of $\lambda \in \mathbb{F}_q \setminus \{0\}$ gives rise to a distinct $k$-dimensional non-degenerate $\mathbb{F}_q$-subspace of the type $\langle x_1, x_2,\cdots, x_{k-1}, 1+\lambda\alpha + x_k\rangle.$ Therefore the number of distinct $k$-dimensional non-degenerate $\mathbb{F}_q$-subspaces of $\mathbb{F}_{q^t}$ of the type $\langle x_1, x_2,\cdots, x_{k-1},1+\lambda \alpha +x_k\rangle$ with $x_j \in \mathcal{V}_1\setminus \{0\}$ for $1\leq j\leq k$ and $\lambda \in \mathbb{F}_q \setminus \{0\},$ is given by $q^{\frac{(k-1)(t-k-1)}{2}}(q-1)(q^{t-k-1}-1){(t-2)/2 \brack (k-1)/2}_{q^2}$ when $k$ is odd.

Henceforth we assume that  $k \geq 2.$  Now we shall enumerate all $k$-dimensional non-degenerate $\mathbb{F}_q$-subspaces of $\mathbb{F}_{q^t}$ of the form $\langle x_1,x_2, \cdots, x_{k-2}, 1+x_{k-1},\alpha+x_k \rangle$ with $k \geq 2,$ where $x_j \in \mathcal{V}_1 \setminus \{0\}$ for $1\leq j \leq k-2$ and $x_{k-1},x_k \in \mathcal{V}_1.$ Here also, we shall distinguish the cases: \textbf{A.} $k$ is even and \textbf{B.} $k$ is odd.

\noindent\textbf{Case A.} First let $k$ be even.
When $x_{k-1}=x_k=0,$ by Theorem 5.1.1 of \cite{szy}, we see that the $k$-dimensional $\mathbb{F}_q$-subspace $\langle x_1,x_2, \cdots, x_{k-2},1,\alpha\rangle$ of $\mathbb{F}_{q^t}$ is non-degenerate if and only if the $(k-2)$-dimensional $\mathbb{F}_q$-subspace $\langle x_1,x_2, \cdots, x_{k-2} \rangle$ of $\mathcal{V}_1$ is non-degenerate.
Now as $\left(\mathcal{V}_{1},\left(\cdot,\cdot\right)_{0}\restriction_{\mathcal{V}_{1} \times \mathcal{V}_{1}} \right)$ is  a symplectic space, working similarly as in Proposition \ref{symp}, we see that the number of  such spaces is given by $q^{\frac{(k-2)(t-k)}{2}}{(t-2)/2 \brack (k-2)/2}_{q^2}.$ When  $x_{k-1}\neq 0$ and $x_k=0,$ by Theorem 5.1.1 of \cite{szy}, we see that the $k$-dimensional $\mathbb{F}_q$-subspace $\langle x_1,x_2, \cdots, x_{k-2},1 +x_{k-1},\alpha \rangle$ of $\mathbb{F}_{q^t}$ is non-degenerate if and only if the $(k-2)$-dimensional $\mathbb{F}_q$-subspace $\langle x_1,x_2, \cdots, x_{k-2}\rangle$ of $\mathcal{V}_1$ is non-degenerate. Now working as in case II of part (a),  we see that  the number of  distinct $k$-dimensional non-degenerate  $\mathbb{F}_{q}$-subspaces of $\mathbb{F}_{q^t}$ of the type $\langle x_1,x_2, \cdots, x_{k-2}, 1+x_{k-1},\alpha\rangle$ with $x_j$'s in $\mathcal{V}_1\setminus\{0\},$  is given by $q^{\frac{(k-2)(t-k)}{2}}(q^{t-k}-1){(t-2)/2 \brack (k-2)/2}_{q^2}.$
When $x_{k-1}=0$ and $x_k\neq 0,$ by Theorem 5.1.1 of \cite{szy}, we see that  the $k$-dimensional $\mathbb{F}_q$-subspace $\langle x_1,x_2, \cdots, x_{k-2},1,\alpha+x_{k}\rangle$ of $\mathbb{F}_{q^t}$ is non-degenerate if and only if the $(k-2)$-dimensional $\mathbb{F}_q$-subspace $\langle x_1,x_2,\cdots, x_{k-2}\rangle$ of $\mathcal{V}_1$ is non-degenerate.
Now working as in case II of part (a), we see that the number of  distinct $k$-dimensional non-degenerate $\mathbb{F}_q$-subspaces of $\mathbb{F}_{q^t}$ of the type $\langle x_1,x_2, \cdots, x_{k-2},1, \alpha+x_{k}\rangle$ with $x_j$'s in $\mathcal{V}_1\setminus\{0\},$   is given by $q^{\frac{(k-2)(t-k)}{2}}(q^{t-k}-1){(t-2)/2 \brack (k-2)/2}_{q^2}.$
Next suppose that both  $x_{k-1}, x_k \in \mathcal{V}_{1}$ are linearly independent over $\mathbb{F}_{q}.$

 For any even integer $k\geq 2,$ if $\mathfrak{G}( x_1, x_2,\cdots, x_{k-2}, 1+x_{k-1}, \alpha+x_k ),$ $\mathfrak{G}(x_1,x_2,\cdots,x_{k-2})$ and $\mathfrak{G}(x_1,x_2,\cdots,x_k)$ are  Gram matrices of $\left<x_1, x_2,\cdots, x_{k-2}, 1+x_{k-1}, \alpha+x_k\right>,$ $\left<x_1,x_2,\cdots,x_{k-2}\right>$ and $\left<x_1,x_2,\cdots,x_k\right>$ respectively, then we observe that \vspace{-3mm}\begin{equation*} \hspace{-4mm}\text{det }\mathfrak{G}( x_1, x_2,\cdots, x_{k-2}, 1+x_{k-1}, \alpha+x_k )=   \text{det }\mathfrak{G}(x_1,x_2,\cdots,x_{k-2}) + \text{det }\mathfrak{G}(x_1,x_2,\cdots,x_k). \vspace{-3mm} \end{equation*}  This, by Theorem 5.1.1 of \cite{szy},  implies that the $k$-dimensional $\mathbb{F}_q$-subspace $\langle x_1,x_2, \cdots, x_{k-2}, 1+x_{k-1},\alpha+x_k \rangle$ of $\mathbb{F}_{q^t}$ is non-degenerate if and only if either  \begin{enumerate} \vspace{-2mm}\item[\textbf{(i)}] the $(k-2)$-dimensional $\mathbb{F}_q$-subspace $\langle x_1,x_2,\cdots,x_{k-2}\rangle$ of $\mathcal{V}_{1}$ is non-degenerate and the $k$-dimensional $\mathbb{F}_q$-subspace $\langle x_1,x_2,\cdots,x_k \rangle$ of $\mathcal{V}_1$ is degenerate, or  \vspace{-2mm}\item[\textbf{(ii)}] the $(k-2)$-dimensional $\mathbb{F}_q$-subspace $\langle x_1,x_2,\cdots,x_{k-2}\rangle $ of $\mathcal{V}_1$ is degenerate and the $k$-dimensional $\mathbb{F}_q$-subspace $\langle x_1,x_2,\cdots,x_k\rangle $ of $\mathcal{V}_1$ is non-degenerate, or \vspace{-2mm}\item[\textbf{(iii)}] both the $\mathbb{F}_q$-subspaces $\langle x_1,x_2,\cdots,x_{k-2}\rangle $ and $\langle x_1,x_2,\cdots,x_k\rangle $ of $\mathcal{V}_1$ are non-degenerate with $\text{det }\mathfrak{G}(x_1, \\ x_2,\cdots,x_k)\neq \text{det }\mathfrak{G}(x_1,x_2,\cdots,x_{k-2}).$
 \vspace{-2mm}\end{enumerate}
We also observe that the condition $\text{det }\mathfrak{G}(x_1,x_2,\cdots,x_k)\neq \text{det }\mathfrak{G}(x_1,x_2,\cdots,x_{k-2})$ trivially holds in the cases \textbf{(i)} and \textbf{(ii)}.
\\\textbf{(i)} We shall first enumerate all $k$-dimensional $\mathbb{F}_q$-subspaces $\left<x_1,x_2,\cdots,x_{k-2},1+x_{k-1},\alpha+x_k\right>$ of $\mathbb{F}_{q^t}$ such that $\langle x_1,x_2,\cdots,x_{k-2}\rangle$ is a $(k-2)$-dimensional non-degenerate $\mathbb{F}_q$-subspace of $\mathcal{V}_1$ and $\langle x_1,x_2,\cdots,x_k\rangle$ is a $k$-dimensional degenerate $\mathbb{F}_q$-subspace of $\mathcal{V}_1.$  For this, as $\left(\mathcal{V}_{1},\left(\cdot,\cdot\right)_{0}\restriction_{\mathcal{V}_{1}
\times\mathcal{V}_{1}}\right)$ is  a symplectic space, working similarly as in Proposition \ref{symp},  we see that the number of $(k-2)$-dimensional non-degenerate $\mathbb{F}_q$-subspaces of $\mathcal{V}_1$  is  $q^{\frac{(k-2)(t-k)}{2}}{(t-2)/2 \brack (k-2)/2}_{q^2}.$  Now for each such $(k-2)$-dimensional non-degenerate $\mathbb{F}_q$-subspace $\left<x_1,x_2,\cdots,x_{k-2}\right>$ of $\mathcal{V}_{1},$ we need to count all choices for linearly independent vectors $x_{k-1},x_k \in \mathcal{V}_1\setminus \{0\}$ such that the $k$-dimensional $\mathbb{F}_q$-subspace $\left<x_1,x_2,\cdots,x_k\right>$ of $\mathcal{V}_{1}$ is degenerate and that give rise to distinct $k$-dimensional non-degenerate $\mathbb{F}_q$-subspaces $\langle x_1,x_2,\cdots,x_{k-2},1+x_{k-1},\alpha+x_k\rangle$ of $\mathbb{F}_{q^t}.$  For this, by Proposition 2.9 of \cite{grove}, we write $\mathcal{V}_1=\langle x_1,x_2,\cdots,x_{k-2}\rangle \perp \langle x_1,x_2,\cdots,x_{k-2}\rangle^{\perp_{0}},$ where the dimension of $\langle x_1,x_2,\cdots,x_{k-2}\rangle^{\perp_{0}}$ is $t-k.$ In view of this, both $x_{k-1},x_k\in \mathcal{V}_1 \setminus \{0\}$ can be uniquely written as $x_{k-1}=w_1+\widetilde{w_1}$ and $x_k=w_2+\widetilde{w_2},$ where $w_1,w_2 \in \langle x_1,x_2,\cdots,x_{k-2}\rangle$ and $\widetilde{w_1},\widetilde{w_2} \in \langle x_1,x_2,\cdots,x_{k-2}\rangle^{\perp_{0}}.$ It is easy to observe that $\langle x_1,x_2,\cdots,x_{k-2},1+x_{k-1},\alpha+x_k\rangle=\langle x_1,x_2,\cdots,x_{k-2},1+w_1+\widetilde{w_1},\alpha+w_2+\widetilde{w_2}\rangle= \langle x_1,x_2,\cdots,x_{k-2},1+\widetilde{w_1},\alpha+\widetilde{w_2}\rangle.$ We further observe that  each pair $(x_{k-1},x_k)$ of linearly independent vectors in $\langle x_1,x_2,\cdots,x_{k-2}\rangle^{\perp_{0}}$ gives rise to a distinct $k$-dimensional $\mathbb{F}_q$-subspace $\langle x_1,x_2,\cdots,x_{k-2},1+x_{k-1},\alpha+x_k\rangle$ of $\mathbb{F}_{q^t}$ with $x_j$'s in $\mathcal{V}_1\setminus\{0\}.$ It is easy to see that $\langle x_1,x_2,\cdots,x_k\rangle=\langle x_1,x_2,\cdots,x_{k-2}\rangle \perp \langle x_{k-1},x_k\rangle.$  This implies that $\text{det }\mathfrak{G}(x_1,x_2,\cdots,x_k) =\text{det }\mathfrak{G}(x_1,x_2,\cdots,x_{k-2})~ \text{det }\mathfrak{G}(x_{k-1},x_k).$
 Therefore by Theorem 5.1.1 of \cite{szy}, we see that the $k$-dimensional $\mathbb{F}_q$-subspace $\langle x_1,x_2,\cdots, x_k\rangle$ of $\mathcal{V}_{1}$ is degenerate if and only if  $\text{Tr}_{q,t}(x_{k-1}x_k)=0$ if and only if $x_k \in \left<x_1,x_2,\cdots,x_{k-1}\right>^{\perp_{0}}.$  As $x_{k-1} \in \langle x_1,x_2,\cdots,x_{k-2}\rangle^{\perp_0}\setminus\{0\},$ it has $q^{t-k}-1$ choices. Further, we note that $x_k \in \langle x_1,x_2, \cdots, x_{k-2}, x_{k-1} \rangle^{\perp_{0}},$ whose  dimension  is $t-k-1$ by Proposition 2.4 of \cite{grove}. Also as $x_{k-1} \in \langle x_1,x_2, \cdots, x_{k-2}, x_{k-1} \rangle^{\perp_{0}},$ we see that $x_k$ has $q^{t-k-1}-q$ choices, which implies that the pair $(x_{k-1},x_k)$ has  $(q^{t-k}-1) (q^{t-k-1}-q)$ choices. Therefore for an even integer $k\geq 4,$  the number of distinct $k$-dimensional non-degenerate $\mathbb{F}_q$-subspaces of $\mathbb{F}_{q^t}$ of the form $\langle x_1,x_2, \cdots,x_{k-2}, 1+x_{k-1}, \alpha+x_k\rangle$  such that $\langle x_1,x_2,\cdots,x_{k-2}\rangle$ is a $(k-2)$-dimensional non-degenerate $\mathbb{F}_q$-subspace of $\mathcal{V}_1$ and $\langle x_1,x_2,\cdots,x_k \rangle$ is a $k$-dimensional degenerate $\mathbb{F}_q$-subspace of $\mathcal{V}_1,$ is given by $q^{\frac{(k-2)(t-k)}{2}}(q^{t-k}-1) (q^{t-k-1}-q){(t-2)/2 \brack (k-2)/2}_{q^2}.$
\\\textbf{(ii)} Next we will enumerate all $k$-dimensional non-degenerate $\mathbb{F}_q$-subspaces $\langle x_1,x_2, \cdots,x_{k-2}, 1+x_{k-1}, \alpha+x_k\rangle$ of $\mathbb{F}_{q^t}$  such that $\langle x_1,x_2,\cdots,x_k\rangle$ is a $k$-dimensional non-degenerate $\mathbb{F}_q$-subspace of $\mathcal{V}_1$ and $\hspace{-0.3mm}\langle x_1,x_2,\cdots\hspace{-0.5mm},x_{k-2}\rangle$ is a $(k-2)$-dimensional degenerate $\mathbb{F}_q$-subspace of $\mathcal{V}_1.$ For this, as $\left(\mathcal{V}_{1},\left(\cdot,\cdot\right)_{0}\restriction_{\mathcal{V}_{1}
\times\mathcal{V}_{1}}\right)$ is  a symplectic space, working similarly as in Proposition \ref{symp}, the number of $k$-dimensional non-degenerate $\mathbb{F}_q$-subspaces of $\mathcal{V}_1$  is  $q^{\frac{k(t-k-2)}{2}}{(t-2)/2 \brack k/2}_{q^2}.$ Working as in Proposition \ref{symp} again, we see that every $k$-dimensional non-degenerate $\mathbb{F}_{q}$-subspace $\left< x_1,x_2,\cdots,x_k\right>$ of $\mathcal{V}_{1}$ has precisely $q^{k-2}{k/2 \brack (k-2)/2}_{q^2}$ distinct  $(k-2)$-dimensional non-degenerate $\mathbb{F}_{q}$-subspaces, which implies that the number of $(k-2)$-dimensional degenerate $\mathbb{F}_{q}$-subspaces of the $k$-dimensional non-degenerate $\mathbb{F}_{q}$-subspace  $\left<x_1,x_2,\cdots,x_k\right>$  of $\mathcal{V}_{1}$ is given by $\left\{ {k \brack k-2}_{q}-q^{k-2}{k/2 \brack (k-2)/2}_{q^2} \right\}.$ Furthermore, we observe that each such $(k-2)$-dimensional  degenerate $\mathbb{F}_{q}$-subspace $\left<y_1,y_2,\cdots,y_{k-2}\right>$ of $\left<x_1,x_2,\cdots,x_k\right>$ gives rise to precisely
 $(q^2-1)(q^2-q)$ distinct $k$-dimensional non-degenerate $\mathbb{F}_q$-subspaces of $\mathbb{F}_{q^t}$ of the form $\langle y_1, y_2,\cdots,y_{k-2}, 1+y_{k-1}, \alpha+y_k\rangle$ satisfying $\langle y_1,y_2,\cdots,y_k\rangle=\langle x_1,x_2,\cdots,x_k\rangle.$  Therefore for an even integer $k\geq 4,$  the number of distinct $k$-dimensional non-degenerate $\mathbb{F}_q$-subspaces of $\mathbb{F}_{q^t}$ of the form $\langle x_1,x_2, \cdots,x_{k-2}, 1+x_{k-1}, \alpha+x_k\rangle$  such that $\langle x_1,x_2,\cdots,x_{k-2}\rangle$ is a $(k-2)$-dimensional degenerate $\mathbb{F}_q$-subspace of $\mathcal{V}_1$ and $\langle x_1,x_2,\cdots,x_k \rangle$ is a $k$-dimensional non-degenerate $\mathbb{F}_q$-subspace of $\mathcal{V}_1,$ is given by $\hspace{-1mm}(q^2-1)(q^2-q)q^{\frac{k(t-k-2)}{2}}{(t-2)/2 \brack k/2}_{q^2}\hspace{-0.4mm}\left\{ {k \brack k-2}_{q}-q^{k-2}{k/2 \brack (k-2)/2}_{q^2} \right\}\hspace{-0.5mm}.$
\\\textbf{(iii)} Finally, we will enumerate all $k$-dimensional non-degenerate $\mathbb{F}_q$-subspaces $\langle x_1,x_2, \cdots \hspace{-0.75mm},x_{k-2}, \hspace{-0.2mm}1+x_{k-1}, \alpha+ \nolinebreak x_k\rangle$ of $\mathbb{F}_{q^t}$ such that $\langle x_1,x_2,\cdots,x_k\rangle$ is a $k$-dimensional non-degenerate $\mathbb{F}_q$-subspace of $\mathcal{V}_1$ and $\hspace{-0.2mm}\langle x_1,x_2,\cdots\hspace{-1mm},x_{k-2}\rangle$ is a $(k-2)$-dimensional non-degenerate $\mathbb{F}_q$-subspace of $\mathcal{V}_1$ with $\text{det }\mathfrak{G}(x_1,  x_2,\cdots,x_k)\neq \text{det }\mathfrak{G}(x_1,x_2,\cdots,x_{k-2}).$ To do so, we see that as  $\left(\mathcal{V}_{1},\left(\cdot,\cdot\right)_{0}\restriction_{\mathcal{V}_{1}
\times\mathcal{V}_{1}}\right)$ is  a $(t-2)$-dimensional symplectic space over $\mathbb{F}_{q},$ working similarly as in Proposition \ref{symp}, the number of distinct $(k-2)$-dimensional non-degenerate $\mathbb{F}_q$-subspaces of $\mathcal{V}_1$ is $q^{\frac{(k-2)(t-k)}{2}}{(t-2)/2 \brack (k-2)/2}_{q^2}.$ Next for each $(k-2)$-dimensional non-degenerate $\mathbb{F}_q$-subspace $\left<x_1,x_2,\cdots,x_{k-2}\right>$ of $\mathcal{V}_{1},$ we need to determine choices for linearly independent vectors $x_{k-1},x_k$ in $\mathcal{V}_{1}$ such that the $\mathbb{F}_q$-subspace $\langle x_1,x_2,\cdots,x_k\rangle$ of $\mathcal{V}_1$ is non-degenerate and $\text{det }\mathfrak{G}(x_1,x_2,\cdots,x_k) \neq \text{det }\mathfrak{G}(x_1,x_2,\cdots,x_{k-2}).$ For this,
 by Proposition 2.9 of \cite{grove}, we write $\mathcal{V}_1=\langle x_1,x_2,\cdots,x_{k-2}\rangle \perp \langle x_1,x_2,\cdots,x_{k-2}\rangle^{\perp_{0}},$ where the dimension of $\langle x_1,x_2,\cdots,x_{k-2}\rangle^{\perp_{0}}$ is $t-k.$ So we need to choose vectors $x_{k-1}$ and $x_k$ from $\langle x_1,x_2,\cdots,x_{k-2}\rangle^{\perp_{0}},$ and  each pair $(x_{k-1},x_k)$ of linearly independent vectors in $\langle x_1,x_2,\cdots,x_{k-2}\rangle^{\perp_{0}}$ gives rise to a distinct $k$-dimensional $\mathbb{F}_q$-subspace $\langle x_1,x_2,\cdots,x_{k-2},1+x_{k-1},\alpha+x_k\rangle$ of $\mathbb{F}_{q^t}$ with $x_j$'s in $\mathcal{V}_1\setminus\{0\}.$  It is easy to see that $\langle x_1,x_2,\cdots,x_k\rangle =\langle x_1,x_2,\cdots,x_{k-2}\rangle \perp \langle x_{k-1},x_k\rangle ,$ which gives \vspace{-2mm}\begin{equation*}\text{det }\mathfrak{G}(x_1,x_2,\cdots,x_k) =\text{det }\mathfrak{G}(x_1,x_2,\cdots,x_{k-2}) ~\text{det }\mathfrak{G}(x_{k-1},x_k).\vspace{-2.5mm}\end{equation*}
 From this, it follows that the $k$-dimensional $\mathbb{F}_q$-subspace $\left<x_1,x_2,\cdots,x_k\right>$ of $\mathcal{V}_1$ is non-degenerate if and only if $\text{det }\mathfrak{G}(x_{k-1},x_k)= \text{Tr}_{q,t}(x_{k-1}x_k)^2 \neq 0$ if and only if $x_k \not\in \left<x_1,x_2,\cdots,x_{k-1}\right>^{\perp_{0}}.$    Besides this, we note that
 $\text{det }\mathfrak{G}(x_1,x_2,\cdots,x_k) \neq \text{det }\mathfrak{G}(x_1,x_2,\cdots,x_{k-2})$ if and only if $\text{det }\mathfrak{G}(x_{k-1},x_k)=\text{Tr}_{q,t}(x_{k-1}x_k)^2 \neq 1$ if and only if $(x_{k-1},x_k)$ is not a hyperbolic pair in $\langle x_1,x_2,\cdots,x_{k-2}\rangle^{\perp_{0}}.$
First of all, we shall count all pairs $(x_{k-1},x_k)$ of linearly independent vectors in $\langle x_1,x_2,\cdots,x_{k-2}\rangle^{\perp_{0}}$ satisfying $x_k \not\in \left<x_1,x_2,\cdots,x_{k-1}\right>^{\perp_{0}}.$ For this, as $x_{k-1} \in \langle x_1,x_2,\cdots,x_{k-2}\rangle^{\perp_0}\setminus\{0\},$ it has $q^{t-k}-1$ choices. Further, we note, by Proposition 2.4 of \cite{grove}, that the dimension of $\langle x_1,x_2, \cdots, x_{k-2}, x_{k-1} \rangle^{\perp_{0}}$ is $t-k-1.$ From this, we see  that $x_k$ has $q^{t-k}-q^{t-k-1}$ choices. This implies that the number of pairs $(x_{k-1},x_k)$ such that both the $\mathbb{F}_{q}$-subspaces $\left<x_1,x_2,\cdots,x_{k-2}\right>$ and $\left<x_1,x_2,\cdots,x_k\right>$ of $\mathcal{V}_{1}$ are non-degenerate, is given by  $(q^{t-k}-1) (q^{t-k}-q^{t-k-1}).$ On the other hand, by \cite[p. 70]{tay}, we see that the number of hyperbolic pairs in $\langle x_1,x_2,\cdots,x_{k-2}\rangle^{\perp_{0}}$ is $H_{\frac{t-k}{2}},0=q^{t-k-1}(q^{t-k}-1),$ which implies that the pair $(x_{k-1},x_k)$ has $(q^{t-k}-1) (q^{t-k}-q^{t-k-1})-H_{\frac{t-k}{2}},0=(q^{t-k}-1) (q^{t-k}-q^{t-k-1})-q^{t-k-1}(q^{t-k}-1)$ relevant choices. Therefore  the number of distinct $k$-dimensional non-degenerate $\mathbb{F}_q$-subspaces of $\mathbb{F}_{q^t}$ of the type $\langle x_1,x_2,\cdots,x_{k-2},1+x_{k-1},\alpha+x_k\rangle$ such that $\left<x_1,x_2,\cdots,x_{k-2}\right>$ is a $(k-2)$-dimensional non-degenerate $\mathbb{F}_{q}$-subspace of $\mathcal{V}_{1}$ and $\left<x_1,x_2,\cdots,x_k\right>$ is a $k$-dimensional non-degenerate $\mathbb{F}_q$-subspace of $\mathcal{V}_{1}$ with $\text{det }\mathfrak{G}(x_1,x_2,\cdots,x_k) \neq  \text{det }\mathfrak{G}(x_1,x_2,\cdots,x_{k-2}),\vspace{-2mm}$ is given by $$ \vspace{-1mm}q^{\frac{(k-2)(t-k)}{2}}  \left\{(q^{t-k}-1)(q^{t-k}-q^{t-k-1})-q^{t-k-1}(q^{t-k}-1)\right\} {(t-2)/2 \brack (k-2)/2}_{q^2}\vspace{-10mm}$$ \\$$\hspace{-40mm}=q^{\frac{(k-2)(t-k)}{2}} q^{t-k-1}(q^{t-k}-1)(q-2){(t-2)/2 \brack (k-2)/2}_{q^2}.\vspace{-1mm}$$

On combining cases \textbf{(i)-(iii)}, we see that  the total number of  $k$-dimensional non-degenerate $\mathbb{F}_{q}$-subspaces of $\mathbb{F}_{q^t}$ of the type $\langle x_1,x_2,\cdots,x_{k-2},1+x_{k-1},\alpha+x_k\rangle$  with $x_1,x_2,\cdots,x_k\in \mathcal{V}_{1}$ as linearly  independent vectors over $\mathbb{F}_{q},$  is given by $
q^{\frac{k(t-k-2)}{2}} (q^{k+1}-q)(q^{k-2}-1) {(t-2)/2 \brack k/2}_{q^2}+q^{\frac{(k-2)(t-k)}{2}}(q^{t-k}-1)(q^{t-k}-q^{t-k-1}-q){(t-2)/2 \brack (k-2)/2}_{q^2}.$

Next let $x_{k-1},x_k \in \mathcal{V}_{1}\setminus \{0\}$ be linearly dependent over $\mathbb{F}_{q}.$ In this case, the  $k$-dimensional $\mathbb{F}_q$-subspace $\langle x_1,x_2, \cdots,  x_{k-2}, 1+x_{k-1},\alpha+x_k \rangle$ of $\mathbb{F}_{q^t}$ is equal to the $\mathbb{F}_q$-subspace $\langle x_1, x_2,\cdots, x_{k-2},1 +x_{k-1},1+\lambda\alpha \rangle$ for some $\lambda \in \mathbb{F}_{q}\setminus \{0\}.$ Further, for each $\lambda \in \mathbb{F}_{q}\setminus \{0\},$ the $k$-dimensional $\mathbb{F}_q$-subspace $\langle x_1,x_2, \cdots, x_{k-2},1 +x_{k-1},1+\lambda\alpha \rangle$ of $\mathbb{F}_{q^t}$ is non-degenerate if and only if the $(k-2)$-dimensional $\mathbb{F}_q$-subspace $\langle x_1,x_2, \cdots, x_{k-2}\rangle$ of $\mathcal{V}_1$ is non-degenerate.
Now working as in case II of part (a)  and noting that each $\lambda \in \mathbb{F}_q\setminus \{0\}$ gives rise to a distinct $k$-dimensional non-degenerate $\mathbb{F}_q$-subspace $\langle x_1,x_2, \cdots, x_{k-2}, 1+x_{k-1},1+\lambda\alpha\rangle$ of $\mathbb{F}_{q^t},$   the number of  distinct $k$-dimensional non-degenerate  $\mathbb{F}_{q}$-subspaces of $\mathbb{F}_{q^t}$ of the type $\langle x_1,x_2, \cdots, x_{k-2}, 1+x_{k-1},1+\lambda\alpha\rangle$ with $x_j$'s in $\mathcal{V}_1\setminus\{0\} $ and $\lambda \in \mathbb{F}_{q}\setminus \{0\},$ is given by $q^{\frac{(k-2)(t-k)}{2}}(q^{t-k}-1)(q-1){(t-2)/2 \brack (k-2)/2}_{q^2}.$

\noindent\textbf{Case B.} Let $k$ be odd. When  $x_{k-1}=x_k=0,$ by Theorem 5.1.1 of \cite{szy}, we see that the $k$-dimensional $\mathbb{F}_q$-subspace $\langle x_1,x_2, \cdots, x_{k-2},1,\alpha\rangle$ of $\mathbb{F}_{q^t}$ is degenerate. When  $x_{k-1}\neq 0$ and $x_k=0,$ by Theorem 5.1.1 of \cite{szy}, we see that the $k$-dimensional $\mathbb{F}_q$-subspace $\langle x_1, x_2,\cdots, x_{k-2},1 +x_{k-1},\alpha \rangle$ of $\mathbb{F}_{q^t}$ is non-degenerate if and only if the $(k-1)$-dimensional $\mathbb{F}_q$-subspace $\langle x_1,x_2, \cdots, x_{k-1}\rangle$ of $\mathcal{V}_1$ is non-degenerate. Now working as in case II of part (a), we see that the number of  distinct $k$-dimensional non-degenerate $\mathbb{F}_q$-subspaces of $\mathbb{F}_{q^t}$ of the type $\langle x_1,x_2, \cdots, x_{k-2}, 1+x_{k-1},\alpha \rangle$ with $x_j$'s in $\mathcal{V}_1\setminus \{0\},$ is given by  $ q^{\frac{(k-1)(t-k-1)}{2}}(q^{k-1}-1) {(t-2)/2 \brack (k-1)/2}_{q^2}.$
 When $x_{k-1}=0$ and $x_k\neq 0,$ by Theorem 5.1.1 of \cite{szy}, we see that the $k$-dimensional $\mathbb{F}_q$-subspace $\langle x_1,x_2, \cdots, x_{k-2},1,\alpha+x_{k}\rangle$ of $\mathbb{F}_{q^t}$ is degenerate.
Let $x_{k-1}, x_k \in \mathcal{V}_{1}$ be linearly independent over $\mathbb{F}_{q}.$ Here by Theorem 5.1.1 of \cite{szy}, we see that the $k$-dimensional $\mathbb{F}_q$-subspace $\langle x_1,x_2, \cdots, x_{k-2}, 1+x_{k-1},\alpha+x_k \rangle$ of $\mathbb{F}_{q^t}$ is non-degenerate if and only if the $(k-1)$-dimensional $\mathbb{F}_q$-subspace $\langle x_1,x_2, \cdots, x_{k-1} \rangle$ of $\mathcal{V}_1$ is  non-degenerate. Now as $\left(\mathcal{V}_{1},\left(\cdot,\cdot\right)_{0}\restriction_{\mathcal{V}_{1}
\times\mathcal{V}_{1}}\right)$ is  a symplectic space, working similarly as in Proposition \ref{symp}, we see that the number of $(k-1)$-dimensional
 non-degenerate $\mathbb{F}_{q}$-subspaces of $\mathcal{V}_1$ is given by $q^{\frac{(k-1)(t-k-1)}{2}}{(t-2)/2 \brack (k-1)/2}_{q^2}.$ 
 Further, we see  that each  $(k-1)$-dimensional  non-degenerate $\mathbb{F}_q$-subspace  of $\mathcal{V}_1$ gives rise to precisely $(q^{k-1}-1)$ distinct $(k-1)$-dimensional non-degenerate $\mathbb{F}_q$-subspaces of the type $\langle x_1,x_2, \cdots, x_{k-2},1+x_{k-1}\rangle$ with $x_j \in \mathcal{V}_1\setminus \{0\}$ for $1\leq j\leq k-1.$  Now for a given $(k-1)$-dimensional non-degenerate $\mathbb{F}_q$-subspace of the type $\langle x_1,x_2,\cdots,x_{k-2},1+x_{k-1}\rangle$ with $x_j \in \mathcal{V}_1\setminus \{0\}$ for $1\leq j\leq k-1,$  we need to count all choices for  $x_k \in \mathcal{V}_1\setminus \{0\}$ that give rise to distinct $k$-dimensional non-degenerate $\mathbb{F}_q$-subspaces  $\left<x_1,x_2,\cdots,x_{k-2},1+x_{k-1},\alpha+x_k\right>$ of $\mathbb{F}_{q^t}.$ For this,
 as $\langle x_1,x_2, \cdots, x_{k-1}\rangle$ is a non-degenerate $\mathbb{F}_q$-subspace of $\mathcal{V}_1,$ by Proposition 2.9 of \cite{grove}, we write $\mathcal{V}_1=\langle x_1,x_2, \cdots, x_{k-1}\rangle \perp \langle x_1,x_2, \cdots, x_{k-1}\rangle^{\perp_{0}},$ where $\langle x_1,x_2, \cdots, x_{k-1}\rangle^{\perp_{0}}$ is a non-degenerate $\mathbb{F}_q$-subspace of $\mathcal{V}_1$ having dimension $t-k-1.$  In view of this, each  $x_k\in \mathcal{V}_1\setminus \{0\}$ can be uniquely written as $x_k=\sum\limits_{\ell=1}^{k-1}a_{\ell}x_{\ell}+ \omega,$ where $\omega \in \langle x_1,x_2, \cdots, x_{k-1}\rangle^{\perp_{0}}$ and $a_{\ell} \in \mathbb{F}_q$ for $1\leq \ell \leq k-1.$ Then it is easy to see that $\langle x_1,x_2, \cdots, x_{k-2},1+x_{k-1}, \alpha+x_k  \rangle=\langle x_1,x_2, \cdots, x_{k-2},1+x_{k-1}, \alpha+ \sum\limits_{\ell=1}^{k-1}a_{\ell}x_{\ell}+ \omega \rangle= \langle x_1,x_2, \cdots,x_{k-2}, 1+x_{k-1}, \alpha+ a_{k-1}x_{k-1}+ \omega \rangle,$ where $\omega \in \langle x_1,x_2, \cdots, x_{k-1}\rangle^{\perp_{0}}.$  This implies that each $x_k=a_{k-1}x_{k-1}+\omega$ with $\omega \in \langle x_1,x_2, \cdots, x_{k-1}\rangle^{\perp_{0}}\setminus \{0\}$ and $a_{k-1} \in \mathbb{F}_q$ gives rise to a distinct $k$-dimensional non-degenerate $\mathbb{F}_q$-subspace $\langle x_1,x_2,\cdots,x_{k-2},1+x_{k-1},\alpha+x_k \rangle$ of $\mathbb{F}_{q^t}.$ Further, we observe that $x_k=a_{k-1}x_{k-1}+\omega \neq 0.$ This is because, if $x_k=a_{k-1}x_{k-1}+\omega=0,$ then $\omega= a_{k-1}x_{k-1} \in \langle x_1,x_2, \cdots, x_{k-1}\rangle \perp \langle x_1,x_2, \cdots, x_{k-1}\rangle^{\perp_{0}} =\{0\}.$ This is a contradiction, as $x_{k-1},x_k$ are linearly independent vectors in $\mathcal{V}_1.$ From this, we observe that $x_k$ has precisely $q(q^{t-k-1}-1)=(q^{t-k}-q)$ relevant choices. Therefore when $k$ is odd, the number of  distinct $k$-dimensional non-degenerate $\mathbb{F}_q$-subspaces of $\mathbb{F}_{q^t}$ of the type $\langle x_1,x_2, \cdots, x_{k-2}, 1+x_{k-1}, \alpha+x_k \rangle$ with $x_j$'s in $\mathcal{V}_1\setminus \{0\},$ is given by $ q^{\frac{(k-1)(t-k-1)}{2}}(q^{k-1}-1)(q^{t-k}-q){(t-2)/2 \brack (k-1)/2}_{q^2}.$
Next let $x_{k-1},x_k \in \mathcal{V}_{1}\setminus \{0\}$ be linearly dependent over $\mathbb{F}_{q}.$ In this case, the  $k$-dimensional $\mathbb{F}_q$-subspace $\langle x_1,x_2, \cdots,  x_{k-2}, 1+x_{k-1},\alpha+x_k \rangle$ of $\mathbb{F}_{q^t}$ is equal to the $\mathbb{F}_q$-subspace $\langle x_1, x_2,\cdots, x_{k-2},1 +x_{k-1},1+\lambda\alpha \rangle$ for some $\lambda \in \mathbb{F}_{q}\setminus \{0\}.$ Further, by Theorem 5.1.1 of \cite{szy}, we see that $\langle x_1,x_2, \cdots, x_{k-2}, 1+x_{k-1},1+\lambda\alpha\rangle $ is a $k$-dimensional non-degenerate $\mathbb{F}_q$-subspace of $\mathbb{F}_{q^t}$ if and only if $\langle x_1,x_2, \cdots,x_{k-1}\rangle $ is a $(k-1)$-dimensional non-degenerate $\mathbb{F}_q$-subspace of $\mathcal{V}_1$  for each $\lambda \in \mathbb{F}_q \setminus \{0\}.$ Now as $\left(\mathcal{V}_{1},\left(\cdot,\cdot\right)_{0}\restriction_{\mathcal{V}_{1} \times\mathcal{V}_{1}}\right)$ is  a symplectic space, working similarly as in Proposition \ref{symp}, we see that the number of $(k-1)$-dimensional non-degenerate $\mathbb{F}_q$-subspaces of $\mathcal{V}_1$ is $q^{\frac{(k-1)(t-k-1)}{2}}{(t-2)/2 \brack (k-1)/2}_{q^2}.$ Further, 
we observe that each $(k-1)$-dimensional non-degenerate $\mathbb{F}_q$-subspace of $\mathcal{V}_1$ gives rise to precisely $(q^{k-1}-1)(q-1)$ distinct k-dimensional non-degenerate $\mathbb{F}_q$-subspaces of the type $\langle x_1,x_2, \cdots, x_{k-2}, 1+x_{k-1},1+\lambda\alpha\rangle$ with $x_j \in \mathcal{V}_1 \setminus \{0\}$ for $1\leq j \leq k-1$ and $\lambda \in \mathbb{F}_{q}\setminus \{0\}.$
 Therefore when $k$ is odd, the number of  distinct $k$-dimensional non-degenerate $\mathbb{F}_q$-subspaces of $\mathbb{F}_{q^t}$ of the type $\langle x_1,x_2, \cdots, x_{k-2}, 1+x_{k-1}, 1+\lambda\alpha \rangle$ with $x_j$'s in $\mathcal{V}_1\setminus \{0\}$ and $\lambda\in \mathbb{F}_q\setminus \{0\},$ is given by  $(q^{k-1}-1) q^{\frac{(k-1)(t-k-1)}{2}}(q-1) {(t-2)/2 \brack (k-1)/2}_{q^2}.$

Now on combining all the above cases, when $t$ is even, for $1 \leq k \leq t-1,$ we get
   \vspace{-2mm} $$
N_{0,k}=\left\{
  \begin{array}{ll}
   q^{\frac{tk-k^2-2}{2}}\Big\{(q^{k}+q-1){(t-2)/2 \brack k/2}_{q^2} +(q^{t-k+1}-q^{t-k}+1){(t-2)/2 \brack (k-2)/2}_{q^2} \Big\}
   &\text{if } k \text{ is even; } \vspace{3mm} \\
    q^{\frac{tk-k^2+t-1}{2}} {(t-2)/2 \brack (k-1)/2}_{q^2}  & \text{if } k  \text{ is odd.}

   \end{array}
   \right.
  \vspace{-1.2mm}$$
This completes the proof of the proposition.
\end{proof}
\vspace{-3.5mm}\subsection{Determination of the number $N_h$ when $h \in \mathfrak{M}$}\label{M}
In the following proposition, we consider the case $h \in \mathfrak{M}$ and enumerate all pairs $\left(\mathcal{C}_{h},\mathcal{C}_{\mu(h)}\right)$ with $\mathcal{C}_{h}$ as a  $\mathcal{K}_h$-subspace of $\mathcal{J}_h$ and $\mathcal{C}_{\mu(h)}$ as a $\mathcal{K}_{\mu(h)}$-subspace of $\mathcal{J}_{\mu(h)}$ satisfying $\mathcal{C}_h \cap \mathcal{C}_{h}^{(\delta)} =\{0\}$ and $\mathcal{C}_{\mu(h)} \cap \mathcal{C}_{\mu(h)}^{(\delta)} =\{0\}$ for each $\delta\in\{\ast,0,\gamma\}.$
\begin{prop}\label{compleuni}
Let $ h\in \mathfrak{M}$ be fixed. For $\delta \in \{\ast,0,\gamma\},$ the number of pairs $\left(\mathcal{C}_{h},\mathcal{C}_{\mu(h)}\right)$ with $\mathcal{C}_{h}$ as a  $\mathcal{K}_h$-subspace of $\mathcal{J}_h$ and $\mathcal{C}_{\mu(h)}$ as a $\mathcal{K}_{\mu(h)}$-subspace of $\mathcal{J}_{\mu(h)}$ satisfying $\mathcal{C}_h \cap \mathcal{C}_{h}^{(\delta)} =\{0\}$ and $\mathcal{C}_{\mu(h)} \cap \mathcal{C}_{\mu(h)}^{(\delta)} =\{0\},$ is given by
   $\displaystyle N_h=\displaystyle 2+\sum\limits_{k=1}^{t-1}
   q^{kd_h(t-k)}{t \brack k}_{q^{d_h}} .$
\end{prop}
To prove the above proposition, let $h \in \mathfrak{M}$ be fixed from now onwards. Throughout this section, we shall represent elements of  $\mathcal{J}_{h} \oplus \mathcal{J}_{\mu(h)}$ as $a(X)+b(X),$ where $a(X) \in \mathcal{J}_{h}$ and $b(X) \in\mathcal{J}_{\mu(h)}.$ On similar lines, we shall represent elements of  $\mathcal{K}_h\oplus\mathcal{K}_{\mu(h)}$ as $u(X)+v(X),$ where $u(X)\in \mathcal{K}_h$ and $v(X) \in \mathcal{K}_{\mu(h)}.$ We shall view  $\mathcal{J}_{h} \oplus \mathcal{J}_{\mu(h)}$ as a $\mathcal{K}_{h} \oplus \mathcal{K}_{\mu(h)}$-module under the following operations:
\vspace{-2mm}\begin{eqnarray} \label{add}\bigl(a(X)+b(X)\bigr)+ \bigl(c(X)+d(X)\bigr)&=& \bigl(a(X)+c(X)\bigr)+\bigl(b(X)+d(X)\bigr) \hspace{10mm} \text{(addition)} \\ \label{scal}\bigl(u(X)+ v(X)\bigr)\bigl(a(X)+b(X)\bigr)&= &u(X)a(X) + v(X)b(X) \hspace{10mm} \text{(scalar multiplication)}\vspace{-4mm}\end{eqnarray} for each $a(X)+b(X), ~c(X)+d(X) \in \mathcal{J}_{h} \oplus \mathcal{J}_{\mu(h)}$ and $u(X)+v(X) \in \mathcal{K}_{h} \oplus \mathcal{K}_{\mu(h)}.$

Next for each $ \delta \in \{\ast,0,\gamma\},$ we observe that \vspace{-2mm} \begin{equation}\label{bil}\big[a(X)+ b(X), c(X)+ d(X) \big]_{\delta} = \big[a(X),d(X) \big]_{\delta} + \big[ b(X), c(X) \big]_{\delta} \in \mathcal{K}_{h} \oplus \mathcal{K}_{\mu(h)}\vspace{-2mm}\end{equation} for all $a(X)+ b(X),~ c(X)+d(X) \in \mathcal{J}_{h} \oplus \mathcal{J}_{\mu(h)}.$
For each $h\in \mathfrak{M}$ and $\delta\in\{\ast,0,\gamma\},$ let $\big[ \cdot, \cdot \big]_{\delta}\restriction_{\mathcal{J}_{h} \oplus \mathcal{J}_{\mu(h)}\times\mathcal{J}_{h} \oplus \mathcal{J}_{\mu(h)}}$ denote the restriction of the form $\big[ \cdot, \cdot \big]_{\delta}$ to $\mathcal{J}_{h} \oplus \mathcal{J}_{\mu(h)}\times\mathcal{J}_{h} \oplus \mathcal{J}_{\mu(h)}.$ Then we make the following observation.
\begin{lem}\label{bilinear}
For each $h \in \mathfrak{M},$ the following hold.
\vspace{-2mm}\begin{enumerate}\item[(a)] $\mathcal{J}_{h} \oplus \mathcal{J}_{\mu(h)}$ is a free $\mathcal{K}_{h} \oplus \mathcal{K}_{\mu(h)}$-module of rank $t.$\vspace{-2mm}\item[(b)] For each $\delta \in \{\ast,0,\gamma\},$ the  form $\big[\cdot , \cdot \big]_{\delta}\restriction_{\mathcal{J}_{h} \oplus \mathcal{J}_{\mu(h)}\times\mathcal{J}_{h} \oplus \mathcal{J}_{\mu(h)}}$ is reflexive and non-degenerate. \vspace{-2mm}\item[(c)] For each $\delta \in \{\ast,0,\gamma\},$ the  form $\big[\cdot , \cdot \big]_{\delta}\restriction_{\mathcal{J}_{h} \oplus \mathcal{J}_{\mu(h)}\times\mathcal{J}_{h} \oplus \mathcal{J}_{\mu(h)}}$ is Hermitian when $\delta\in\{\ast,0\}$ and is skew-Hermitian when $\delta =\gamma.$
\vspace{-2mm}\end{enumerate}\end{lem}
\vspace{-2mm}\begin{proof}  Proof is trivial.
\vspace{-2mm}\end{proof}
It is easy to see that both $\mathcal{J}_{h}$ and $\mathcal{J}_{\mu(h)}$ can be viewed as $\mathcal{K}_h \oplus \mathcal{K}_{\mu(h)}$-submodules of $\mathcal{J}_{h} \oplus \mathcal{J}_{\mu(h)}$ with respect to operations defined by \eqref{add} and \eqref{scal}.
From this, we observe that if $\mathcal{C}_{h}$ is a $\mathcal{K}_{h}$-subspace of $\mathcal{J}_{h}$ and $\mathcal{C}_{\mu(h)}$ is a $\mathcal{K}_{\mu(h)}$-subspace of $\mathcal{J}_{\mu(h)},$ then $\mathcal{C}_{h}\oplus \mathcal{C}_{\mu(h)}$ is a $\mathcal{K}_{h} \oplus \mathcal{K}_{\mu(h)}$-submodule of  $\mathcal{J}_{h} \oplus \mathcal{J}_{\mu(h)}.$ Further, if $\text{dim}_{\mathcal{K}_h}\mathcal{C}_h= \ell$ and $\text{dim}_{\mathcal{K}_{\mu(h)}}\mathcal{C}_{\mu(h)}=r,$ then working as in Theorem 6 of Huffman \cite{huff}, we see that the size of the smallest generating set of $\mathcal{C}_{h} \oplus \mathcal{C}_{\mu(h)}$ is $\max\{\ell,r\}.$
Further, for each $\delta\in\{\ast,0,\gamma\},$ the orthogonal complement of $\mathcal{C}_{h} \oplus \mathcal{C}_{\mu(h)},$ denoted by  $\bigl(\mathcal{C}_{h} \oplus \mathcal{C}_{\mu(h)}\bigr)^{\perp_{\delta}},$ is defined as $\bigl(\mathcal{C}_{h} \oplus \mathcal{C}_{\mu(h)}\bigr)^{\perp_{\delta}} =\{ a(X)+ b(X) \in \mathcal{J}_{h} \oplus \mathcal{J}_{\mu(h)}: \left[ a(X)+ b(X), u(X)+ v(X)\right]_{\delta}=0 \text{ for all } u(X)+ v(X) \in \mathcal{C}_{h} \oplus \mathcal{C}_{\mu(h)}\}.$ It is easy to see that $\bigl(\mathcal{C}_{h} \oplus \mathcal{C}_{\mu(h)}\bigr)^{\perp_{\delta}}$ is also a $\mathcal{K}_h \oplus \mathcal{K}_{\mu(h)}$-submodule of  $\mathcal{J}_{h}\oplus \mathcal{J}_{\mu(h)}.$
Further, the $\mathcal{K}_h\oplus\mathcal{K}_{\mu(h)}$-submodule $ \mathcal{C}_{h} \oplus \mathcal{C}_{\mu(h)}$  of $\mathcal{J}_{h} \oplus \mathcal{J}_{\mu(h)}$ is said to be non-degenerate if the restricted form  $\left[\cdot,\cdot\right]_{\delta}\restriction_{\mathcal{C}_{h} \oplus \mathcal{C}_{\mu(h)}\times\mathcal{C}_{h} \oplus \mathcal{C}_{\mu(h)}}$ is non-degenerate, which is equivalent to saying that  $\bigl(\mathcal{C}_{h} \oplus \mathcal{C}_{\mu(h)}\bigr) \cap \bigl(\mathcal{C}_{h} \oplus \mathcal{C}_{\mu(h)}\bigr)^{\perp_{\delta}}=\{0\}.$ The $\mathcal{K}_{h}\oplus \mathcal{K}_{\mu(h)}$-submodule $\mathcal{C}_{h} \oplus \mathcal{C}_{\mu(h)}$ of $\mathcal{J}_{h} \oplus \mathcal{J}_{\mu(h)}$ is said to be Hermitian (skew-Hermitian) if the restricted  form $\left[\cdot,\cdot\right]_{\delta}\restriction_{\mathcal{C}_{h} \oplus \mathcal{C}_{\mu(h)}\times\mathcal{C}_{h} \oplus \mathcal{C}_{\mu(h)}}$ is  Hermitian (skew-Hermitian). We will say that $\bigl(\mathcal{C}_{h} \oplus \mathcal{C}_{\mu(h)}, \left[\cdot,\cdot\right]_{\delta}\restriction_{\mathcal{C}_{h} \oplus \mathcal{C}_{\mu(h)}\times\mathcal{C}_{h} \oplus \mathcal{C}_{\mu(h)}}\bigr)$ is a Hermitian (skew-Hermitian) $\mathcal{K}_{h}\oplus \mathcal{K}_{\mu(h)}$-space if $\mathcal{C}_{h} \oplus \mathcal{C}_{\mu(h)}$ is a non-degenerate and Hermitian (skew-Hermitian) $\mathcal{K}_{h}\oplus \mathcal{K}_{\mu(h)}$-submodule of $\mathcal{J}_{h} \oplus \mathcal{J}_{\mu(h)}.$
 In view of Lemma \ref{bilinear}, we see that   $\left(\mathcal{J}_{h} \oplus \mathcal{J}_{\mu(h)}, \left[ \cdot,\cdot\right]_{\delta}\restriction_{\mathcal{J}_{h} \oplus \mathcal{J}_{\mu(h)} \times \mathcal{J}_{h} \oplus \mathcal{J}_{\mu(h)}}\right)$ is a Hermitian $\mathcal{K}_{h}\oplus \mathcal{K}_{\mu(h)}$-space when $\delta\in\{\ast,0\}$ and is a skew-Hermitian $\mathcal{K}_{h}\oplus \mathcal{K}_{\mu(h)}$-space when $\delta=\gamma.$

 In the following lemma, we observe that for each $\delta \in \{\ast,0,\gamma\},$ if $\mathcal{C}_{h}$ is a  $\mathcal{K}_h$-subspace of $\mathcal{J}_h$ and $\mathcal{C}_{\mu(h)}$ is a $\mathcal{K}_{\mu(h)}$-subspace of $\mathcal{J}_{\mu(h)}$ satisfying $\mathcal{C}_h \cap \mathcal{C}_{h}^{(\delta)} =\{0\}$ and $\mathcal{C}_{\mu(h)} \cap \mathcal{C}_{\mu(h)}^{(\delta)} =\{0\},$ then $\mathcal{C}_{h}\oplus \mathcal{C}_{\mu(h)}$ is a non-degenerate $\mathcal{K}_{h}\oplus \mathcal{K}_{\mu(h)}$-submodule of $\mathcal{J}_{h}\oplus \mathcal{J}_{\mu(h)}$ and vice versa. \begin{lem} \label{nondeg-mod} Let $h \in \mathfrak{M}$ be fixed. Let $\mathcal{C}_{h}$ be a $\mathcal{K}_{h}$-subspace of $\mathcal{J}_{h}$ and $\mathcal{C}_{\mu(h)}$ be a $\mathcal{K}_{\mu(h)}$-subspace of $\mathcal{J}_{\mu(h)}.$   Then for each  $\delta \in \{\ast, 0 ,\gamma\},$  the following hold.
\begin{enumerate}
\vspace{-2mm}\item[(a)] $\bigr(\mathcal{C}_{h}\oplus \mathcal{C}_{\mu(h)}\bigl)^{\perp_{\delta}}=
\mathcal{C}_{h}^{(\delta)}\oplus \mathcal{C}_{\mu(h)}^{(\delta)}.$
\vspace{-2mm}\item[(b)] $\mathcal{C}_h \cap \mathcal{C}_{h}^{(\delta)} =\{0\}$ and $\mathcal{C}_{\mu(h)} \cap \mathcal{C}_{\mu(h)}^{(\delta)} =\{0\}$ if and only if $(\mathcal{C}_{h} \oplus \mathcal{C}_{\mu(h)}) \cap \bigl(\mathcal{C}_{h}^{(\delta)} \oplus \mathcal{C}_{\mu(h)}^{(\delta)}\bigr)=\{0\}.$
\vspace{-2mm}\item[(c)] $\mathcal{C}_h \cap \mathcal{C}_{h}^{(\delta)} =\{0\}$ and $\mathcal{C}_{\mu(h)} \cap \mathcal{C}_{\mu(h)}^{(\delta)} =\{0\}$ if and only if $(\mathcal{C}_{h} \oplus \mathcal{C}_{\mu(h)}) \cap \bigl(\mathcal{C}_{h} \oplus \mathcal{C}_{\mu(h)}\bigr)^{\perp_{\delta}} =\{0\}.$
\end{enumerate}
\vspace{-2mm}\end{lem}
\vspace{-4mm}\begin{proof} Proof is trivial.
\end{proof}

Next we observe that any $\mathcal{K}_{h}\oplus \mathcal{K}_{\mu(h)}$-submodule $\mathfrak{N}$ of $\mathcal{J}_{h}\oplus \mathcal{J}_{\mu(h)}$ can be written as $\mathfrak{N}=\mathfrak{N}_{h}\oplus \mathfrak{N}_{\mu(h)},$ where $\mathfrak{N}_{h}=\mathfrak{N}\cap \mathcal{J}_{h}$ is a $\mathcal{K}_{h}$-subspace of $\mathcal{J}_{h}$ and $\mathfrak{N}_{\mu(h)}=\mathfrak{N}\cap \mathcal{J}_{\mu(h)}$ is a $\mathcal{K}_{\mu(h)}$-subspace of $\mathcal{J}_{\mu(h)}.$ In view of this and  Lemma \ref{nondeg-mod}, we see that for each $h\in \mathfrak{M}$ and $\delta\in\{\ast,0,\gamma\},$  the number of pairs $\big(\mathcal{C}_{h},\mathcal{C}_{\mu(h)}\big)$ with $\mathcal{C}_{h}$ as a  $\mathcal{K}_h$-subspace of $\mathcal{J}_h$ and $\mathcal{C}_{\mu(h)}$ as a $\mathcal{K}_{\mu(h)}$-subspace of $\mathcal{J}_{\mu(h)}$ satisfying $\mathcal{C}_h \cap \mathcal{C}_{h}^{(\delta)} =\{0\}$ and $\mathcal{C}_{\mu(h)} \cap \mathcal{C}_{\mu(h)}^{(\delta)} =\{0\}$ is equal  to the number of non-degenerate $\mathcal{K}_{h} \oplus \mathcal{K}_{\mu(h)}$-submodules  of $\mathcal{J}_{h} \oplus \mathcal{J}_{\mu(h)}$ with respect to $\left[\cdot,\cdot\right]_{\delta}\restriction_{\mathcal{J}_{h} \oplus \mathcal{J}_{\mu(h)} \times \mathcal{J}_{h} \oplus \mathcal{J}_{\mu(h)}}.$ Further, we make the following observation.
\begin{rem} \label{gamma} When $\delta=\gamma,$ by Lemma \ref{bilinear}, we see that $\left[\cdot,\cdot\right]_{\gamma} $ is a non-degenerate and skew-Hermitian form on $\mathcal{J}_h \oplus \mathcal{J}_{\mu(h)}.$ In this case, we will first transform the skew-Hermitian form $\left[\cdot,\cdot\right]_{\gamma}$ to a Hermitian form  $\left[\cdot,\cdot\right]_{\gamma^{(H)}}$ as follows: Let $\vartheta(X) \in \mathcal{K}_{h}\setminus \{0\}.$ Then the corresponding element $\chi(X)=\vartheta(X)-\tau_{1,-1}(\vartheta(X)) \in \mathcal{K}_{h}\oplus \mathcal{K}_{\mu(h)}\setminus \{0\}$ satisfies $\tau_{1,-1}\big(\chi(X)\big)=-\chi(X).$
 Now we define the form $\left[\cdot,\cdot\right]_{\gamma^{(H)}}$ on $\mathcal{J}_h\oplus \mathcal{J}_{\mu(h)}\times \mathcal{J}_h\oplus \mathcal{J}_{\mu(h)}$ as $\left[\alpha_1(X)+\alpha_2(X),\beta_1(X)+ \beta_2(X)\right]_{\gamma^{(H)}}= \chi(X)\left[\alpha_1(X)+\alpha_2(X),\beta_1(X)+\beta_2(X)\right]_{\gamma}$ for all $\alpha_1(X)+\alpha_2(X),\beta_1(X)+\beta_2(X)\in \mathcal{J}_{h} \oplus \mathcal{J}_{\mu(h)}.$ It is easy to see that $\left[\cdot,\cdot\right]_{\gamma^{(H)}}$ is a non-degenerate and Hermitian  form on $\mathcal{J}_{h} \oplus \mathcal{J}_{\mu(h)}.$ Further, as $\vartheta(X)$ and $\tau_{1,-1}(\vartheta(X))$ are  non-zero elements of $\mathcal{K}_h$ and $\mathcal{K}_{\mu(h)}$ respectively, we observe that if two elements in $\mathcal{J}_h\oplus \mathcal{J}_{\mu(h)}$ are orthogonal with respect to $ \left[\cdot,\cdot\right]_{\gamma},$ then those two elements are also  orthogonal with respect to $\left[\cdot,\cdot\right]_{\gamma^{(H)}}$ and vice versa. This implies that if $\mathcal{C}_{h}$ is a $\mathcal{K}_{h}$-subspace of $\mathcal{J}_{h}$ and $\mathcal{C}_{\mu(h)}$ is a $\mathcal{K}_{\mu(h)}$-subspace of $\mathcal{J}_{\mu(h)},$ then the dual code $\big(\mathcal{C}_h\oplus \mathcal{C}_{\mu(h)}\big)^{\perp_{\gamma}}$ of the $\mathcal{K}_{h}\oplus \mathcal{K}_{\mu(h)}$-submodule $\mathcal{C}_h\oplus \mathcal{C}_{\mu(h)}$ of $\mathcal{J}_{h}\oplus \mathcal{J}_{\mu(h)}$ with respect to $\left[\cdot,\cdot\right]_{\gamma}\restriction_{\mathcal{J}_{h}\oplus \mathcal{J}_{\mu(h)} \times \mathcal{J}_{h}\oplus \mathcal{J}_{\mu(h)}}$ coincides with the dual code $\bigl(\mathcal{C}_h\oplus \mathcal{C}_{\mu(h)}\bigr)^{\perp_{\gamma^{(H)}}}=\{a(X) +b(X)\in \mathcal{J}_h \oplus \mathcal{J}_{\mu(h)}:\left[a(X)+b(X),c(X)+d(X)\right]_{\gamma^{(H)}}=0 \text{ for all }c(X)+d(X) \in \mathcal{C}_{h}\oplus \mathcal{C}_{\mu(h)}\}$ of $\mathcal{C}_h\oplus \mathcal{C}_{\mu(h)}$ with respect to $\left[\cdot,\cdot\right]_{\gamma^{(H)}}.$
 Therefore the number of non-degenerate  $\mathcal{K}_h \oplus \mathcal{K}_{\mu(h)}$-submodules of  $\mathcal{J}_{h} \oplus \mathcal{J}_{\mu(h)}$ with respect to $\left[\cdot,\cdot\right]_{\gamma}\restriction_{\mathcal{J}_{h}\oplus \mathcal{J}_{\mu(h)} \times \mathcal{J}_{h}\oplus \mathcal{J}_{\mu(h)}}$ is equal to the number of non-degenerate $\mathcal{K}_{h} \oplus \mathcal{K}_{\mu(h)}$-submodules of $\mathcal{J}_{h} \oplus \mathcal{J}_{\mu(h)}$ with respect to $\left[\cdot,\cdot \right]_{\gamma^{(H)}} .$ In view of this, we will consider the Hermitian and non-degenerate  form $\left[\cdot,\cdot\right]_{\gamma^{(H)}}$ on $\mathcal{J}_{h}\oplus \mathcal{J}_{\mu(h)} \times \mathcal{J}_h\oplus \mathcal{J}_{\mu(h)}$ in lieu  of the skew-Hermitian and non-degenerate  form $\left[\cdot,\cdot\right]_{\gamma}\restriction_{\mathcal{J}_{h} \oplus \mathcal{J}_{\mu(h)}\times\mathcal{J}_{h} \oplus \mathcal{J}_{\mu(h)}}$ from now onwards.\end{rem}

In order to enumerate the number $N_h$ for each $h \in \mathfrak{M},$  we will first prove that if $\mathcal{C}_{h}$ is a $\mathcal{K}_{h}$-subspace of $\mathcal{J}_{h}$ and $\mathcal{C}_{\mu(h)}$ is a $\mathcal{K}_{\mu(h)}$-subspace of $\mathcal{J}_{\mu(h)}$ such that $\mathcal{C}_{h}\oplus \mathcal{C}_{\mu(h)}$ is a non-degenerate $\mathcal{K}_{h}\oplus \mathcal{K}_{\mu(h)}$-submodule of $\mathcal{J}_{h} \oplus \mathcal{J}_{\mu(h)}$ with respect to $\left[\cdot,\cdot\right]_{\delta}$ for $\delta \in \{\ast,0,\gamma^{(H)}\},$ then $\dim_{\mathcal{K}_{h}}\mathcal{C}_{h}=
\dim_{\mathcal{K}_{\mu(h)}}\mathcal{C}_{\mu(h)}.$ Note that for each $\delta \in \{\ast,0,\gamma^{(H)}\},$ the  form $\left[\cdot,\cdot\right]_{\delta}$ is non-degenerate and Hermitian on $\mathcal{J}_{h} \oplus \mathcal{J}_{\mu(h)},$ i.e., $\bigl(\mathcal{J}_{h} \oplus \mathcal{J}_{\mu(h)}, \left[\cdot,\cdot\right]_{\delta}\restriction_{\mathcal{J}_{h} \oplus \mathcal{J}_{\mu(h)}\times\mathcal{J}_{h} \oplus \mathcal{J}_{\mu(h)}}\bigr)$ is a Hermitian $\mathcal{K}_h \oplus \mathcal{K}_{\mu(h)}$-space.

\begin{prop}\label{dim}
Let $h \in \mathfrak{M}$ and $\delta\in\{\ast,0,\gamma^{(H)}\}$ be fixed. Let  $\mathcal{C}_{h}\oplus \mathcal{C}_{\mu(h)}$ be a  $\mathcal{K}_{h}\oplus \mathcal{K}_{\mu(h)}$-submodule of $\mathcal{J}_{h}\oplus \mathcal{J}_{\mu(h)},$ where $\mathcal{C}_h$ is a $\mathcal{K}_{h}$-subspace of $\mathcal{J}_{h}$ and $\mathcal{C}_{\mu(h)}$ is a $\mathcal{K}_{\mu(h)}$-subspace of $\mathcal{J}_{\mu(h)}.$ If $\mathcal{C}_{h}\oplus \mathcal{C}_{\mu(h)}$ is non-degenerate, then the $\mathcal{K}_h$-dimension of $\mathcal{C}_h$ is equal to the $\mathcal{K}_{\mu(h)}$-dimension of $\mathcal{C}_{\mu(h)}.$
\end{prop}
In order to prove this proposition, we need to define the following:

Let $\mathcal{A}_{h}$ be a $\mathcal{K}_{h}$-subspace of $\mathcal{J}_{h}$ and $\mathcal{A}_{\mu(h)}$ be a $\mathcal{K}_{\mu(h)}$-subspace of $\mathcal{J}_{\mu(h)}.$  A non-zero element $a(X)+b(X) \in \mathcal{A}_{h} \oplus \mathcal{A}_{\mu(h)}$  is said to be isotropic if  $\left[a(X)+b(X),a(X)+b(X) \right]_{\delta}=0,$ otherwise
$a(X)+b(X)$ is called an anisotropic element of $\mathcal{A}_{h} \oplus \mathcal{A}_{\mu(h)}.$ Note that all the  non-zero elements of $\mathcal{A}_{h}$ and $\mathcal{A}_{\mu(h)}$ are trivially isotropic.  An isotropic element $a(X)+b(X)$ of $\mathcal{A}_{h}\oplus \mathcal{A}_{\mu(h)}$ is said to be non-trivial if both $a(X), b(X)$ are non-zero.  We say that two non-trivial isotropic elements $a(X)+ b(X)$ and $c(X)+d(X)$ of $\mathcal{A}_{h} \oplus \mathcal{A}_{\mu(h)} $ form a hyperbolic pair if   $\left[ a(X)+ b(X), c(X)+d(X) \right]_{\delta} = e_h(X)+e_{\mu(h)}(X),$ where $e_h(X)$ and $e_{\mu(h)}(X)$ are multiplicative identities  of $ \mathcal{K}_{h}$ and $ \mathcal{K}_{\mu(h)}$ respectively. Throughout this section, we will denote the $\mathcal{K}_h \oplus \mathcal{K}_{\mu(h)}$-submodule of $\mathcal{A}_h\oplus \mathcal{A}_{\mu(h)}$ generated by the elements $\alpha_1(X)+\beta_1(X),\alpha_2(X)+\beta_2(X), \cdots, \alpha_m(X)+\beta_m(X) \in \mathcal{A}_{h}\oplus \mathcal{A}_{\mu(h)}$ by $\langle \alpha_1(X)+\beta_1(X),\alpha_2(X)+\beta_2(X), \cdots, \alpha_m(X)+\beta_m(X) \rangle.$ Furthermore, it is easy to see that $\langle \alpha_1(X)+\beta_1(X),\alpha_2(X)+\beta_2(X), \cdots, \alpha_m(X)+\beta_m(X) \rangle^{\perp_{\delta}}=\langle \beta_1(X),\beta_2(X), \cdots,\beta_m(X)\rangle^{\perp_{\delta}} \oplus \langle \alpha_1(X),\alpha_2(X), \cdots,\alpha_m(X)\rangle^{\perp_{\delta}},$ where $\langle \beta_1(X),\beta_2(X), \cdots,\beta_m(X)\rangle^{\perp_{\delta}}=\{a(X)\in \mathcal{A}_h: \left[a(X),\beta_j(X)\right]_{\delta}=0 \text{ for }1\leq j\leq m\}$ is a $\mathcal{K}_h$-subspace of $\mathcal{A}_h$ and $\langle \alpha_1(X),\alpha_2(X), \cdots,\alpha_m(X)\rangle^{\perp_{\delta}}=\{c(X)\in \mathcal{A}_{\mu(h)}: \left[c(X),\alpha_j(X)\right]_{\delta}\\= 0 \text{ for }1\leq j\leq m\}$ is a $\mathcal{K}_{\mu(h)}$-subspace of $\mathcal{A}_{\mu(h)}.$

In the following lemma, we will discuss the existence of a non-trivial isotropic element and a hyperbolic pair in a non-degenerate $\mathcal{K}_h\oplus\mathcal{K}_{\mu(h)}$-submodule of $\mathcal{J}_h\oplus\mathcal{J}_{\mu(h)}.$
\begin{lem}\label{iso} Let  $h \in \mathfrak{M}$ and $\delta\in\{\ast,0,\gamma^{(H)}\}$ be fixed. Let  $\mathcal{A}_{h}\oplus \mathcal{A}_{\mu(h)}$ be a  non-degenerate $\mathcal{K}_{h}\oplus \mathcal{K}_{\mu(h)}$-submodule of $\mathcal{J}_{h}\oplus \mathcal{J}_{\mu(h)},$ where $\mathcal{A}_h$ is a $\mathcal{K}_{h}$-subspace of $\mathcal{J}_{h}$ having dimension $\ell$ and $\mathcal{A}_{\mu(h)}$ is a $\mathcal{K}_{\mu(h)}$-subspace of $\mathcal{J}_{\mu(h)}$ having dimension $r.$ Then we have the following:
\begin{enumerate}
\vspace{-2mm}\item[(a)] There does not exist any non-trivial isotropic element in $\mathcal{A}_{h} \oplus \mathcal{A}_{\mu(h)}$ when either $\ell=1$ or $r=1.$
\vspace{-2mm}\item[(b)] There exists a non-trivial isotropic element in $\mathcal{A}_{h} \oplus \mathcal{A}_{\mu(h)}$ when  both $\ell,r \geq 2.$\vspace{-2mm}\item[(c)] When both $\ell, r \geq 2,$  there exists a hyperbolic pair in  $\mathcal{A}_{h} \oplus \mathcal{A}_{\mu(h)}.$ \vspace{-2mm}\item[(d)] The $\mathcal{K}_{h} \oplus \mathcal{K}_{\mu(h)}$-submodule of $\mathcal{A}_{h} \oplus \mathcal{A}_{\mu(h)}$ generated by a hyperbolic pair is non-degenerate.
\end{enumerate}
\end{lem}
\begin{proof}
\begin{enumerate}
\vspace{-2mm}\item[(a)] To prove this, without any loss of generality, suppose that $r= 1.$ Let $\alpha(X)+ \beta(X)$ be a non-zero element of $\mathcal{A}_{h}\oplus\mathcal{A}_{\mu(h)},$ where $\alpha(X)\in \mathcal{A}_h$ and $\beta(X)\in \mathcal{A}_{\mu(h)}.$ Now $\alpha(X)+\beta(X)$ is an isotropic element if and only if it satisfies $\left[\alpha(X)+ \beta(X), \alpha(X)+\beta(X) \right]_{\delta}=0,$ which holds if and only if $[\alpha(X),\beta(X)]_{\delta}=0$ and $[\beta(X),\alpha(X)]_{\delta}=0,$ which is equivalent to saying that  $\alpha(X) \in \langle \beta(X)\rangle^{\perp_{\delta}}$ and $\beta(X) \in \langle \alpha(X)\rangle^{\perp_{\delta}},$ where $\langle \beta(X)\rangle^{\perp_{\delta}}$ is a $\mathcal{K}_h$-subspace of $\mathcal{A}_h$ and $\langle \alpha(X)\rangle^{\perp_{\delta}}$ is a $\mathcal{K}_{\mu(h)}$-subspace of $\mathcal{A}_{\mu(h)}.$
We first assert that $\text{dim}_{\mathcal{K}_{\mu(h)}}\langle \alpha(X)\rangle^{\perp_{\delta}}=0$ if $\alpha(X) \neq 0$ and $\text{dim}_{\mathcal{K}_{h}}\langle \beta(X)\rangle^{\perp_{\delta}}=\ell-1$ if $\beta(X) \neq 0.$ To prove this assertion, we define a map $\Phi_{\alpha(X)}: \mathcal{A}_{\mu(h)} \rightarrow \mathcal{K}_{\mu(h)}$ as $\Phi_{\alpha(X)}(r(X))=\left[r(X),\alpha(X) \right]_{\delta}$ for all $r(X)\in \mathcal{A}_{\mu(h)}.$ It is easy to see that $\Phi_{\alpha(X)}$ is a non-zero vector space homomorphism and $\langle \alpha(X)\rangle^{\perp_{\delta}}$ is the kernel of the map $\Phi_{\alpha(X)}.$  Now by Sylvester's law of nullity, we have $\text{dim}_{\mathcal{K}_{\mu(h)}}\langle \alpha(X)\rangle^{\perp_{\delta}}= \text{dim}_{\mathcal{K}_{\mu(h)}}\mathcal{A}_{\mu(h)}-1=0,$ which proves the assertion. Working in a similar manner as above, one can show that if $\beta(X)\neq 0,$ then $\text{dim}_{\mathcal{K}_{h}}\langle \beta(X)\rangle^{\perp_{\delta}}=\ell-1\geq 0.$
Further, we observe that $\langle \alpha(X)\rangle ^{\perp_{\delta}}=\mathcal{A}_{\mu(h)}$ if $\alpha(X)=0$ and $\langle \beta(X)\rangle ^{\perp_{\delta}}=\mathcal{A}_{h}$ if $\beta(X)=0.$ From this, it follows that $\alpha(X)+\beta(X)$ is an isotropic element of $\mathcal{A}_h \oplus \mathcal{A}_{\mu(h)}$ if and only if either $\alpha(X)=0$ and $\beta(X) \in \mathcal{A}_{\mu(h)}\setminus \{0\}$ or $\alpha(X) \in \mathcal{A}_{h} \setminus \{0\}$ and $\beta(X)=0.$ This shows that there does not exist any non-trivial isotropic element in $\mathcal{A}_h\oplus \mathcal{A}_{\mu(h)}$ when $r=1.$
\vspace{-2mm}\item[(b)] Suppose that both  $\ell,r \geq 2.$
 Let $u(X)+v(X)$ be an element of $\mathcal{A}_{h}\oplus \mathcal{A}_{\mu(h)},$ where $u(X)\in \mathcal{A}_{h}\setminus \{0\}$ and
  $v(X)\in \mathcal{A}_{\mu(h)}\setminus \{0\}. $  If $u(X)+ v(X)$ is isotropic, then there is nothing to prove.
  Otherwise, we have $[u(X)+v(X),u(X)+v(X)]_{\delta}\neq0.$ By \eqref{bil}, we get $[u(X)+ v(X), u(X)+ v(X)]_{\delta}=\left[ u(X),v(X) \right]_{\delta}+\left[v(X),u(X)\right]_{\delta}.$ Now  by Lemma 3.3(v) of Sharma and Kaur \cite{sh},  Lemma 6(iv)-(v) of Huffman \cite{huff} and using Remark \ref{gamma}, we see that $\left[v(X) ,u(X) \right]_{\delta}=\tau_{1,-1}(\left[ u(X),v(X) \right]_{\delta}).$ This gives $\left[ u(X)+ v(X), u(X)+ v(X) \right]_{\delta}=\left[ u(X),v(X) \right]_{\delta}+\tau_{1,-1}(\left[ u(X),v(X) \right]_{\delta})=d(X)+ \tau_{1,-1}\big(d(X)\big),$ where $d(X)\in\mathcal{K}_{h}\setminus\{0\}.$
Now consider the $\mathcal{K}_{h} \oplus \mathcal{K}_{\mu(h)}$-submodule $\langle u(X)+v(X) \rangle^{\perp_{\delta}}$ of
$\mathcal{A}_{h}\oplus \mathcal{A}_{\mu(h)}.$ It is easy to observe that $\langle u(X)+v(X) \rangle^{\perp_{\delta}}= \langle v(X) \rangle^{\perp_{\delta}}\oplus \langle u(X)\rangle^{\perp_{\delta}}.$ Now working as in part (a), we see that $\text{dim}_{\mathcal{K}_h}\langle v(X)\rangle^{\perp_{\delta}}=\ell-1 \geq 1$ and $\text{dim}_{\mathcal{K}_{\mu(h)}}\langle u(X) \rangle^{\perp_{\delta}} =r-1 \geq 1.$  Using this, we see that there exist non-zero elements $\alpha(X)\in \langle v(X)\rangle^{\perp_{\delta}}$ and $\beta(X) \in \langle u(X)\rangle^{\perp_{\delta}},$ which implies that  $\alpha(X)+\beta(X)\in \langle u(X)+v(X)\rangle^{\perp_{\delta}}.$
Now if the element $\alpha(X)+\beta(X) $ is isotropic, then we are through. Otherwise, using \eqref{bil}, we get  $[ \alpha(X)+ \beta(X), \alpha(X)+ \beta(X) ]_{\delta} =\left[ \alpha(X),\beta(X) \right]_{\delta}+\left[\beta(X),\alpha(X)\right]_{\delta}.$ Now by  Sharma and Kaur \cite[Lemma 3.3(v)]{sh},  Huffman \cite[Lemma 6(iv)]{huff} and using Remark \ref{gamma}, we see that $\left[\beta(X) ,\alpha(X) \right]_{\delta}=\tau_{1,-1}(\left[\alpha(X),\beta(X) \right]_{\delta}).$ This implies that $\left[\alpha(X)+ \beta(X), \alpha(X)+ \beta(X) \right]_{\delta}=\left[\alpha(X),\beta(X) \right]_{\delta}+\tau_{1,-1}(\left[\alpha(X),\beta(X) \right]_{\delta})=\eta(X)+ \tau_{1,-1}\big(\eta(X)\big),$ where $\eta(X)\in\mathcal{K}_{h}\setminus\{0\}.$
Now it is easy to see that $u(X)+ v(X)+\bigl \{-d(X)+ \tau_{1,-1}\bigl(\eta(X)^{-1} \bigr)\bigr \} \bigl(\alpha(X)+\beta(X)\bigr)$ is a non-trivial isotropic element of $\mathcal{A}_h\oplus \mathcal{A}_{\mu(h)}.$
\vspace{-2mm}\item[(c)] As both $\ell, r \geq 2,$ by part (b), we see that $\mathcal{A}_{h}\oplus\mathcal{A}_{\mu(h)}$ has a non-trivial isotropic element, say  $a(X)+b(X).$ Further, we have $\langle a(X)+b(X)\rangle^{\perp_{\delta}}=\langle b(X)\rangle^{\perp_{\delta}} \oplus \langle a(X)\rangle^{\perp_{\delta}}.$ Working as in part (a), we see that
$\text{dim}_{\mathcal{K}_{h}}\langle b(X)\rangle^{\perp_{\delta}}=\ell-1 \geq 1$ and $\text{dim}_{\mathcal{K}_{\mu(h)}}\langle a(X)\rangle^{\perp_{\delta}}=r-1\geq 1.$ Now we choose a non-zero element $c(X) \in \mathcal{A}_h\setminus\langle b(X)\rangle^{\perp_{\delta}}$ and $d(X) \in \mathcal{A}_{\mu(h)}\setminus \langle a(X)\rangle^{\perp_{\delta}}.$ Then $\left[a(X)+b(X),c(X)+d(X)\right]_{\delta}=\ell_1(X)+\ell_2(X),$ where $\ell_1(X)\in \mathcal{K}_h\setminus\{0\}$ and $\ell_2(X)\in \mathcal{K}_{\mu(h)}\setminus\{0\}.$ This gives $\left[a(X)+b(X),\tau_{1,-1}\big(\ell_2(X)^{-1}\big)c(X)\right.\\\left.+ \tau_{1,-1}\big(\ell_1(X)^{-1}\big)d(X)
 \right]_{\delta} = e_h(X)+ e_{\mu(h)}(X).$ Therefore, without any loss of generality, we assume that $c(X)+d(X)\in \mathcal{A}_h\oplus\mathcal{A}_{\mu(h)}$ is such that $\left[a(X)+b(X),c(X)+ d(X)\right]_{\delta}=e_h(X)+e_{\mu(h)}(X).$\\
 If $c(X)+ d(X)$ is an isotropic element, then we are done. Otherwise, using \eqref{bil}, we have $[ c(X)+d(X),c(X)+ d(X)]_{\delta}=\left[c(X),d(X) \right]_{\delta}+\left[d(X),c(X)\right]_{\delta}.$ Now by Sharma and Kaur \cite[Lemma 3.3(v)]{sh},  Huffman \cite[Lemma 6(iv)]{huff} and using Remark \ref{gamma}, we see that $\left[d(X) ,c(X) \right]_{\delta}=\tau_{1,-1}(\left[c(X),d(X) \right]_{\delta}).$  This implies that $\left[c(X)+ d(X), c(X)+ d(X) \right]_{\delta}=\left[c(X),d(X) \right]_{\delta}+\tau_{1,-1}(\left[c(X),d(X) \right]_{\delta})=w(X)+ \tau_{1,-1}\big(w(X)\big),$ where $w(X)\in\mathcal{K}_{h}\setminus\{0\}.$
Now it is easy to see that  $c(X)+ d(X)-w(X)\bigl(a(X)+ b(X)\bigr) =c(X)+d(X)-w(X) a(X)$ is a non-trivial isotropic element in $\mathcal{A}_h\oplus\mathcal{A}_{\mu(h).}$ On the other hand, we observe that $ [a(X)+b(X),c(X)+d(X)-w(X)a(X)]_{\delta} =[a(X),d(X)]_{\delta}+[b(X),c(X)-w(X)a(X)]_{\delta}= e_h(X)+ e_{\mu(h)}(X),$ which implies that $a(X)+b(X)$ and $c(X)+d(X)-w(X)a(X)$ form a hyperbolic pair in $\mathcal{A}_{h} \oplus \mathcal{A}_{\mu(h)}.$
\item[(d)] Proof is trivial.
\end{enumerate}
\vspace{-5mm}\end{proof}

The following lemma is quite useful in writing an orthogonal direct sum decomposition of a non-degenerate $\mathcal{K}_{h}\oplus \mathcal{K}_{\mu(h)}$-submodule of $\mathcal{J}_{h}\oplus\mathcal{J}_{\mu(h)}.$
\begin{lem} \label{nondeg} Let $h\in \mathfrak{M}$ and $\delta \in \{\ast,0,\gamma^{(H)}\}$ be fixed. Let $\mathcal{A}_{h} \oplus \mathcal{A}_{\mu(h)}$ be a non-degenerate $\mathcal{K}_{h}\oplus \mathcal{K}_{\mu(h)}$-submodule of  $\mathcal{J}_{h}\oplus \mathcal{J}_{\mu(h)}$ and $\mathcal{W}_{h} \oplus \mathcal{W}_{\mu(h)}$ be a non-degenerate $\mathcal{K}_{h}\oplus \mathcal{K}_{\mu(h)}$-submodule of $\mathcal{A}_{h}\oplus \mathcal{A}_{\mu(h)},$ where $\mathcal{A}_j$ is a $\mathcal{K}_{j}$-subspace of $\mathcal{J}_{j}$ and $\mathcal{W}_{j}$ is a $\mathcal{K}_{j}$-subspace of $\mathcal{A}_{j}$ for each $j \in \{h,\mu(h)\}.$ Then $(\mathcal{W}_{h} \oplus \mathcal{W}_{\mu(h)})^{\perp_{\delta}}=\{a(X)+b(X) \in \mathcal{A}_{h} \oplus \mathcal{A}_{\mu(h)}: \left[ a(X)+b(X), c(X)+d(X)\right]_{\delta}=0 \text{ for all } c(X)+d(X)\in \mathcal{W}_{h} \oplus \mathcal{W}_{\mu(h)} \}$ is a non-degenerate $\mathcal{K}_{h} \oplus \mathcal{K}_{\mu(h)}$-submodule of $\mathcal{A}_{h} \oplus \mathcal{A}_{\mu(h)}$ and   \vspace{-2mm}\begin{equation*}\mathcal{A}_{h} \oplus \mathcal{A}_{\mu(h)}= (\mathcal{W}_{h} \oplus \mathcal{W}_{\mu(h)}) \perp (\mathcal{W}_{h} \oplus \mathcal{W}_{\mu(h)})^{\perp_{\delta}}.\vspace{-2mm}\end{equation*}\end{lem}
\begin{proof}
It follows from  Lemma 1.1 of Knebusch et al. \cite{kne}.
\end{proof}
 In the following lemma,  we write an orthogonal direct sum decomposition of a non-degenerate $\mathcal{K}_{h} \oplus \mathcal{K}_{\mu(h)}$-submodule
of $\mathcal{J}_{h} \oplus \mathcal{J}_{\mu(h)}.$
\begin{lem}\label{decom} Let $h\in \mathfrak{M}$ and $\delta \in \{\ast,0,\gamma^{(H)}\}$ be fixed. Let $\mathcal{A}_{h} \oplus \mathcal{A}_{\mu(h)} $ be a non-degenerate $\mathcal{K}_{h} \oplus \mathcal{K}_{\mu(h)}$-submodule
of $\mathcal{J}_h \oplus \mathcal{J}_{\mu(h)},$ where $\mathcal{A}_{h}$ is a $\mathcal{K}_{h}$-subspace of $\mathcal{J}_{h}$ and
 $\mathcal{A}_{\mu(h)}$ is a $\mathcal{K}_{\mu(h)}$-subspace of $\mathcal{J}_{\mu(h)}$ with $\text{dim}_{\mathcal{K}_h}\mathcal{A}_h= \ell >0$ and
 $\text{dim}_{\mathcal{K}_{\mu(h)}}\mathcal{A}_{\mu(h)} =r >0.$ Then the following hold:
\begin{enumerate}
\vspace{-2mm}\item[(a)] When $r$ is even and $\ell \geq r,$  we have $\mathcal{A}_{h} \oplus \mathcal{A}_{\mu(h)}=
\langle m_1(X)+n_1(X),p_1(X)+q_1(X)\rangle \perp \langle m_2(X)+n_2(X),p_2(X)+q_2(X)\rangle \perp\cdots \perp \langle
m_{\frac{r}{2}}(X)+n_{\frac{r}{2}}(X), p_{\frac{r}{2}}(X)+ q_{\frac{r}{2}}(X)\rangle\perp \langle \theta_1(X)\rangle \perp \langle \theta_2(X) \rangle \perp \cdots
\perp \langle \theta_{\ell-r}(X)\rangle,$ where $\big(m_{j}(X)+n_{j}(X),p_{j}(X)+q_{j}(X)\big)$ is a hyperbolic pair in
 $\mathcal{A}_{h} \oplus \mathcal{A}_{\mu(h)}$ for $1\leq j\leq \frac{r}{2}$ and $\{\theta_1(X),\theta_2(X),\cdots, \theta_{\ell-r}(X)\}$ is a basis of  the $(\ell-r)$-dimensional
 $\mathcal{K}_h$-subspace $\langle n_1(X),q_1(X),\cdots,\hspace{-0.4mm} n_{\frac{r}{2}}(X),\\ q_{\frac{r}{2}}(X)\rangle^{\perp_{\delta}}$ of $\mathcal{A}_h.$
\vspace{-2mm}\item[(b)] When $r$ is odd and  $\ell \geq r,$ we have $\mathcal{A}_{h} \oplus \mathcal{A}_{\mu(h)}=
\langle m_1(X)+n_1(X),p_1(X)+q_1(X)\rangle \perp \langle m_2(X)+n_2(X),p_2(X)+q_2(X)\rangle \perp\cdots \perp
\langle m_{\frac{r-1}{2}}(X)+n_{\frac{r-1}{2}}(X), p_{\frac{r-1}{2}}(X)+ q_{\frac{r-1}{2}}(X)\rangle\perp \langle
\alpha(X)+\beta(X)\rangle \perp \langle \theta_1(X)\rangle \perp \langle \theta_2(X) \rangle \perp\cdots \perp \langle \theta_{\ell-r}(X)\rangle,$
 where $\big(m_{j}(X)+n_{j}(X),p_{j}(X)+q_{j}(X)\big)$ is a hyperbolic pair in  $\mathcal{A}_{h}\oplus \mathcal{A}_{\mu(h)}$
 for $1\leq j \leq \frac{r-1}{2},$ $\alpha(X)+\beta(X)$ is an anisotropic element of $\mathcal{A}_{h} \oplus \mathcal{A}_{\mu(h)}$ and
 $\{\theta_1(X),\theta_2(X),\cdots, \theta_{\ell-r}(X)\}$ is a basis of the $(\ell-r)$-dimensional $\mathcal{K}_h$-subspace
 $\langle n_1(X),q_1(X),\cdots, \\ n_{\frac{r-1}{2}}(X), q_{\frac{r-1}{2}}(X),\beta(X)\rangle^{\perp_{\delta}}$ of $\mathcal{A}_h.$
\vspace{-2mm}\item[(c)] When $\ell$ is even and $r \geq \ell,$  we have $\mathcal{A}_{h} \oplus \mathcal{A}_{\mu(h)}=
\langle m_1(X)+n_1(X),p_1(X)+q_1(X)\rangle \perp \langle m_2(X)+n_2(X),p_2(X)+q_2(X)\rangle \perp\cdots
\perp \langle m_{\frac{\ell}{2}}(X)+n_{\frac{\ell}{2}}(X), p_{\frac{\ell}{2}}(X)+ q_{\frac{\ell}{2}}(X)\rangle\perp \langle \eta_1(X)\rangle \perp \langle \eta_2(X) \rangle \perp \cdots
\perp \langle \eta_{r-\ell}(X)\rangle,$ where $\big(m_{j}(X)+n_{j}(X),p_{j}(X)+q_{j}(X)\big)$ is a hyperbolic pair in
$\mathcal{A}_{h} \oplus \mathcal{A}_{\mu(h)}$ for $1\leq j\leq \frac{\ell}{2}$ and $\{\eta_1(X),\eta_2(X),\cdots, \eta_{r-\ell}(X)\}$ is a basis of the $(r-\ell)$-dimensional
$\mathcal{K}_{\mu(h)}$-subspace $ \langle m_1(X),p_1(X), \cdots,\\m_{\frac{\ell}{2}}(X),p_{\frac{\ell}{2}}(X)\rangle^{\perp_{\delta}}$ of $\mathcal{A}_{\mu(h)}.$
\vspace{-2mm}\item[(d)] When $\ell$ is odd and $r \geq \ell,$ we have $\mathcal{A}_{h} \oplus \mathcal{A}_{\mu(h)}=
\langle m_1(X)+n_1(X),p_1(X)+q_1(X)\rangle \perp \langle m_2(X)+n_2(X),p_2(X)+q_2(X)\rangle \perp\cdots \perp \langle m_{\frac{\ell-1}{2}}(X)+
n_{\frac{\ell-1}{2}}(X), p_{\frac{\ell-1}{2}}(X) + q_{\frac{\ell-1}{2}}(X)\rangle\perp \langle \alpha(X)+\beta(X)\rangle
\perp \langle \eta_1(X)\rangle \perp \langle \eta_2(X) \rangle \perp \cdots \perp \langle \eta_{r-\ell}(X)\rangle,$ where $\big(m_{j}(X)+n_{j}(X),p_{j}(X)+q_{j}(X)\big)$ is a
hyperbolic pair in $\mathcal{A}_{h}\oplus \mathcal{A}_{\mu(h)}$ for $1\leq j\leq \frac{\ell-1}{2},$
$\alpha(X)+\beta(X)$ is an anisotropic element of $\mathcal{A}_{h} \oplus \mathcal{A}_{\mu(h)}$ and
$\{\eta_1(X),\eta_2(X),\cdots, \eta_{r-\ell}(X)\}$ is a basis of the $(r-\ell)$-dimensional $\mathcal{K}_{\mu(h)}$-subspace $\langle m_1(X),p_1(X),\cdots, \\ m_{\frac{\ell-1}{2}}(X), p_{\frac{\ell-1}{2}}(X) ,\alpha(X)\rangle^{\perp_{\delta}}$ of $\mathcal{A}_{\mu(h)}.$
\end{enumerate}
\end{lem}
\begin{proof}
\begin{enumerate}\vspace{-2mm}\item[(a)] Suppose that $r$ is even and $\ell \geq r.$ Let us write $r=2r_1,$ where $r_1 \geq 1$ is an integer. In order to prove (a), we will apply induction on $r_1.$

First let $r_1=1.$ In this case, by Lemma \ref{iso}(c) and (d), we see that $\mathcal{A}_{h} \oplus \mathcal{A}_{\mu(h)}$ has a hyperbolic pair, say $(m_1(X)+n_1(X),p_1(X)+q_1(X)),$ and the $\mathcal{K}_h\oplus \mathcal{K}_{\mu(h)}$-submodule $\langle m_1(X)+n_1(X),p_1(X)+q_1(X) \rangle$ of $\mathcal{A}_{h} \oplus \mathcal{A}_{\mu(h)}$ is non-degenerate. Now by Lemma \ref{nondeg}, we have $\mathcal{A}_{h} \oplus \mathcal{A}_{\mu(h)}= \langle m_1(X)+n_1(X),p_1(X)+q_1(X)\rangle \perp \langle m_1(X)+n_1(X),p_1(X)+q_1(X)\rangle^{\perp_{\delta}},$ where $\langle m_1(X)+n_1(X),p_1(X)+q_1(X)\rangle^{\perp_{\delta}}$ is a non-degenerate $\mathcal{K}_h\oplus \mathcal{K}_{\mu(h)}$-submodule of $\mathcal{A}_h \oplus \mathcal{A}_{\mu(h)}.$
Further,  it is easy to see that $\langle m_1(X)+n_1(X),p_1(X)+q_1(X)\rangle^{\perp_{\delta}}= \langle n_1(X),q_1(X) \rangle^{\perp_{\delta}} \oplus \langle m_1(X),p_1(X)\rangle^{\perp_{\delta}}.$ We next observe that $\text{dim}_{\mathcal{K}_h}\langle n_1(X),q_1(X) \rangle^{\perp_{\delta}}=\ell-2$ and $\text{dim}_{\mathcal{K}_{\mu(h)}}\langle m_1(X),p_1(X) \rangle^{\perp_{\delta}}=0.$ This implies that $\langle m_1(X)+n_1(X),p_1(X)+q_1(X)\rangle^{\perp_{\delta}}= \langle n_1(X),q_1(X) \rangle^{\perp_{\delta}}.$ Now let $\{\theta_1(X),\theta_2(X),\cdots,\theta_{\ell-2}(X)\}$ be a basis of the $\mathcal{K}_h$-subspace $\langle n_1(X),q_1(X) \rangle^{\perp_{\delta}}$ of $\mathcal{A}_h.$ It is clear that $\left[\theta_{j_1}(X),\theta_{j_2}(X)\right]_{\delta}=0$ for all $1\leq j_1,j_2 \leq \ell-2.$ From this, we obtain $\mathcal{A}_h\oplus \mathcal{A}_{\mu(h)}=\langle m_1(X)+n_1(X),p_1(X)+q_1(X)\rangle \perp \langle \theta_1(X)\rangle \perp \langle \theta_2(X) \rangle \perp \cdots \perp \langle \theta_{\ell-2}(X) \rangle.$ Thus the result holds for $r_1=1.$

Next let $r_1 >1$ and suppose that the result holds for every non-degenerate $\mathcal{K}_h \oplus \mathcal{K}_{\mu(h)}$-submodule of the type $\mathcal{W}_{h} \oplus \mathcal{W}_{\mu(h)},$ where $\mathcal{W}_{h}$ is a $\mathcal{K}_{h}$-subspace of $\mathcal{J}_{h}$ and $\mathcal{W}_{\mu(h)}$ is a $\mathcal{K}_{\mu(h)}$-subspace of $\mathcal{J}_{\mu(h)}$ with $\text{dim}_{\mathcal{K}_{\mu(h)}}\mathcal{W}_{\mu(h)}=
2(r_1-1)$ and $\text{dim}_{\mathcal{K}_{h}}\mathcal{W}_{h} \geq 2(r_1-1).$

We will now prove the result  for a non-degenerate $\mathcal{K}_{h} \oplus \mathcal{K}_{\mu(h)}$-submodule $\mathcal{A}_{h} \oplus \mathcal{A}_{\mu(h)},$ where $\mathcal{A}_{h}$ is a $\mathcal{K}_{h}$-subspace of $\mathcal{J}_{h}$ and $\mathcal{A}_{\mu(h)}$ is a $\mathcal{K}_{\mu(h)}$-subspace of $\mathcal{J}_{\mu(h)}$   with $\text{dim}_{\mathcal{K}_{\mu(h)}}\mathcal{A}_{\mu(h)}=2r_1$ and $\text{dim}_{\mathcal{K}_h}\mathcal{A}_{h}=\ell \geq 2 r_1.$ By Lemma \ref{iso}(c) and (d), we see that $\mathcal{A}_{h} \oplus \mathcal{A}_{\mu(h)}$ has a hyperbolic pair, say $(\widetilde{m}(X)+\widetilde{n}(X),\widetilde{p}(X)+\widetilde{q}(X))$ and the $\mathcal{K}_h\oplus \mathcal{K}_{\mu(h)}$-submodule $\langle \widetilde{m}(X)+\widetilde{n}(X),\widetilde{p}(X)+\widetilde{q}(X) \rangle$ of $\mathcal{A}_{h} \oplus \mathcal{A}_{\mu(h)}$ is non-degenerate. Now by Lemma \ref{nondeg}, we have $\mathcal{A}_{h} \oplus \mathcal{A}_{\mu(h)}= \langle \widetilde{m}(X)+\widetilde{n}(X), \widetilde{p}(X)+\widetilde{q}(X)\rangle \perp \langle \widetilde{m}(X)+\widetilde{n}(X), \widetilde{p}(X)+\widetilde{q}(X)\rangle^{\perp_{\delta}},$ where $\langle \widetilde{m}(X)+\widetilde{n}(X),\widetilde{p}(X)+\widetilde{q}(X)\rangle^{\perp_{\delta}}$ is a non-degenerate $\mathcal{K}_h\oplus \mathcal{K}_{\mu(h)}$-submodule of $\mathcal{A}_h \oplus \mathcal{A}_{\mu(h)}.$ Further, we see that $\langle \widetilde{m}(X)+\widetilde{n}(X), \widetilde{p}(X) + \widetilde{q}(X) \rangle^{\perp_{\delta}}= \langle \widetilde{n}(X), \widetilde{q}(X) \rangle^{\perp_{\delta}} \oplus \langle \widetilde{m}(X),\widetilde{p}(X)\rangle^{\perp_{\delta}}.$ It is easy to observe  that $\text{dim}_{\mathcal{K}_h}\langle \widetilde{n}(X), \widetilde{q}(X) \rangle^{\perp_{\delta}}=\ell-2$ and $\text{dim}_{\mathcal{K}_{\mu(h)}}\langle \widetilde{m}(X),\widetilde{p}(X) \rangle^{\perp_{\delta}} = 2r_1-2.$ Now by applying induction hypothesis,  we have $\langle \widetilde{m}(X)+\widetilde{n}(X), \widetilde{p}(X)+ \widetilde{q}(X) \rangle^{\perp_{\delta}}=\langle \widetilde{m}_1(X)+\widetilde{n}_1(X),\widetilde{p}_1(X)+\widetilde{q}_1(X)\rangle \perp \langle \widetilde{m}_2(X)+ \widetilde{n}_2(X),\widetilde{p}_2(X)+ \widetilde{q}_2(X)\rangle \perp\cdots \perp \langle \widetilde{m}_{r_1-1}(X)+ \widetilde{n}_{r_1-1}(X), \widetilde{p}_{r_1-1}(X)+ \widetilde{q}_{r_1-1}(X)\rangle\perp \langle \theta_1(X)\rangle \perp \langle \theta_2(X) \rangle \perp \cdots \perp \langle \theta_{\ell-2r_1}(X)\rangle,$ where $\big(\widetilde{m}_{j}(X)+\widetilde{n}_{j}(X), \widetilde{p}_{j}(X)+\widetilde{q}_{j}(X)\big)$ is a  hyperbolic pair in $\langle \widetilde{m}(X)+\widetilde{n}(X), \widetilde{p}(X)+ \widetilde{q}(X) \rangle^{\perp_{\delta}}$ for $1\leq j\leq r_1-1$ and $\{\theta_1(X),\theta_2(X),\cdots, \theta_{\ell-2r_1}(X)\}$ is a basis of  the $\mathcal{K}_h$-subspace $\langle \widetilde{n}_1(X),\widetilde{q}_1(X),\cdots, \widetilde{n}_{r_1-1}(X), \widetilde{q}_{r_1-1}(X)\rangle^{\perp_{\delta}}$ of $\langle \widetilde{n}(X), \widetilde{q}(X) \rangle^{\perp_{\delta}}.$ From this, we obtain $\mathcal{A}_{h} \oplus \mathcal{A}_{\mu(h)}= \langle \widetilde{m}(X)+\widetilde{n}(X), \widetilde{p}(X)+\widetilde{q}(X)\rangle \perp \langle \widetilde{m}_1(X)+\widetilde{n}_1(X),\widetilde{p}_1(X)+\widetilde{q}_1(X)\rangle \perp \langle \widetilde{m}_2(X)+ \widetilde{n}_2(X),\widetilde{p}_2(X)+ \widetilde{q}_2(X)\rangle \perp\cdots \perp \langle \widetilde{m}_{r_1-1}(X)+ \widetilde{n}_{r_1-1}(X), \widetilde{p}_{r_1-1}(X)+ \widetilde{q}_{r_1-1}(X)\rangle\perp \langle \theta_1(X)\rangle \perp \langle \theta_2(X) \rangle \perp \cdots \perp \langle \theta_{\ell-2r_1}(X)\rangle,$ where $(\widetilde{m}(X)+\widetilde{n}(X), \widetilde{p}(X)+\widetilde{q}(X)), \big(\widetilde{m}_{j}(X)+\widetilde{n}_{j}(X), \widetilde{p}_{j}(X)+\widetilde{q}_{j}(X)\big)$ are  hyperbolic pairs in $\mathcal{A}_h\oplus \mathcal{A}_{\mu(h)}$ for $1\leq j\leq r_1-1$ and $\{\theta_1(X),\theta_2(X),\cdots, \theta_{\ell-2r_1}(X)\}$ is a basis of  the $\mathcal{K}_h$-subspace $\langle \widetilde{n}(X), \widetilde{q}(X), \widetilde{n}_1(X),\widetilde{q}_1(X),\cdots,  \widetilde{n}_{r_1-1}(X), \widetilde{q}_{r_1-1}(X)\rangle^{\perp_{\delta}}$ of $\mathcal{A}_{h}.$  This proves (a).
\vspace{-2mm}\item[(b)] Suppose that $r$ is odd and  $\ell \geq r.$ Let us write $r=2r_1+1,$ where $r_1\geq 0$ is an integer. Here we will apply induction on $r_1 \geq 0.$

    We first suppose that $r_1=0$ so that $\dim_{\mathcal{K}_{\mu(h)}}\mathcal{A}_{\mu(h)}=1.$
    Here we assert that there exists an anisotropic element in $\mathcal{A}_h \oplus \mathcal{A}_{\mu(h)}.$ To prove this, we choose $\beta(X) \in \mathcal{A}_{\mu(h)} \setminus \{0\}.$ It is easy to see that $\dim_{\mathcal{K}_{h}}\langle \beta(X)\rangle^{\perp_{\delta}}=\ell-1.$ From this, we see that there exists $\alpha(X) \in \mathcal{A}_{h} \setminus \langle \beta(X)\rangle^{\perp_{\delta}}$ so that $\left[\alpha(X),\beta(X)\right]_{\delta}=a(X) \in \mathcal{K}_{h}\setminus \{0\}.$ This gives $\left[\alpha(X)+\beta(X),\alpha(X)+\beta(X)\right]_{\delta}=a(X)+\tau_{1,-1}(a(X)),$ i.e., $\alpha(X)+\beta(X)$ is an anisotropic element of $\mathcal{A}_{h}\oplus\mathcal{A}_{\mu(h)}.$
 Further, it is easy to observe that $\langle \alpha(X)+\beta(X)\rangle$ is a non-degenerate $\mathcal{K}_h\oplus \mathcal{K}_{\mu(h)}$-submodule of $\mathcal{A}_h \oplus \mathcal{A}_{\mu(h)}.$ By Lemma \ref{nondeg}, we get $\mathcal{A}_h \oplus \mathcal{A}_{\mu(h)}=\langle \alpha(X)+\beta(X)\rangle \perp \langle \alpha(X)+\beta(X)\rangle^{\perp_{\delta}},$ where $\langle \alpha(X)+\beta(X)\rangle^{\perp_{\delta}}$ is a non-degenerate $\mathcal{K}_h \oplus \mathcal{K}_{\mu(h)}$-submodule of $\mathcal{A}_h \oplus \mathcal{A}_{\mu(h)}.$ Next we observe that $\langle \alpha(X)+\beta(X)\rangle^{\perp_{\delta}}=\langle \beta(X)\rangle^{\perp_{\delta}}\oplus \langle \alpha(X)\rangle^{\perp_{\delta}}.$ Working as in Lemma \ref{iso}(a),  we see that $\text{dim}_{\mathcal{K}_h}\langle \beta(X)\rangle^{\perp_{\delta}}=\ell-1$ and $\text{dim}_{\mathcal{K}_{\mu(h)}}\langle \alpha(X)\rangle^{\perp_{\delta}}=0.$ This implies that $\langle \alpha(X)+\beta(X)\rangle^{\perp_{\delta}}=\langle \beta(X)\rangle^{\perp_{\delta}}$ is a $\mathcal{K}_h$-subspace of $\mathcal{A}_h$ and suppose that  $\{\theta_1(X), \theta_2(X),\cdots, \theta_{\ell-1}(X)\}$ is its basis.  As $\left[\theta_{j_1}(X),\theta_{j_2}(X)\right]_{\delta}=0$ for all $1 \leq j_1,j_2 \leq \ell-1,$ we have $\langle \beta(X)\rangle^{\perp_{\delta}} =\langle \theta_1(X)\rangle \perp \langle \theta_2(X)\rangle \perp \cdots \perp \langle \theta_{\ell-1}(X)\rangle.$ From this, we obtain $\mathcal{A}_h \oplus \mathcal{A}_{\mu(h)}=\langle \alpha(X)+\beta(X)\rangle \perp\langle \theta_1(X)\rangle \perp \langle \theta_2(X) \rangle \perp \cdots \perp \langle \theta_{\ell-1}(X)\rangle.$ Thus the result holds when $r_1=0.$

When $r_1 \geq 1,$ working in a similar way as in part (a) and  using Lemmas \ref{iso} and \ref{nondeg}, part (b) follows.\vspace{-2mm}\item[(c)] Working in a similar manner as in part (a), part (c) follows.
\vspace{-2mm}\item[(d)] Working in a similar manner as in part (b), part (d) follows.
\end{enumerate}
\vspace{-5mm}\end{proof}
\noindent\textit{Proof of Proposition \ref{dim}.} It follows immediately from Lemma \ref{decom} and using the fact that $(\mathcal{C}_{h} \oplus \mathcal{C}_{\mu(h)}) \cap \big(\mathcal{C}_{h} \oplus \mathcal{C}_{\mu(h)}\big)^{\perp_{\delta}} =\{0\}.$ $\hfill \Box$

\begin{rem}\label{rrr} Let $\mathcal{A}_{h} \oplus \mathcal{A}_{\mu(h)}$ be a non-degenerate $\mathcal{K}_{h}
 \oplus \mathcal{K}_{\mu(h)}$-submodule of $\mathcal{J}_{h} \oplus \mathcal{J}_{\mu(h)},$ where $\mathcal{A}_{h}$ is a $\mathcal{K}_{h}$-subspace of $\mathcal{J}_{h}$ and
 $\mathcal{A}_{\mu(h)}$ is a $\mathcal{K}_{\mu(h)}$-subspace of $\mathcal{J}_{\mu(h)}.$
 Then in view of Proposition \ref{dim}, we must have $\text{dim}_{\mathcal{K}_h}\mathcal{A}_h= \text{dim}_{\mathcal{K}_{\mu(h)}}
  \mathcal{A}_{\mu(h)}=r.$ Further, it is easy to see that $\mathcal{A}_{h} \oplus \mathcal{A}_{\mu(h)}$ is a free
  $\mathcal{K}_{h} \oplus \mathcal{K}_{\mu(h)}$-submodule of $\mathcal{J}_{h} \oplus \mathcal{J}_{\mu(h)}$ having rank $r.$
When $r$ is even, by Lemma \ref{decom}(a), we see that $\mathcal{A}_h\oplus \mathcal{A}_{\mu(h)}$ has an orthogonal direct sum decomposition of the type $\langle m_1(X)+n_1(X),p_1(X)+q_1(X)\rangle \perp \cdots \perp \langle m_{\frac{r}{2}}(X)+n_{\frac{r}{2}}(X),p_{\frac{r}{2}}(X)+q_{\frac{r}{2}}(X)\rangle,$ where  $(m_j(X)+n_j(X),p_j(X)+q_j(X))$ is a hyperbolic pair in $\mathcal{A}_h \oplus \mathcal{A}_{\mu(h)}$ for $1 \leq j \leq \frac{r}{2}.$
When $r$ is odd, by Lemma \ref{decom}(b), we see that $\mathcal{A}_h\oplus \mathcal{A}_{\mu(h)}$ has an orthogonal direct sum decomposition of the type $\langle m_1(X)+n_1(X),p_1(X)+q_1(X)\rangle \perp \cdots \perp \langle m_{\frac{r-1}{2}}(X)+n_{\frac{r-1}{2}}(X),p_{\frac{r-1}{2}}(X)+q_{\frac{r-1}{2}}(X)\rangle \perp \langle \alpha(X)+\beta(X)\rangle,$ where $(m_j(X)+n_j(X),p_j(X)+q_j(X))$ is a hyperbolic pair in $\mathcal{A}_h \oplus \mathcal{A}_{\mu(h)}$ for $1 \leq j \leq \frac{r-1}{2}$ and $\alpha(X)+\beta(X)$ is an anisotropic element of $\mathcal{A}_{h}\oplus \mathcal{A}_{\mu(h)}.$ It is easy to see that $\mathcal{A}_{h}\cup \mathcal{A}_{\mu(h)}\setminus\{0\}$ is the set consisting of all the trivial isotropic elements in $\mathcal{A}_h\oplus \mathcal{A}_{\mu(h)}.$ From this, it follows that the number of trivial isotropic elements in $\mathcal{A}_{h}\oplus \mathcal{A}_{\mu(h)}$
   is given by $2(q^{rd_h}-1).$ When $r=1,$ by Lemma \ref{iso}(a), we see that there does not exist any non-trivial isotropic element in $\mathcal{A}_{h}\oplus \mathcal{A}_{\mu(h)}.$
  \end{rem}
 In the following proposition, we determine the number of non-trivial isotropic elements in $\mathcal{A}_{h} \oplus \mathcal{A}_{\mu(h)}$ when $\text{dim}_{\mathcal{K}_{h}}\mathcal{A}_{h}=
   \text{dim}_{\mathcal{K}_{\mu(h)}}\mathcal{A}_{\mu(h)}=
   r \geq 2.$

\begin{prop}\label{iso3} Let $\mathcal{A}_{h} \oplus \mathcal{A}_{\mu(h)}$ be a non-degenerate $\mathcal{K}_{h} \oplus \mathcal{K}_{\mu(h)}$-submodule
of $\mathcal{J}_{h} \oplus \mathcal{J}_{\mu(h)},$ where $\mathcal{A}_{h}$ is a  $\mathcal{K}_{h}$-subspace of $\mathcal{J}_{h}$ and $\mathcal{A}_{\mu(h)}$ is a $\mathcal{K}_{\mu(h)}$-subspace of $\mathcal{J}_{\mu(h)}$ with $\text{dim}_{\mathcal{K}_{h}}\mathcal{A}_{h}=
\text{dim}_{\mathcal{K}_{\mu(h)}}\mathcal{A}_{\mu(h)}=r \geq 2.$ Then the number of non-trivial isotropic elements in $\mathcal{A}_{h}
\oplus \mathcal{A}_{\mu(h)}$ is given by $i_r=(q^{(r-1)d_h}-1)(q^{rd_h}-1).$
\end{prop}
\noindent In order to prove this proposition, we need to prove the following lemma:
\begin{lem}\label{iso4}
Let $\mathcal{A}_{h} \oplus \mathcal{A}_{\mu(h)}$ be a non-degenerate $\mathcal{K}_{h} \oplus \mathcal{K}_{\mu(h)}$-submodule of
$\mathcal{J}_{h} \oplus \mathcal{J}_{\mu(h)},$ where $\mathcal{A}_{h}$ is a  $\mathcal{K}_{h}$-subspace of $\mathcal{J}_{h}$ and $\mathcal{A}_{\mu(h)}$ is a $\mathcal{K}_{\mu(h)}$-subspace of $\mathcal{J}_{\mu(h)}$ with $\text{dim}_{\mathcal{K}_{h}}\mathcal{A}_{h}=
\text{dim}_{\mathcal{K}_{\mu(h)}}\mathcal{A}_{\mu(h)}=2.$   Then the number of non-trivial isotropic elements in $\mathcal{A}_{h}
 \oplus \mathcal{A}_{\mu(h)}$ is given by $(q^{2d_h}-1)(q^{d_h}-1).$
\end{lem}
\begin{proof}
Since $\text{dim}_{\mathcal{K}_h}\mathcal{A}_h= \text{dim}_{\mathcal{K}_{\mu(h)}} \mathcal{A}_{\mu(h)}= 2,$
 by Lemma \ref{decom}(a), we have  $\mathcal{A}_{h} \oplus \mathcal{A}_{\mu(h)}= \langle \alpha_1(X)+\alpha_2(X),\beta_1(X)\\+\beta_2(X)\rangle,$
 where $\left(\alpha_1(X)+\alpha_2(X),\beta_1(X)+\beta_2(X)\right)$ is a hyperbolic pair in $\mathcal{A}_h \oplus \mathcal{A}_{\mu(h)}.$
 From this, we get $\left[\alpha_1(X), \beta_2(X)\right]_{\delta}= e_h(X),$ $\left[\alpha_2(X),\beta_1(X) \right]_{\delta}=e_{\mu(h)}(X),$
 $\left[ \alpha_1(X),\alpha_2(X)\right]_{\delta} =0$ and $\left[\beta_1(X),\beta_2(X)\right]_{\delta}=0.$
 Now any element $h_1(X)+h_2(X)$ in $\mathcal{A}_{h} \oplus \mathcal{A}_{\mu(h)}$ with $h_1(X)\in \mathcal{A}_h\setminus\{0\}$ and
 $h_2(X)\in\mathcal{A}_{\mu(h)}\setminus\{0\}$ can be written as  $h_1(X)+h_2(X)= \bigl (k_1(X) + k_2(X) \bigr)\bigl(\alpha_1(X)+\alpha_2(X)\bigr)+
 \bigl(\ell_1(X)+\ell_2(X)\bigr)\bigl(\beta_1(X)+ \beta_2(X)\bigr),$ where $k_1(X)+ k_2(X),\ell_1(X)+\ell_2(X) \in \mathcal{K}_{h} \oplus
 \mathcal{K}_{\mu(h)}.$ From this, we obtain $h_1(X)=k_1(X)\alpha_1(X)+ \ell_1(X)\beta_1(X)$ and $h_2(X)=k_2(X)\alpha_2(X)+\ell_2(X)\beta_2(X).$
Next we observe that $h_1(X)=k_1(X)\alpha_1(X)+ \ell_1(X)\beta_1(X) \neq 0$ if and only if $(k_1(X),\ell_1(X))\neq(0,0)$ and
$h_2(X) =k_2(X)\alpha_2(X)+ \ell_2(X)\beta_2(X) \neq 0$ if and only if $(k_2(X),\ell_2(X))\neq(0,0).$\\
Now in order to count all the  non-trivial isotropic elements in $\mathcal{A}_{h} \oplus \mathcal{A}_{\mu(h)},$ we will consider the following three cases separately:
\textbf{I. } $k_1(X)+k_2(X)=0$  \textbf{II. } $\ell_1(X)+\ell_2(X)=0$  and \textbf{III. } both $k_1(X)+k_2(X)$ and $\ell_1(X)+\ell_2(X)$ are non-zero.
\\\textbf{Case I.} Let $k_1(X)+k_2(X)=0,$ that is, $k_1(X)=0$ and $k_2(X)=0.$ Then we have $\ell_1(X)\neq 0$ and $\ell_2(X)\neq0.$
This gives $h_1(X)+h_2(X)=(\ell_1(X)+\ell_2(X))(\beta_1(X)+\beta_2(X)),$ which is clearly a non-trivial isotropic element for all $\ell_1(X)\in
\mathcal{K}_h\setminus\{0\}$ and $\ell_2(X)\in \mathcal{K}_{\mu(h)}\setminus\{0\}.$ So in this case, there are precisely $(q^{d_h}-1)^2$  such non-trivial
 isotropic elements in $\mathcal{A}_h\oplus \mathcal{A}_{\mu(h)}.$
\\\textbf{Case II.} Let $\ell_1(X)+\ell_2(X)=0.$ Then we have $k_1(X)\neq 0$ and $k_2(X)\neq0,$ so that $h_1(X)+h_2(X)= \bigl (k_1(X) + k_2(X)
\bigr)\bigl(\alpha_1(X)+\alpha_2(X)\bigr),$ which is clearly a non-trivial isotropic element for all $k_1(X)\in \mathcal{K}_h\setminus\{0\}$
and $k_2(X)\in \mathcal{K}_{\mu(h)}\setminus\{0\}.$ In this case also, there are precisely $(q^{d_h}-1)^2$ such  non-trivial isotropic elements in
$\mathcal{A}_h\oplus \mathcal{A}_{\mu(h)}.$
\\\textbf{Case III.} Let $k_1(X)+k_2(X)\neq0$ and $\ell_1(X)+\ell_2(X)\neq0.$ Then we see that $h_1(X)+h_2(X)$ is an isotropic element if and
only if $[\bigl(k_1(X)+k_2(X)\bigr)\bigl(\alpha_1(X)+\alpha_2(X)\bigr)+ \bigl(\ell_1(X) +\ell_2(X) \bigr) \bigl(\beta_1(X)+ \beta_2(X)\bigr),
\bigl(k_1(X)+k_2(X)\bigr)\bigl(\alpha_1(X)+\alpha_2(X)\bigr)+\bigl(\ell_1(X) +\ell_2(X) \bigr) \bigl(\beta_1(X)+ \beta_2(X)\bigr)]_{\delta}=0$
if and only if $\bigl(k_1(X) +k_2(X)\bigr)\tau_{1,-1}\bigl(\ell_1(X)+\ell_2(X)\bigr)+ \bigl(\ell_1(X)+ \ell_2(X)\bigr)\tau_{1,-1}\bigl(k_1(X)+k_2(X)\bigr)=0$
 if and only if \vspace{-2mm}\begin{equation}\label{isoeq}k_1(X)\tau_{1,-1}(\ell_2(X))+\ell_1(X)\tau_{1,-1}(k_2(X))=0.\vspace{-2mm}\end{equation}

Now if $k_1(X)=0$ and $k_2(X)\neq0,$ then $\ell_1(X)\neq 0$ and $\ell_2(X)\in \mathcal{K}_{\mu(h)}.$ In this case, the condition \eqref{isoeq} is not satisfied, which implies
that $h_1(X)+h_2(X)$ is an anisotropic element of $\mathcal{A}_{h}
\oplus \mathcal{A}_{\mu(h)}.$ Next  if $k_2(X)=0$ and $ k_1(X)\neq0,$ then $\ell_2(X)\neq 0$ and $\ell_1(X)\in \mathcal{K}_h.$ Here also, the condition \eqref{isoeq} is not satisfied, which implies
that $h_1(X)+h_2(X)$ is an anisotropic element of $\mathcal{A}_{h}
\oplus \mathcal{A}_{\mu(h)}.$ Working in a similar way, we see that the element $h_1(X)+h_2(X)$ is anisotropic when either $\ell_1(X)$ or $\ell_2(X)$ (but not both) is zero. Finally, we assume that  $k_1(X),k_2(X),\ell_1(X),\ell_2(X)$ all are non-zero. Then by equation \eqref{isoeq}, we obtain $k_1(X)= -\ell_1(X) \tau_{1,-1}(k_2(X)\ell_2(X)^{-1}),$ which implies that the number of such non-trivial isotropic elements in $\mathcal{A}_{h}
\oplus \mathcal{A}_{\mu(h)}$
 is given by $(q^{d_h}-1)^3.$

 On combining all the above cases, we see that the total number
 of non-trivial isotropic elements in $\mathcal{A}_{h} \oplus \mathcal{A}_{\mu(h)}$ is given by $2(q^{d_h}-1)^2+(q^{d_h}-1)^3=(q^{d_h}-1)(q^{2d_h}-1).$
\end{proof}
\noindent \textit{Proof of Proposition \ref{iso3}.}  In order to prove this proposition, we will apply induction on $r=\text{dim}_{\mathcal{K}_h}\mathcal{A}_h=
\dim_{\mathcal{K}_{\mu(h)}}\mathcal{A}_{\mu(h)}\geq 2.$ When $r=2,$  by Lemma \ref{iso4}, the result holds.  Now we assume that $r >2$ and suppose that the result holds for all  $\mathcal{K}_{h}\oplus \mathcal{K}_{\mu(h)}$-submodules
 $\mathcal{W}_{h}\oplus \mathcal{W}_{\mu(h)}$ of $\mathcal{J}_{h}\oplus\mathcal{J}_{\mu(h)},$ where $\mathcal{W}_{h}$ is a $\mathcal{K}_{h}$-subspace of $\mathcal{J}_{h}$ and $\mathcal{W}_{\mu(h)}$ is a $\mathcal{K}_{\mu(h)}$-subspace of $\mathcal{J}_{\mu(h)}$ with $2\leq \text{dim}_{\mathcal{K}_{h}}\mathcal{W}_{h}=
 \text{dim}_{\mathcal{K}_{\mu(h)}}\mathcal{W}_{\mu(h)} \leq r-1.$
Now let $\mathcal{A}_h \oplus \mathcal{A}_{\mu(h)}$ be a non-degenerate $\mathcal{K}_{h}\oplus \mathcal{K}_{\mu(h)}$-submodule of $\mathcal{J}_{h}\oplus \mathcal{J}_{\mu(h)},$ where $\mathcal{A}_{h}$ is a $\mathcal{K}_{h}$-subspace of $\mathcal{J}_{h}$ and $\mathcal{A}_{\mu(h)}$ is a $\mathcal{K}_{\mu(h)}$-subspace of $\mathcal{J}_{\mu(h)}$ with $\text{dim}_{\mathcal{K}_{h}}\mathcal{A}_{h}=
 \text{dim}_{\mathcal{K}_{\mu(h)}}\mathcal{A}_{\mu(h)} =r.$

  First of all, we assert that there exists an anisotropic element in $\mathcal{A}_{h}\oplus\mathcal{A}_{\mu(h)}.$ For this, we choose $w_2(X) \in \mathcal{A}_{\mu(h)} \setminus \{0\}.$ Working as in Lemma \ref{iso}(a), we see that $\dim_{\mathcal{K}_{h}}\langle w_2(X)\rangle^{\perp_{\delta}}=r-1.$ Thus there exists $w_1(X) \in \mathcal{A}_{h} \setminus \langle w_2(X)\rangle^{\perp_{\delta}}$ so that $\left[w_1(X),w_2(X)\right]_{\delta}=a(X) \in \mathcal{K}_{h}\setminus \{0\}.$ This gives $\left[w_1(X)+w_2(X),w_1(X)+w_2(X)\right]_{\delta}=a(X)+\tau_{1,-1}(a(X)),$ i.e., $w_1(X)+w_2(X)$ is an anisotropic element of $\mathcal{A}_{h}\oplus\mathcal{A}_{\mu(h)}.$ Now as $\langle w_1(X)+w_2(X)\rangle$ is a non-degenerate $\mathcal{K}_h \oplus \mathcal{K}_{\mu(h)}$-submodule of
 $\mathcal{A}_h \oplus \mathcal{A}_{\mu(h)},$ by Lemma \ref{nondeg}, we have $\mathcal{A}_{h} \oplus \mathcal{A}_{\mu(h)} =\langle w_1(X)+w_2(X)\rangle
\perp \langle w_1(X)+w_2(X)\rangle^{\perp_{\delta}}.$ We observe that $\langle w_1(X)+w_2(X)\rangle^{\perp_{\delta}}= \langle w_2(X)\rangle^{\perp_{\delta}}
\oplus \langle w_1(X)\rangle^{\perp_{\delta}}.$ Working as in Lemma \ref{iso}(a), we get $\text{dim}_{\mathcal{K}_h}\langle w_2(X)\rangle^{\perp_{\delta}}= \text{dim}_{\mathcal{K}_{\mu(h)}}
\langle w_1(X)\rangle^{\perp_{\delta}} = r-1.$ This implies that $\langle w_1(X)+w_2(X)\rangle^{\perp_{\delta}}$ is a non-degenerate
$\mathcal{K}_{h}\oplus \mathcal{K}_{\mu(h)}$-submodule of $\mathcal{A}_{h}\oplus\mathcal{A}_{\mu(h)}$ having rank $r-1.$

Now any element $h_1(X)+h_2(X)\in \mathcal{A}_{h} \oplus \mathcal{A}_{\mu(h)}$ with $h_1(X) \in \mathcal{A}_{h}\setminus \{0\}$ and $h_2(X) \in \mathcal{A}_{\mu(h)}\setminus \{0\}$ can be uniquely written as
$h_1(X)+h_2(X)= \bigl(\lambda_1(X)+\lambda_2(X)\bigr) \bigl(w_1(X)+ w_2(X)\bigr) + \bigl(v_1(X)+v_2(X)\bigr),$ where $\lambda_1(X)+\lambda_2(X) \in
\mathcal{K}_{h} \oplus \mathcal{K}_{\mu(h)}$ and $v_1(X)+v_2(X)\in \langle w_1(X)+w_2(X)\rangle^{\perp_{\delta}}$ with
$v_1(X) \in \langle w_2(X)\rangle^{\perp_{\delta}}$ and $v_2(X) \in \langle w_1(X)\rangle^{\perp_{\delta}}.$   We next observe that $h_1(X)=
\lambda_1(X)w_1(X)+v_1(X)\neq 0$ if and only if $\bigl(\lambda_1(X),v_1(X)\bigr)\neq (0,0),$ while $h_2(X)=\lambda_2(X) w_2(X)+ v_2(X)\neq 0$
if and only if $\bigl(\lambda_2(X),v_2(X)\bigr) \neq(0,0).$  Now in order to count all non-trivial isotropic elements in $\mathcal{A}_{h}\oplus \mathcal{A}_{\mu(h)},$ we shall distinguish the following
two cases: \textbf{I. } $\lambda_1(X)+\lambda_2(X)=0$ and \textbf{II. } $\lambda_1(X)+
\lambda_2(X)\neq0.$
\\ \textbf{Case I.} Let $\lambda_1(X)+\lambda_2(X)=0.$ Then we must have $v_1(X)\neq0$ and $v_2(X)\neq0.$ So $h_1(X)+h_2(X)=v_1(X)+v_2(X)
\in \langle w_1(X)+w_2(X)\rangle^{\perp_{\delta}},$ which is a non-degenerate $\mathcal{K}_{h}\oplus \mathcal{K}_{\mu(h)}$-submodule of $\mathcal{A}_{h}\oplus \mathcal{A}_{\mu(h)}$ having rank $r-1.$ Now by applying induction hypothesis, we see that the number of such non-trivial isotropic elements in $\mathcal{A}_{h}\oplus \mathcal{A}_{\mu(h)}$ is given by
 $i_{r-1}.$
\\\textbf{Case II.} Let $\lambda_1(X)+\lambda_2(X)\neq0.$

If $\lambda_1(X)=0,$ then $\lambda_2(X)\neq 0,$ $v_1(X)\neq 0$ and $v_2(X)\in
\langle w_1(X)\rangle^{\perp_{\delta}}.$ Now if $v_2(X)=0,$ then we see that $h_1(X)+ h_2(X)= \lambda_2(X)w_2(X)+v_1(X)$ and clearly
$[\lambda_2(X)w_2(X)+v_1(X),\lambda_2(X)w_2(X)+v_1(X)]_{\delta} =0,$ which implies that $h_1(X)+h_2(X)$ is a non-trivial isotropic element
 for all $\lambda_2(X)\in
\mathcal{K}_{\mu(h)}\setminus\{0\}$ and $v_1(X)\in \langle w_2(X)\rangle^{\perp_{\delta}} \setminus\{0\}.$ In this case, there are precisely
$q^{(r-1)d_h}-1$ choices for $v_1(X)$ and $q^{d_h}-1$ choices for $\lambda_2(X).$ So there are precisely $(q^{(r-1)d_h}-1)(q^{d_h}-1)$
  such non-trivial isotropic elements in $\mathcal{A}_{h}\oplus \mathcal{A}_{\mu(h)}.$ On the other hand, if $v_2(X)\neq 0,$ then for all $\lambda_2(X)\in \mathcal{K}_{\mu(h)}\setminus\{0\},$ we see that
$h_1(X)+ h_2(X)= v_1(X)+ \lambda_2(X)w_2(X) + v_2(X)$ is an isotropic element if and only if $[v_1(X)+ \lambda_2(X)w_2(X)+v_2(X), v_1(X)+
\lambda_2(X)w_2(X)+v_2(X)]_{\delta}=\left[v_1(X)+v_2(X), v_1(X)+v_2(X)\right]_{\delta}=0$ if and only if $v_1(X)+v_2(X)$ is a non-trivial isotropic
element in $\langle w_1(X)+w_2(X)\rangle^{\perp_{\delta}}.$ Now $\lambda_2(X)$ has $q^{d_h}-1$ choices and by the induction hypothesis, the number of
choices for a non-trivial isotropic element $v_1(X)+v_2(X) \in \langle w_1(X)+w_2(X)\rangle^{\perp_{\delta}} $ is  given by $i_{r-1}.$  Thus in this case, the number of non-trivial
 isotropic elements in $\mathcal{A}_{h} \oplus
\mathcal{A}_{\mu(h)}$ is given by $(q^{d_h}-1)i_{r-1}.$
 
 Next suppose that  $\lambda_2(X)=0$ so that $\lambda_1(X)\neq 0.$ Then $v_2(X)\neq 0$ and $v_1(X)\in \langle w_2(X)\rangle^{\perp_{\delta}}.$
Now if $v_1(X)=0,$ then we see that $h_1(X)+h_2(X)=\lambda_1(X)w_1(X ) +v_2(X),$ which is clearly a non-trivial isotropic element for all $\lambda_1(X)
\in\mathcal{K}_h\setminus\{0\}$ and $v_2(X)\in \langle w_1(X)\rangle^{\perp_{\delta}}\setminus\{0\}.$ So the number of such non-trivial
isotropic elements  in $\mathcal{A}_{h} \oplus
\mathcal{A}_{\mu(h)}$ is given by $(q^{(r-1)d_h}-1)(q^{d_h}-1).$
 If $v_1(X)\neq 0,$ then $h_1(X)+h_2(X)=\lambda_1(X)w_1(X )+v_1(X) +v_2(X)$ is an isotropic element if and only if $[v_1(X)+ \lambda_1(X)w_1(X)+v_2(X), v_1(X)+
 \lambda_1(X)w_1(X)+v_2(X)]_{\delta}=\left[v_1(X)+v_2(X), v_1(X)+v_2(X)\right]_{\delta}=0$ if and only if $v_1(X)+v_2(X)$ is a non-trivial
  isotropic element of $\langle w_1(X)+w_2(X) \rangle^{\perp_{\delta}}.$ Now $\lambda_1(X)$ has $q^{d_h}-1$ choices and by the induction hypothesis,
  there are precisely $i_{r-1}$ distinct non-trivial isotropic elements in $\langle w_1(X)+w_2(X) \rangle^{\perp_{\delta}}.$ So in this case,
  the number of  such non-trivial isotropic elements in $\mathcal{A}_{h}\oplus \mathcal{A}_{\mu(h)}$ is given by $(q^{d_h}-1)i_{r-1}.$

Further, suppose that both $\lambda_1(X), \lambda_2(X)$ are non-zero. Then $v_1(X)\in \langle w_2(X)\rangle^{\perp_{\delta}}$ and $v_2(X)\in \langle w_1(X)
\rangle^{\perp_{\delta}}.$ Now $h_1(X)+h_2(X)$ is an isotropic element if and only if $[\bigl(\lambda_1(X)+\lambda_2(X)\bigr) \bigl(w_1(X)+w_2(X)\bigr)+
\bigl(v_1(X)+v_2(X)\bigr), \bigl(\lambda_1(X)+\lambda_2(X)\bigr) \bigl(w_1(X)+ w_2(X)\bigr)+\bigl(v_1(X) +v_2(X)\bigr)]_{\delta}=0$ if and only if
$\bigl(\lambda_1(X)+\lambda_2(X)\bigr) \tau_{1,-1}\bigl(\lambda_1(X)+\lambda_2(X)\bigr)\bigl(a(X)+\tau_{1,-1}\bigl(a(X)\bigr)\bigr)=
-\left[v_1(X)+v_2(X), v_1(X)+ v_2(X) \right]_{\delta}$ if and only if \vspace{-2mm}\begin{equation}\label{isoeq1} \lambda_1(X)\tau_{1,-1}(\lambda_2(X))a(X)
=-\left[v_1(X),v_2(X)\right]_{\delta}.\vspace{-2mm}\end{equation} We observe that when both $v_1(X),v_2(X)$ are zero, we have $h_1(X)+h_2(X)=
\bigl(\lambda_1(X)+\lambda_2(X)\bigr) \bigl(w_1(X)+ w_2(X)\bigr),$ which is clearly an anisotropic element. If either $v_1(X)$ or $v_2(X)$ (but not both) is non-zero, then $h_1(X)+h_2(X)=\bigl(\lambda_1(X)+\lambda_2(X)\bigr) \bigl(w_1(X)+ w_2(X)\bigr)+v_j(X)$ with $v_j(X) \neq 0$ for some $j \in \{1,2\}.$ Then by \eqref{isoeq1}, we see that
$h_1(X)+h_2(X)$ is an isotropic element if and only if $\lambda_1(X)\tau_{1,-1}(\lambda_2(X))a(X)=0,$ which is not possible, and hence
$h_1(X)+h_2(X)$ is  an anisotropic element.
Now if both $v_1(X),v_2(X)$ are non-zero, then by \eqref{isoeq1}, $h_1(X)+h_2(X)$ is an isotropic element
if and only if  $\lambda_1(X)\tau_{1,-1}(\lambda_2(X))a(X)= -\left[v_1(X),v_2(X)\right]_{\delta} .$ We observe that as $\lambda_1(X),
\lambda_2(X)$ and $a(X)$ all are non-zero, we must have $\left[v_1(X),v_2(X)\right]_{\delta} \neq 0,$ which implies that $v_1(X)+v_2(X)$ is an
 anisotropic element of $\langle w_1(X)+w_2(X)\rangle^{\perp_{\delta}}.$ We also note that $\lambda_2(X)\in \mathcal{K}_{\mu(h)}\setminus\{0\}$
 is determined uniquely by $\lambda_1(X)$ and $\left[v_1(X),v_2(X)\right]_{\delta} .$ We see that $\lambda_1(X)\in \mathcal{K}_h\setminus\{0\}$ has $q^{d_h}-1$ choices and by the induction hypothesis, the number of anisotropic elements in $\langle w_1(X)+w_2(X)\rangle^{\perp_{\delta}}$ is given by $(q^{2d_h(r-1)}-1)
   -2(q^{(r-1)d_h}-1)- i_{r-1}=(q^{2d_h(r-1)}-1) -2(q^{(r-1)d_h}-1)- (q^{(r-2)d_h}-1)(q^{(r-1)d_h}-1)=q^{(r-2)d_h}(q^{d_h}-1)(q^{(r-1)d_h}-1).$
    Thus the number of such non-trivial isotropic elements  in $\mathcal{A}_{h} \oplus
\mathcal{A}_{\mu(h)}$ is given by $q^{(r-2)d_h}(q^{d_h}-1)^2(q^{(r-1)d_h}-1).$

On combining all the above cases, we see that the total
number of non-trivial isotropic elements in $\mathcal{A}_{h} \oplus \mathcal{A}_{\mu(h)}$ is given by
$i_r=i_{r-1}+2(q^{d_h}-1) (q^{(r-1)d_h}-1) + 2(q^{d_h}-1)i_{r-1}+q^{(r-2)d_h} (q^{d_h}-1)^2(q^{(r-1)d_h}-1),$ which equals $(q^{(r-1)d_h}-1) (q^{rd_h}-1)$ after a simple computation.
      $\hfill \Box$\\
\noindent In the following proposition, we determine the number of hyperbolic pairs in a non-degenerate
$\mathcal{K}_{h} \oplus \mathcal{K}_{\mu(h)}$-submodule $\mathcal{A}_{h} \oplus \mathcal{A}_{\mu(h)}$ of
$\mathcal{J}_{h} \oplus \mathcal{J}_{\mu(h)},$ where $\mathcal{A}_{h}$ is a $\mathcal{K}_{h}$-subspace of $\mathcal{J}_{h}$ and $\mathcal{A}_{\mu(h)}$ is a $\mathcal{K}_{\mu(h)}$-subspace of $\mathcal{J}_{\mu(h)}$ with $\text{dim}_{\mathcal{K}_{h}}\mathcal{A}_{h}=
 \text{dim}_{\mathcal{K}_{\mu(h)}}\mathcal{A}_{\mu(h)} =r\geq 2.$
 \begin{lem}\label{counthyp} Let $\mathcal{A}_{h} \oplus \mathcal{A}_{\mu(h)}$ be a non-degenerate $\mathcal{K}_{h} \oplus \mathcal{K}_{\mu(h)}$-submodule
of $\mathcal{J}_{h} \oplus \mathcal{J}_{\mu(h)},$  where $\mathcal{A}_{h}$ is a $\mathcal{K}_{h}$-subspace of $\mathcal{J}_{h}$ and $\mathcal{A}_{\mu(h)}$ is a $\mathcal{K}_{\mu(h)}$-subspace of $\mathcal{J}_{\mu(h)}$ with $\text{dim}_{\mathcal{K}_{h}}\mathcal{A}_{h}=
 \text{dim}_{\mathcal{K}_{\mu(h)}}\mathcal{A}_{\mu(h)} =r \geq 2.$  Then the total number of hyperbolic pairs in $\mathcal{A}_{h} \oplus \mathcal{A}_{\mu(h)}$ is given by $\mathfrak{H}_r=q^{d_h(2r-3)}(q^{(r-1)d_h}-1)(q^{rd_h}-1).$
\end{lem}
\begin{proof}
Since $r \geq 2,$ by Lemma \ref{iso}(a), $\mathcal{A}_{h} \oplus \mathcal{A}_{\mu(h)}$ has a non-trivial isotropic element,
say $\alpha_1(X)+\alpha_2(X).$ By Proposition \ref{iso3}, the number of choices for $\alpha_1(X)+\alpha_2(X)$ is given by
$i_{r}=(q^{(r-1)d_h}-1) (q^{rd_h}-1).$ Further by Lemma \ref{iso}(c) and (d), there exists another non-trivial isotropic element $\beta_1(X)+\beta_2(X) $ in $\mathcal{A}_{h}\oplus\mathcal{A}_{\mu(h)}$ such that $(\alpha_1(X)+\alpha_2(X),\beta_1(X)+\beta_2(X))$ is a hyperbolic pair in $\mathcal{A}_{h} \oplus \mathcal{A}_{\mu(h)}$ and the $\mathcal{K}_h\oplus \mathcal{K}_{\mu(h)}$-submodule $\langle \alpha_1(X)+\alpha_2(X),\beta_1(X)+\beta_2(X)\rangle$ of $\mathcal{A}_{h}\oplus \mathcal{A}_{\mu(h)}$ is non-degenerate. Now by Lemma \ref{nondeg}, we can write $\mathcal{A}_{h} \oplus \mathcal{A}_{\mu(h)} =\langle \alpha_1(X)+\alpha_2(X),\beta_1(X)+\beta_2(X)\rangle \perp \langle \alpha_1(X)+\alpha_2(X), \beta_1(X)+\beta_2(X) \rangle^{\perp_{\delta}},$ where $\langle \alpha_1(X)+\alpha_2(X),\beta_1(X)+\beta_2(X)\rangle^{\perp_{\delta}}$ is a non-degenerate $\mathcal{K}_{h} \oplus \mathcal{K}_{\mu(h)}$-submodule of $\mathcal{A}_{h}\oplus\mathcal{A}_{\mu(h)}.$ Further, we observe that $\langle \alpha_1(X)+\alpha_2(X),\beta_1(X)+\beta_2(X)\rangle^{\perp_{\delta}}= \langle \alpha_2(X),\beta_2(X) \rangle^{\perp_{\delta}} \oplus \langle \alpha_1(X),\beta_1(X)\rangle^{\perp_{\delta}},$ where $\text{dim}_{\mathcal{K}_h}\langle \alpha_2(X),\beta_2(X) \rangle^{\perp_{\delta}}= \text{dim}_{\mathcal{K}_{\mu(h)}}\langle \alpha_1(X),\beta_1(X) \rangle^{\perp_{\delta}}=r-2.$ So any element $h_1(X)+h_2(X) \in \mathcal{A}_h \oplus \mathcal{A}_{\mu(h)}$ with $h_1(X) \in \mathcal{A}_{h}\setminus \{0\}$ and $h_2(X) \in \mathcal{A}_{\mu(h)}\setminus \{0\}$ can be written as $h_1(X)+h_2(X)=\bigl(u_1(X)+ u_2(X)\bigr) \bigl (\alpha_1(X)+ \alpha_2(X)\bigr) +\bigl(v_1(X)+v_2(X)\bigr)\bigl(\beta_1(X)+\beta_2(X)\bigr)+\bigl(\eta_1(X)+\eta_2(X)\bigr),$ where $u_1(X)+u_2(X),$ $v_1(X)+v_2(X) \in \mathcal{K}_{h} \oplus \mathcal{K}_{\mu(h)}$ and $\eta_1(X)+\eta_2(X) \in \langle \alpha_1(X)+\alpha_2(X),\beta_1(X)+\beta_2(X)\rangle^{\perp_{\delta}}.$ Note that $\bigl(h_1(X) + h_2(X) , \alpha_1(X)+\alpha_2(X)\bigr)$ is a hyperbolic pair if and only if $\left[h_1(X)+h_2(X),\alpha_1(X)+\alpha_2(X)\right]_{\delta} =e_h(X)+ e_{\mu(h)}(X)$ and $[h_1(X)+h_2(X),h_1(X)+h_2(X)]_{\delta}=0.$ Now $\left[h_1(X)+h_2(X),\alpha_1(X)+\alpha_2(X)\right]_{\delta}=e_h(X)+ e_{\mu(h)}(X)$ if and only if $v_1(X)=e_h(X)$ and $v_2(X)=e_{\mu(h)}(X),$ which gives
$h_1(X)+h_2(X) =  \bigl(u_1(X)+ u_2(X)\bigr) \bigl (\alpha_1(X)+ \alpha_2(X)\bigr)+ \bigl(\beta_1(X)+\beta_2(X) \bigr)+
\bigl(\eta_1(X)+\eta_2(X)\bigr).$  We also observe that both $h_1(X)=u_1(X)\alpha_1(X)+\beta_1(X)+\eta_1(X)$ and $h_2(X)=u_2(X)\alpha_2(X)+\beta_2(X) +\eta_2(X)$ are non-zero. Further, we see that $\left[h_1(X)+h_2(X),h_1(X)+h_2(X) \right]_{\delta}=0$ if and only if $u_1(X)+u_2(X)+
 \tau_{1,-1} \bigl(u_1(X)+u_2(X)\bigr)+\left[\eta_1(X)+\eta_2(X),\eta_1(X)+\eta_2(X) \right]_{\delta} =0$ if and only if
  $u_1(X)+\tau_{1,-1}\bigl(u_2(X)\bigr)= -\left[\eta_1(X),\eta_2(X)\right]_{\delta} .$ Now we observe that $u_1(X)$ is determined uniquely by
  $u_2(X)$ and $\eta_1(X)+\eta_2(X)\in \langle \alpha_1(X)+\alpha_2(X),\beta_1(X)+\beta_2(X)\rangle^{\perp_{\delta}}.$ Further, it is easy to see that
 each of $\eta_1(X)$ and $\eta_2(X)$ has $q^{(r-2)d_h}$ choices and $u_2(X)$ has $q^{d_h}$ choices. From this, it follows that the element $h_1(X)+h_2(X)$ has $q^{(2r-3)d_h}$ choices for $(\alpha_1(X)+\alpha_2(X),h_1(X)+h_2(X))$ to be a hyperbolic pair in $\mathcal{A}_h \oplus \mathcal{A}_{\mu(h)}.$ Therefore the total number of hyperbolic pairs in $\mathcal{A}_h \oplus \mathcal{A}_{\mu(h)}$ is given by $\mathfrak{H}_{r}=q^{d_h(2r-3)}i_{r},$ which equals $q^{d_h(2r-3)} (q^{(r-1)d_h}-1)(q^{rd_h}-1).$
\vspace{-2mm}\end{proof}
\noindent \textit{Proof of Proposition \ref{compleuni}.}
 In view of Remark \ref{gamma}, we need to consider  $\delta \in\{\ast,0,\gamma^{(H)}\}$ so that $\big(\mathcal{J}_{h}\oplus \mathcal{J}_{\mu(h)},\left[\cdot,\cdot\right]_{\delta}\restriction_{\mathcal{J}_h \oplus \mathcal{J}_{\mu(h)}\times \mathcal{J}_h \oplus \mathcal{J}_{\mu(h)}}\big)$ is a Hermitian $\mathcal{K}_h \oplus \mathcal{K}_{\mu(h)}$-space having rank $t.$

Now to prove this proposition, for each integer $k~(0 \leq k \leq t),$ let $N_{h,k}$ denote the number of pairs $\left(\mathcal{C}_{h},\mathcal{C}_{\mu(h)}\right)$ with $\mathcal{C}_{h}$ as a $k$-dimensional $\mathcal{K}_h$-subspace of $\mathcal{J}_h$ and $\mathcal{C}_{\mu(h)}$ as a $k$-dimensional $\mathcal{K}_{\mu(h)}$-subspace of $\mathcal{J}_{\mu(h)}$ satisfying $\mathcal{C}_h \cap \mathcal{C}_{h}^{(\delta)} =\{0\}$ and $\mathcal{C}_{\mu(h)} \cap \mathcal{C}_{\mu(h)}^{(\delta)} =\{0\}.$ Then we have $N_h=\sum\limits_{k=0}^{t}N_{h,k}.$

To begin with, we note,  by Lemma \ref{nondeg-mod}(c), that for $0\leq k \leq t,$ the number $N_{h,k}$ is equal to the number of  $\mathcal{K}_{h} \oplus \mathcal{K}_{\mu(h)}$-submodules $\mathcal{C}_h\oplus\mathcal{C}_{\mu(h)}$ of $\mathcal{J}_{h}\oplus \mathcal{J}_{\mu(h)}$ satisfying $\text{dim}_{\mathcal{K}_{h}}\mathcal{C}_{h}=
\text{dim}_{\mathcal{K}_{\mu(h)}}\mathcal{C}_{\mu(h)}=k$ and $(\mathcal{C}_h\oplus \mathcal{C}_{\mu(h)}) \cap \bigl(\mathcal{C}_h\oplus\mathcal{C}_{\mu(h)}\bigr)^{\perp_{\delta}} =\{0\},$ that is, $N_{h,k}$ is equal to the number of non-degenerate $\mathcal{K}_{h} \oplus \mathcal{K}_{\mu(h)}$-submodules $\mathcal{C}_h\oplus\mathcal{C}_{\mu(h)}$  of $\mathcal{J}_{h} \oplus \mathcal{J}_{\mu(h)}$ satisfying $\text{dim}_{\mathcal{K}_{h}}\mathcal{C}_{h}=
\text{dim}_{\mathcal{K}_{\mu(h)}}\mathcal{C}_{\mu(h)}=k,$ where $\mathcal{C}_h$ is a $\mathcal{K}_h$-subspace of $\mathcal{J}_h$ and $\mathcal{C}_{\mu(h)}$ is a $\mathcal{K}_{\mu(h)}$-subspace of $\mathcal{J}_{\mu(h)}.$

Here by Lemma \ref{bilinear}, we see that  $N_{h,0}=N_{h,t}=1.$
 Further, for any integer $k$ satisfying $1 \leq k \leq t-1,$ using Remark \ref{rrr}, Proposition \ref{iso3} and Lemma \ref{counthyp} and  working in a similar manner as in Proposition \ref{complet}, we see that $ N_{h,k}=\displaystyle
   q^{kd_h(t-k)}{t \brack k}_{q^{d_h}}$ for $1 \leq k \leq t-1.$ From this, the desired result follows.
 $\hfill\Box$\\
\noindent \textit{Proof of Theorem \ref{complecount}.}
It follows immediately from Propositions  \ref{completheo}-\ref{compleuni}. $\hfill \Box$

\end{document}